\documentclass[a4paper,12pt]{article}
\usepackage{amsmath}
\usepackage{amsthm}
\usepackage{amssymb}
\usepackage{amsfonts}
\usepackage{esint}
\usepackage{bbm, dsfont}
\usepackage[english]{babel}
\usepackage{color}
\newcommand{\stef}[1]{{#1}}

\usepackage{fullpage}
\usepackage{enumerate}

\newtheorem{lemma}{Lemma}
\newtheorem{proposition}{Proposition}
\newtheorem{corollary}{Corollary}
\newtheorem{theorem}{Theorem}
\newtheorem{remark}{Remark}
\newtheorem*{remark*}{Remark}

\theoremstyle{definition}
\newtheorem{definition}{Definition}

\DeclareMathOperator*{\osc}{osc}
\DeclareMathOperator{\id}{Id}
\DeclareMathOperator{\ones}{\mathbbm{1}}

\newcommand{\R}{\mathbb{R}}

\newcommand{\Z}{\mathbb{Z}}

\newcommand{\N}{\mathbb{N}}

\newcommand{\e}{\varepsilon}

\mathchardef\emptyset="001F

\newcommand{\expec}[1]{\left\langle #1 \right\rangle}

\newcommand{\step}[1]{\noindent \textit{Step} #1.}

\newcommand{\Fcal}{\mathcal{F}}
\newcommand{\Mfrak}{\mathfrak{M}}
\newcommand{\C}{\mathbb{C}}

\begin{document}

\title{Moment bounds on the corrector of stochastic homogenization
    of non-symmetric elliptic finite difference equations}
  
\author{Jonathan Ben-Artzi%
  \thanks{\texttt{j.ben-artzi@imperial.ac.uk},
    Department of Mathematics, South Kensington Campus,
    Imperial College London,
    London SW7 2AZ,
    United Kingdom}
  \and 
  Daniel Marahrens%
  \thanks{\texttt{daniel.marahrens@mis.mpg.de}, Max-Planck-Institut f\"ur Mathematik
    in den Naturwissenschaften,
    Inselstra\ss e 22,
    04103 Leipzig,
    Germany}
  \and
  Stefan Neukamm%
  \thanks{\texttt{stefan.neukamm@tu-dresden.de},
    Technische Universit\"at Dresden, Fachrichtung Mathematik,
    01062 Dresden, Germany}
}
\maketitle
 
\paragraph{Abstract.}
We consider the corrector equation from the stochastic homogenization of uniformly elliptic finite-difference equations with random, possibly non-symmetric coefficients.
Under the assumption that the coefficients are stationary and ergodic  in the quantitative form of a Logarithmic Sobolev inequality (LSI), we obtain optimal bounds on the
corrector and its gradient in dimensions $d \geq 2$.
Similar estimates have recently been obtained in the special case of diagonal coefficients making extensive use of the maximum principle and scalar techniques.
Our new method only invokes arguments that are also available for
elliptic systems and does not use the maximum principle. In
particular, our proof relies on the LSI to quantify ergodicity and on
regularity estimates on the derivative of the discrete Green's
function in weighted spaces.
In the critical case $d=2$ our argument for the estimate on the
gradient of the elliptic Green's function uses a Calder\'{o}n-Zygmund
estimate in discrete weighted spaces, which we state and prove.
As applications, we provide a quantitative two-scale expansion and a quantitative approximation of the homogenized  coefficients.
\medskip

\tableofcontents

\section{Introduction}

We study the \textit{modified corrector equation}
\begin{equation}\label{eq:1}
  \frac{1}{T}\phi_T+\nabla^*(a\nabla\phi_T)=-\nabla^*(a\xi)\qquad\text{in }\Z^d,\ d\geq2,
\end{equation}
which is a discrete elliptic finite-difference equation for  the real valued function $\phi_T$, called the \textit{modified
  corrector}. As we explain below, it arises in  stochastic
homogenization. The symbols $\nabla$ and $\nabla^*$ denote the discrete
(finite-difference) gradient and the  negative divergence, see
Section~\ref{S:FW} below for the precise definition. In the modified corrector equation
$T$ denotes a positive ``cut-off'' parameter (which we think of to be very large),
and $\xi\in\R^d$ is a vector, fixed throughout this paper. We consider \eqref{eq:1} with a \textit{random, uniformly elliptic} field of coefficients
$a:\Z^d\to\R^{d\times d}$. To be precise, for a fixed constant of
ellipticity $\lambda>0$  we denote
by $\Omega_0$ those matrices $a_0\in\R^{d\times d}$ that are uniformly
elliptic in the sense that
\begin{equation}\label{ass:ell}
  \forall v\in\R^d\,:\qquad v\cdot a_0v\geq \lambda|v|^2\ \ \text{ and
  }\ \ |a_0v|\leq|v|,
\end{equation}
and define the set of admissible coefficient fields
\begin{equation*}
  \Omega:=\Omega_0^{\Z^d}=\{\,a:\Z^d\to\Omega_0\,\}.
\end{equation*}
In this paper we derive optimal bounds for finite
moments of the modified corrector and its gradient, under the assumption
that the coefficients are distributed according to a \textit{stationary and
ergodic} law on $\Omega$, where ergodicity holds in the \textit{quantitative} form of
a \textit{Logarithmic Sobolev Inequality} (LSI), see Definition~\ref{def:LSI}
below. \textcolor{blue}{The bounds in the symmetric case and in the non-symmetric discrete case are new. Below we shall discuss in more detail the novelties of this paper and their relationship with existing results.} Our main results are presented in Theorems~\ref{T1} and
\ref{T2} below. For easy reference, let us state them already here, somewhat
  informally. Throughout the paper, we write $\expec{\cdot}$ for the expected value
  associated to the law on $\Omega$.
\smallskip

{\bf The first result} concerns a bound on all moments of the gradient
of the corrector. Under the assumptions of stationarity and LSI, we
have for all $1\leq p<\infty$ and $T\geq 2$ that 
\[
    \langle |\nabla\phi_T(0) + \xi|^{2p} \rangle \le C |\xi|^{2p},
\]
where the constant $C$ is independent of $T$. (Note that here and
throughout the paper the 
constant ``2'' in ``$T\geq 2$'' has no special meaning. In fact, since we are interested in the
behavior $T\uparrow\infty$, we could replace ``$2$'' with any number greater than $1$).
\smallskip

{\bf The second result} is a bound on the corrector itself. Under the same assumptions (even under a slightly weaker assumption than LSI, see Theorem~\ref{T2} below), we have that
\[
    \langle |\phi_T(0)|^{2p} \rangle \le C \expec{|\nabla \phi_T(0)+\xi|^{2p}}\times
    \begin{cases}
      (\log T)^p&\text{for }d=2,\\
      1&\text{for }d>2.
    \end{cases}
\]
\medskip

These estimates are optimal, even in dimension $d=2$ where we
  recover the optimal logarithmic rate of divergence of the moment of
  $\phi_T$. While the first result is relatively easy to \textcolor{blue}{prove}, the argument for
the second result is substantially harder and the main purpose of our
paper. Let us emphasize that the coefficients in \eqref{eq:1} are not
  assumed to be \textcolor{blue}{diagonal or even symmetric}. Thus, equation~\eqref{eq:1} in general \textit{does not enjoy a maximum
principle}; this constitutes a major difference to previous works
where the maximum principle played a major role and exclusively the case of diagonal
coefficients was studied, see e.g.\
\cite{GO1,GO2,GNO1}.  In fact, the method presented in this paper only 
relies on arguments that are also available in the case of elliptic
systems. The extension of our findings to discrete systems, in
particular a discrete version of linear elasticity, is work in
progress. Recently, Bella and Otto
considered in   \cite{Bella-Otto-14} \textit{systems} of
elliptic equations (on $\R^d$) with periodic (but still
random) coefficients. As a main result, they obtain moment bounds on
the \textit{gradient} of the corrector with help of an argument that avoids the maximum
principle and even the use of Green's functions. Still, the derivation
of moment bounds on the \textit{corrector itself} -- which is the
main purpose of our paper -- remains open.
\medskip

{\bf Applications and relation to stochastic homogenization.} The modified corrector equation
\eqref{eq:1} appears in stochastic 
homogenization: For $\e>0$ and $a\in\Omega$ distributed according to
$\expec{\cdot}$ 
\color{blue}
we consider the equation
\begin{equation}\label{eq:13}
u^\e-\nabla^*_\e(a(\tfrac{\cdot}{\e})\nabla_\e u^\e)=f\qquad\text{in }\e\Z^d,
\end{equation}
where $\nabla^\e$ and $\nabla^{\e*}$ denote the discrete gradient and divergence (see Section~\ref{sec:qtwoscale}). As
shown in \cite{Papanicolaou-Varadhan-79, Kozlov-79, Kozlov-87, Kunnemann-83}, in the homogenization limit $\e\downarrow 0$
the solution $u^\e(a;\cdot)$ converges for almost
every $a\in\Omega$ to the unique solution $u^0\in H^1(\R^d)$ of the homogenized equation
\color{black}
\begin{equation*}
  u_{\hom}-\operatorname{div}(a_{\hom}\nabla u_{\hom})=f\qquad\text{in }\R^d.
\end{equation*}
Here $a_{\hom}\in\Omega_0$ is deterministic and determined by the
formula
\begin{equation}\label{eq:hom-formula}
  e_i\cdot
  a_{\hom}e_j=\lim\limits_{T\uparrow\infty}\expec{(e_i+\nabla\phi_{T,i}(0))\cdot
    a(0)(e_j+\nabla\phi_{T,j}(0))},
\end{equation}
where $\phi_{T,j}$ is the solution to \eqref{eq:1} with $\xi=e_j$.
Let us comment on the appearance of the limit as $T\uparrow\infty$ in this formula.
Formally, and in analogy to periodic homogenization, we expect that
\[
  e_i\cdot
  a_{\hom}e_j=
\expec{(e_i+\nabla\phi_{i}(0))\cdot a(0)(e_j+\nabla\phi_{j}(0))},
\]
where $\phi_i$ is a solution to the \textit{corrector equation}
\begin{equation}\label{eq:corr}
  \nabla^*(a(\nabla\phi_i+e_i))=0\qquad\text{in }\Z^d,
\end{equation}
that is \textit{stationary} in the sense of 
\begin{equation}\label{LMC:1a}
  \phi_i(a;x+z)=\phi_i(a(\cdot+z);x)\qquad\text{$\expec{\cdot}$-almost
    every
  }a\in\Omega\text{ and all }x,z\in\Z^d.
\end{equation}
In the case of deterministic, periodic homogenization, it suffices to
solve \eqref{eq:corr} on the reference torus of periodicity and
existence essentially follows from Poincar\'e's inequality on the torus. In the stochastic case, the corrector equation~\eqref{eq:corr} has to be solved  on the
infinite space $\Z^d$ subject to the stationarity
condition~\eqref{LMC:1a}.
\stef{Since this is not possible in general, the corrector equation
  \eqref{eq:corr} is typically regularized by
adding the zeroth-order term $\frac{1}{T}\phi_i$ with parameter $T\gg
1$.} In fact this was already done in the pioneering work of
Papanicolaou and Varadhan \cite{Papanicolaou-Varadhan-79} and leads to the modified corrector equation
\eqref{eq:1}, which in contrast to \eqref{eq:corr}, admits for all
$a\in\Omega$ a unique bounded solution
$\phi_T(a;\cdot)\in\ell^\infty(\Z^d)$ that automatically is
stationary, see Lemma~\ref{LMC} below. While simple energy bounds, cf.~\eqref{T1.3}, make it relatively easy to pass to the regularization-limit
$T\uparrow\infty$ on the level of $\nabla\phi_T$ (and thus in the
homogenization formula \eqref{eq:hom-formula}), it is difficult, and
in general even impossible, to do the same on the level of
$\phi_T$ itself. For similar reasons (and in contrast to the periodic case), it is
difficult to \textit{quantify} errors in stochastic homogenization, such as the
homogenization error $u_\e-u_{\hom}$. 
{In Section~\ref{sec:qtwoscale} we provide an optimal $H^1$-estimate for the two-scale expansion
  \begin{equation}\label{eq:11}
    u^\e\approx  u_{\hom}+\e\sum_{j=1}^d\phi_j(\tfrac{\cdot}{\e})\partial_ju_{\hom}(\cdot).
  \end{equation}
  in dimensions $d\geq 3$ and obtain an estimate for the homogenization error as corollary.
}
\medskip

{\bf Previous quantitative results  and novelty of the paper.}
For periodic homogenization the quantitative behavior of
\eqref{eq:13} \stef{and the expansion \eqref{eq:11}} is
reasonably well understood (e.g.\ see
\cite{Avellaneda-Lin-87, Allaire-Amar-99, Gerard-Varet-12}). 
\stef{In the stochastic case, due to the lack of compactness,
  the quantitative understanding of \eqref{eq:13} is less developed
  and in most cases only suboptimal estimates are obtained, see
  \cite{Yurinskii-76, Naddaf-Spencer-98, Conlon-Naddaf-00,
    Conlon-Spencer-13, Caputo-Ioffe-03, Bourgeat-04, Armstrong-Smart-14}.}   \stef{In particular, the first
  quantitative result is due to Yurinskii \cite{Yurinskii-76} who
  proved an algebraic rate of convergence (with an suboptimal exponent) for the homogenization error $u_\e-u_{\hom}$ in dimensions
  $d>2$ for algebraically mixing coefficients. For refinements and extensions to dimensions
  $d\geq 2$ we refer to the inspiring work by Naddaf and Spencer
  \cite{Naddaf-Spencer-98}, and the recent works by Conlon and Naddaf
  \cite{Conlon-Naddaf-00} and Conlon and Spencer
  \cite{Conlon-Spencer-13}. Most recently, Armstrong and Smart
  \cite{Armstrong-Smart-14} obtained the first result on the
  homogenization error for the stochastic homogenization of convex
  minimization problems. Their
  approach, which builds up on ideas of Avellaneda and Lin
  \cite{Avellaneda-Lin-87}, substantially differs from what has been done before in stochastic
  homogenization of divergence form equations. It in particular applies to the continuum version of
 \eqref{eq:13} with symmetric coefficients, and potentially extends to symmetric systems (at least under sufficiently strong ellipticity
  assumptions). For results on non-divergence form  elliptic
  equations see \cite{Cafarelli-Souganidis-10, Armstrong-Smart-13}.}
\medskip

While qualitative stochastic homogenization only requires $\expec{\cdot}$ to be stationary and
ergodic, the derivation of error estimates requires a quantification
of ergodicity. \stef{\textcolor{blue}{Pursuing} optimal error bounds}, in a series of papers
\cite{GO1,GO2,GO3,GNO1, GNO3,MO1,LNO1, Mourrat-Otto-14} (initiated by Gloria and Otto) a quantitative theory for \eqref{eq:13} is developed
based  on \emph{Spectral Gap} (SG) and LSI as tools to quantify ergodicity. In contrast to earlier results, the estimates in the papers
mentioned above are optimal: E.g.\ \cite{GNO1} contains a complete and optimal
analysis of the approximation of $a_{\hom}$ via periodic representative volume elements and \cite{GNO3}
establishes optimal estimates for the homogenization error and the
expansion in \eqref{eq:11}. A fundamental step in the derivation of
these results are \stef{optimal} moment bounds for the corrector,  see
\cite{GO1,GO2,GNO1}. The extension to the continuum case has been
discussed in recent papers: In \cite{GO3} moment bounds on the corrector
and its gradient have been obtained for scalar equations with elliptic
coefficients. 
\medskip

In the present contribution we continue the theme of quantitative
stochastic homogenization  and present a new approach that relies on
methods, that -- we believe -- extend with only few modifications to
the case of systems satisfying sufficiently strong ellipticity assumptions. In the works
discussed above, arguments restricted to scalar equations are used at
central places. Most significantly, \textit{Green's function
estimates} are required and derived via De Giorgi-Nash-Moser regularity
theory (\stef{e.g.\
see \cite[Theorem~3]{GNO1}}). This method is based on the \textit{maximum principle}, which
holds for diagonal coefficients, but not for general symmetric or possibly non-symmetric
coefficients as considered here. In fact, in our case the Green's
function is not in general positive everywhere. We derive the
required estimates on the gradient of the Green's function from the
corresponding estimate on the constant coefficient Green's function by a
perturbation argument that invokes a Helmholtz projection; this is inspired by \cite{Conlon-Spencer-11}. Secondly, previous works rely on a
gain of stochastic integrability obtained by a nonlinear Caccioppoli
inequality (see Lemma~2.7 in \cite{GO1}). In the present contribution
we appeal to an alternative argument that invokes the LSI instead. While SG, which is weaker than LSI (see~\cite{Guionnet-Zegarlinski-03}),  has
been introduced into the field of stochastic homogenization by Naddaf
and Spencer \cite[Theorem~1]{Naddaf-Spencer-98} (in form of the
Brascamp-Lieb inequality), the LSI has been used in \cite{MO1} in the
context of stochastic homogenization to obtain optimal annealed estimates on the gradient of the Green's
function and bounds on the random part of the homogenization error
  $u_\e - \langle u_\e \rangle$. 

\medskip
Note that in the special case of diagonal
coefficients (i.e.\ when the maximum principle \stef{and the
De~Giorgi-Nash-Moser regularity theory} is available) our
results are not new: The $T$-independent results on $\phi_T$
and $\nabla \phi_T$ in $d>2$ dimensions have already been established
in~\cite{GO1,GNO1} under the slightly weaker assumption SG on the
statistics (see \eqref{eq:SG} below), \stef{and the estimate on the corrector in the optimal
form of $\langle |\phi_T|^{2p} \rangle \le C (\log T)^p$ with a
constant independent of $T$ is obtained in \cite{GNO1}.}
\medskip

{\bf Relation to random walks in random environments.} There is a
strong link between stochastic homogenization and random walks in
random environments (see \cite{Biskup-11} and \cite{Kumagai-14} for
recent surveys). Suppose for a moment that $\expec{\cdot}$
concentrates on diagonal matrices. Then for each
diagonal-matrix-valued field $a:\Z^d\to\in\R^{d\times d}$, we may interpret
\eqref{eq:13} as a conductance network, where each edge
$[x,x+e_i]$ ($x\in\Z^d$, $i=1,\ldots,d$) is endowed with the
conductance $a_{ii}(x)$. The elliptic operator $\nabla^*(a\nabla)$
generates a stochastic process, called the \emph{variable speed random walk}
$X=(X_a(t))_{t\geq 0}$ in a random environment with law $\expec{\cdot}$.
Using arguments from stochastic homogenization, Kipnis and
Varadhan \cite{Kipnis-Varadhan-86} (see also \cite{Kunnemann-83} for
an earlier result) show
that the law of the rescaled process
$\sqrt{\e}X(\e t)$ converges weakly to that of a Brownian motion
with covariance $2a_{\hom}$. This \textit{annealed} invariance
principle for $X$ has been upgraded to a \textit{quenched} result by
Sidoravicious and Sznitman \cite{Sidoravicius-Sznitman-04}. The key
ingredient in their argument is to prove that the ``anchored
corrector'' (i.e.\ the function $\varphi$ introduced in
Corollary~\ref{cor:1} (a) below) satisfies a \textit{quenched
sublinear growth} property. 
The quantitative analysis derived in the present paper is stronger. Indeed, our
estimate on $\nabla\phi_T$ almost immediately
implies that the anchored corrector grows sublinearly. On top of
  that in dimensions $d>2$ the moment bound on $\phi_T$ implies that the anchored corrector is almost
bounded, in the sense that it grows \textit{slower than any rate}, see
Corollary~\ref{cor:1} and the subsequent remark.
\smallskip 

If the coefficients are not diagonal, then \eqref{eq:13} is not any
longer related to a random conductance model. As mentioned before, for
non-symmetric $a$ (and even for certain symmetric coefficients) the
maximum principle for $\nabla^*(a\nabla)$ generally fails to hold. In that case
the semigroup generated by $\nabla^*(a\nabla)$ is not a Markov process and there is no natural
probabilistic interpretation for \eqref{eq:13}. {This may also be seen in terms of Dirichlet forms. While the (non-symmetric) elliptic operator $-\operatorname{div}(a_{\hom}\nabla)$ acting on functions on $\R^d$ generates a Dirichlet form $\int_{\R^d} \nabla u \cdot a_{\hom} \nabla v dx$ in the sense of \cite[Definition~I.4.5]{Ma-Rockner} and a corresponding Markov process, the discrete operator $\nabla^*(a\nabla)$ with associated bilinear form $\sum_{\Z^d}\nabla u \cdot a \nabla v$ defined on $\ell^2(\Z^d)\times\ell^2(\Z^d)$  does not. Indeed, the contraction property~(4.4) in \cite{Ma-Rockner} (which encodes a maximum principle) generally fails to hold in the non-diagonal discrete case.}
However, the  limiting process can be approximated by (non-symmetric)
Markov processes, see \cite{Deuschel-Kumagai-13} for a recent construction.  
\medskip

Let us finally remark that we do not use any ingredients from probability theory except for the quantification of ergodicity via SG and LSI in this paper. Furthermore,
since we view our present contribution as a first step towards systems
(which certainly are unrelated to
probability theory), we do not further investigate the connection to
random walks in the present paper. 
\bigskip

{\bf Outline of the paper.} In Section~\ref{S:FW}, we present the main
results of our paper and give a brief sketch of our proof. The proof
of the main results and auxiliary lemmas are contained in Section~\ref{S:P}. 
Let us mention that in the critical dimension $d=2$, we invoke a Calder\'{o}n-Zygmund estimate
on weighted $\ell^p$-spaces on $\Z^d$. We give a proof of this
estimate, which may be of independent interest, in Section~\ref{S:CZ}.  Finally, in Section \ref{sec:app} we present some applications, including a quantitative two-scale expansion and a variance estimate for a representative volume approximation of the homogenized coefficients.
\bigskip

{\it{\bf Acknowledgements.} The authors gratefully acknowledge Felix
  Otto for suggesting the problem and for helpful discussions.
  J.~B.-A.~and S.~N.~thank the Max-Planck-Institute for Mathematics in
  the Sciences, Leipzig, for its hospitality. S.~N.~was
partially supported by ERC-2010-AdG no.267802 AnaMultiScale. }

\section{Main results and sketch of proof}\label{S:FW}
\subsection{General framework}

{\bf Discrete functions and derivatives.} 
Let $\{e_i\}_{i=1}^d$ denote the canonical basis of $\R^d$. For a
scalar function $u:\Z^d\to\R$ and a vector field $g:\Z^d\to\R^d$
with components $g=(g_1,\ldots,g_d)$ we define the
discrete gradient $\nabla u:\Z^d\to\R^d$ and negative divergence $\nabla^*g:\Z^d\to\R$ as follows:
\begin{align*}
  &\nabla u:=(\nabla_1u,\ldots,\nabla_du),\qquad \nabla^*g:=\sum_{i=1}^d\nabla^*_ig_i,\qquad\text{where}\\
  &\nabla_iu(x):=u(x+e_i)-u(x),\qquad
  \nabla^*_iu(x):=u(x-e_i)-u(x).
\end{align*}
We denote by $\ell^p(\Z^d)$,
$1\leq p\leq \infty$, the space of functions $u:\Z^d\to\R$ with
$\|u\|_{\ell^p}<\infty$, where
$\|u\|_{\ell^p}:=\left(\sum_{x\in\Z^d}|u(x)|^p\right)^{\frac{1}{p}}$
for $p<\infty$ and $\|u\|_{\ell^\infty}:=\sup_{x\in\Z^d}|u(x)|$.  Note that $\nabla$ and $\nabla^*$ are adjoint: We have the discrete integration by parts formula
\begin{equation*}
  \sum_{x\in\Z^d}\nabla u(x)\cdot g(x)=\sum_{x\in\Z^d}u(x)\nabla^*g(x)
\end{equation*}
for all exponents $1\le p,q \le \infty$ such that $1 = \frac{1}{p} + \frac{1}{q}$ and all functions $u\in\ell^p(\Z^d)$ and $g\in\ell^q(\Z^d,\R^d)$.
\medskip

{\bf Random coefficients and quantitative ergodicity.}
In order to describe random coefficients, we endow $\Omega$ with the product topology induced by $\R^{d \times
  d}$ and denote by $C_b(\Omega)$ the set of continuous functions
$\zeta:\Omega\to\R$ that are uniformly bounded in the sense that
\begin{equation*}
  \|\zeta\|_\infty:=\sup_{a\in\Omega}|\zeta(a)|<\infty.
\end{equation*}
Throughout this work, we consider a probability measure on $\Omega$ with respect to the Borel-$\sigma$-algebra. Following the convention in statistical mechanics, we call this probability measure an \emph{ensemble} and write $\expec{\cdot}$ for the associated expected
value, the ensemble average. We assume that $\expec{\cdot}$ is
\textit{stationary} w.~r.~t.~translation on $\Z^d$, i.e.~for all
$x\in\Z^d$, the mapping $\tau_x : \Omega\to\Omega, a\mapsto a(\cdot+x)$ is
measurable and measure preserving:
\[
 \forall \zeta: \Omega\to\R:\quad \langle \zeta( \tau_x \cdot) \rangle = \langle \zeta(\cdot) \rangle.
\]
Our key  assumption is that $\expec{\cdot}$ is \textit{quantitatively ergodic} where the ergodicity is quantified through either LSI or SG. To be precise, we make the following definitions:
\begin{definition}[Definition~1 in~\cite{MO1}]\label{def:LSI}
 We say that $\langle \cdot \rangle$ satisfies the LSI with constant $\rho>0$ if
\begin{equation}
  \label{eq:LSI}
  \expec{\zeta^2\log\frac{\zeta^2}{\expec{\zeta^2}}}\le \frac{1}{2\rho}\expec{\sum_{x\in\Z^d} \Big( \osc_{a(x)} \zeta\Big)^2 }.
\end{equation}
for all $\zeta\in C_b(\Omega)$.
\end{definition}
Here the {\it oscillation} of a
function $\zeta\in C_b(\Omega)$ is defined by taking the oscillation over all $\tilde a\in\Omega$ that coincide with $a$ outside of $x\in\Z^d$, i.e.\
\begin{multline}\label{eq:osc}
  \osc_{a(x)} \zeta(a) := \sup\{ \zeta(\tilde a) \ | \ \tilde a\in \Omega \text{ s.t.\ } \tilde a(y)=a(y)\ \forall y\neq x \}\\ - \inf\{ 
\zeta(\tilde a) \ | \ \tilde a\in \Omega \text{ s.t.\ } \tilde a(y)=a(y)\ \forall y\neq x  \}.
\end{multline}
The continuity assumption on $\zeta$ ensures that the oscillation is well-defined. A weaker form of quantitative ergodicity is the SG which is defined as follows.
\begin{definition}\label{def:SG}
 We say that $\langle \cdot \rangle$ satisfies the SG with constant $\rho>0$ if
\begin{equation}
  \label{eq:SG}
  \expec{(\varphi-\expec{\varphi})^2}\le \frac{1}{\rho}\expec{\sum_{x\in\Z^d} \Big( \osc_{a(x)} \varphi\Big)^2 }
\end{equation}
for all $\varphi\in C_b(\Omega)$.
\end{definition}
The SG~\eqref{eq:SG} is automatically satisfied if LSI~\eqref{eq:LSI} holds, which may be seen by expanding $\zeta = 1+\epsilon\varphi$ in powers of $\epsilon$.
Moreover, LSI and SG are satisfied in the case of independently and identically distributed coefficients, i.e.~when $\expec{\cdot}$ is the
$\Z^d$-fold product of a probability measure on $\Omega_0$,
cf.~\cite[Lemma~1]{MO1}. We refer to~\cite{Guionnet-Zegarlinski-03}
for a recent exposition on LSI and to~\cite{GNO1} for a systematic application of SG to stochastic
homogenization.

\medskip

\subsection{Main results} 
Throughout this paper the modified corrector $\phi_T$ is defined as the
unique bounded solution to
\eqref{eq:1}, see Lemma~\ref{LMC} below for details. Our first result yields boundedness of the finite
moments of  $\nabla\phi_T$.

\begin{theorem}\label{T1}
  Assume that $\langle \cdot \rangle$ is stationary and satisfies LSI~\eqref{eq:LSI} with constant $\rho>0$.  Then the modified corrector defined via \eqref{eq:1} satisfies
  \begin{equation}\label{eq:Dphi}
    \langle |\nabla\phi_T(x) + \xi|^{2p} \rangle \le C(d,\lambda,p,\rho) |\xi|^{2p}
  \end{equation}
   for all $x\in\Z^d$, $p<\infty$ and $T\ge2$.
  Here and throughout this work, $C(d,\lambda,p,\rho)$ stands for a constant which may change from line to line and that only depends on the exponent $p$, the LSI-constant $\rho$, the ellipticity ratio $\lambda$ and the dimension $d$.
 \end{theorem}
As already mentioned earlier, the lower bound ``2'' for $T$ is arbitrary and may be replaced by any other constant greater than 1.
The second result establishes moment bounds on the corrector itself.
More precisely, we establish control of moments of $\phi_T$ by moments
of $\nabla\phi_T$. As opposed to Theorem~\ref{T1}, we just need to assume that the ensemble satisfies SG, i.e.~Definition~\ref{def:SG}.

\begin{theorem}\label{T2}
  Assume that $\langle \cdot \rangle$ is stationary and satisfies SG~\eqref{eq:SG} with constant $\rho>0$. {There exists $p_0 = p_0(d,\lambda)$ such that} the the modified corrector defined via~\eqref{eq:1} satisfies
  \begin{equation}\label{eq:phi}
    \langle |\phi_T(x)|^{2p} \rangle \le C(d,\lambda,p,\rho) \expec{|\nabla \phi_T(x)+\xi|^{2p}}\times
    \begin{cases}
      (\log T)^p&\text{for }d=2,\\
      1&\text{for }d>2,
    \end{cases}
  \end{equation}
 for all {$x\in\Z^d$, $p\ge p_0$ and $T\ge2$.}
\end{theorem}

By letting $T\uparrow\infty$, we obtain the following estimate for
  the (unmodified)  corrector.

\begin{corollary}
  \label{cor:1}
  Assume that $\langle \cdot \rangle$ is stationary and
  satisfies LSI~\eqref{eq:LSI} with constant $\rho>0$. Then:
  \begin{enumerate}[(a)]
  \item[(a)] In dimensions $d\geq 2$ there
    exists a unique measurable function
    $\varphi:\Omega\times\Z^d\to\R$ that solves \eqref{eq:corr}  for
    $\expec{\cdot}$-almost every $a\in\Omega$ and
    \begin{itemize}
    \item[(a1)] $\varphi$ satisfies the anchoring condition $\varphi(a,0)=0$ for $\expec{\cdot}$-almost every $a\in\Omega$,
    \item[(a2)] $\nabla\varphi$ is stationary in the sense 
of~\eqref{LMC:1a} and $\expec{\nabla\varphi(x)}=0$ for all $x\in\Z^d$,
    \item[(a3)] $\expec{|\nabla\varphi(x)|^p}<\infty$ for all $x\in\Z^d$ and $p<\infty$. 
    \end{itemize}
  \item[(b)] In dimensions $d>2$ there
    exists a unique measurable function $\phi:\Omega\times\Z^d\to\R$ that solves
    \eqref{eq:corr}  for $\expec{\cdot}$-almost every $a\in\Omega$,
    and
    \begin{itemize}
    \item[(b1)] $\phi$  is stationary in the sense of~\eqref{LMC:1a},
    \item[(b2)] $\expec{|\phi(x)|^p}<\infty$ for all $x\in\Z^d$ and $p<\infty$. 
    \end{itemize}

  \end{enumerate}
\end{corollary}
\begin{remark}
  \label{sec:labelt2-assume-that}
  \begin{itemize}
  \item The ``anchored corrector'' $\varphi$ defined in  Corollary~\ref{cor:1} (a) has already been considered in the
    seminal works by Papanicolaou and Varadhan
    \cite{Papanicolaou-Varadhan-79} and Kozlov \cite{Kozlov-79}. In
    fact, for existence and uniqueness -- which can be proved by soft
    arguments -- only (a1) and
    (a2) are required. The new estimate
    (a3) follows from Theorem~\ref{T1} in the
    limit $T\uparrow\infty$. Note that (a3) implies (by a short ergodicity argument) sublinearity
    of the anchored corrector in the sense that
    \begin{equation*}
      \lim\limits_{R\uparrow\infty}\max_{|x|\leq R}\frac{|\varphi(a,x)|}{R}=0
    \end{equation*}
    for $\expec{\cdot}$-almost every $a\in\Omega$.
    \item Existence, uniqueness and moment bounds of the ``stationary corrector''
      $\phi$ defined in  Corollary~\ref{cor:1} (b) have been obtained in the
      case of diagonal coefficients in \cite{GO1}, see also
      \cite{GNO1}. Note that the anchored corrector $\varphi$ can be obtained
      from $\phi$ via $\varphi(x,a):=\phi(a,x)-\phi(a,0)$, and, as explained in the discussion below \cite[Corollary~1]{LNO1}, the moment bound (b2) implies that 
    \begin{equation*}
      \forall\theta\in (0,1]\,:\qquad \lim\limits_{R\uparrow\infty}\max_{|x|\leq R}\frac{|\varphi(a,x)|}{R^\theta}=0
    \end{equation*}
    for $\expec{\cdot}$-almost every $a\in\Omega$.

  \end{itemize}
\end{remark}

\begin{remark}
  Instead of the modified corrector, one might consider the periodic corrector which in the stochastic context is defined as follows: For
$L\in\N$ let
  \begin{equation*}
    \Omega_L:=\{\,a\in\Omega\,:\,a(\cdot+Lz)=a\;\;\text{for all }z\in\Z^d\,\}
  \end{equation*}
  denote the set of $L$-periodic coefficient fields. In the $L$-periodic case, one considers the corrector equation \eqref{eq:corr}
together with an $L$-periodic ensemble, i.~e. a stationary probability measure on $\Omega_L$. In that case, equation \eqref{eq:corr}
admits a unique solution $\phi_L$ with $\sum_{x\in([0,L)\cap\Z)^d}\phi_L(x)=0$ for all $a\in\Omega_L$. The $L$-periodic versions of LSI and
SG are obtained by replacing the sum $\sum_{x\in\Z^d}$ in \eqref{eq:LSI} and \eqref{eq:SG} by $\sum_{x\in([0,L)\cap\Z)^d}$. With
these modifications, Theorem~\ref{T1} and Theorem~\ref{T2} extend to the $L$-periodic case ({with $L=\sqrt{T}$ since the cut-off term involving $T$ effectively restricts the equation to a domain of side length $\sqrt{T}$}). In particular, if
the $L$-periodic ensemble satisfies an $L$-periodic LSI with constant $\rho>0$, then the $L$-periodic corrector satisfies for all
$p<\infty$ 
  \begin{equation*}
    \expec{\phi_L^{2p}}^{\frac{1}{2p}}\lesssim
    \begin{cases}
      (\log L)^{\frac{1}{2}}&\text{for }d=2,\\
      1&\text{otherwise.}
    \end{cases}
  \end{equation*}
  The proof follows along the same lines and can easily be adapted.
  For estimates on the periodic corrector $\phi_L$ in the case of diagonal
  coefficients, see \cite{GNO1}.
\end{remark}

\subsection{Sketch of proof of Theorem~\ref{T1}}\label{SS:sketch1}
Theorem~\ref{T1} is relatively straight-forward to prove. We simply
follow the approach developed in~\cite{MO1} and use the
LSI~\eqref{eq:LSI} of Definition~\ref{def:LSI} to upgrade a lower
order $L^2_{\langle \cdot \rangle}(\Omega)$-bound to a bound in
$L^{2p}_{\langle \cdot \rangle}(\Omega)$. Note that by stationarity of
$\expec{\cdot}$ and $\phi_T$, see \eqref{LMC:1a}, it suffices to prove the estimates~\eqref{eq:Dphi} at $x=0$. The lower order bound
\[
 \langle |\nabla\phi_T(0) + \xi|^{2} \rangle \le C(d,\lambda) |\xi|^{2},\qquad\text{cf.~\eqref{T1.3},}
\]
follows from a simple energy argument, i.e.\ an $L^2$-estimate obtained by testing the equation for $\phi_T$ with $\phi_T$ itself. The integral here is the ensemble average and not the sum over $\Z^d$; this is possible thanks to stationarity of $\phi_T$. For details, we refer to Step~1 in the proof of Theorem~\ref{T1}. This bound is then upgraded via the following consequence of LSI~\eqref{eq:LSI}:
\[
 \langle |\nabla \phi_T(0) + \xi|^{2p} \rangle \le C(d,p,\rho,\delta) \langle |\nabla \phi_T(0) + \xi|^{2} \rangle^p + \delta \expec{ \bigg( \sum_{x\in\Z^d} \Big| \osc_{a(x)} \nabla\phi(0) \Big|^2 \bigg)}
\]
for all $\delta > 0$, where we have implicitly taken the oscillation
of the vector $\nabla \phi_T$ component-wise. This reverse Jensen
inequality is the content of Lemma~\ref{L1} below. Next, we need an
expression for $\osc_{a(x)} \nabla \phi_T$. In Lemma~\ref{L:RP} we
will show that the response to a variation at $x$ in the coefficient field is given via the Green's function $G_T$ as:
\[
 \osc_{a(x)}(\nabla_j \phi_T(a;0) + \xi_j) \leq C(d,\lambda) |\nabla \nabla G_T(a;0,x)| |\nabla \phi_T(a;x) + \xi|,
\]
where $G_T$ is the Green's function associated to~\eqref{eq:1}, see Definition~\ref{def:1}. 
Throughout this work, $\nabla\nabla G_T(x,y) = \nabla_x\nabla_y G_T(x,y)\in\R^{d\times d}$ denotes the mixed derivative and we use the spectral norm on $\R^{d\times d}$.
The above estimate on the oscillation then yields
\begin{align*}
 \expec{ \bigg( \sum_{x\in\Z^d} \Big| \osc_{a(x)} \nabla\phi(0) \Big|^2 \bigg)} &\le C(d,\lambda,p) \expec{ \bigg( \sum_{x\in\Z^d} |\nabla \nabla G_T(a;0,x)|^2 |\nabla \phi_T(a;x) + \xi|^2 \bigg)}\\
 &\le C(d,\lambda,p) \expec{ |\nabla \phi_T(a;0) + \xi|^2 },
\end{align*}
where in Step~2 of the proof of Theorem~\ref{T1}, we will obtain the last inequality from stationarity and the energy estimate~\eqref{P1.18}, i.e.\
\[
 \sum_{x\in\Z^d} |\nabla\nabla G_T(x,y)|^2\le C(d,\lambda),
\]
which holds in any dimension $d\ge 2$.
\subsection{Sketch of proof of Theorem~\ref{T2}}\label{SS:sketch2}
By stationarity of $\expec{\cdot}$ and $\phi_T$, it suffices to prove \eqref{eq:phi} at $x=0$. 
In contrast to Theorem~\ref{T1}, the proof of Theorem~\ref{T2} only requires the weaker ergodicity assumption SG of Definition~\ref{def:SG}, which we will use in form of
\[
\expec{|\phi_T(0)|^{2p}}\le C(p,\rho) \expec{\bigg(\sum_{x\in\Z^d} \Big( \osc_{a(x)} \phi_T(0) \Big)^2 \bigg)^p},
\]
see Lemma~\ref{LSGp} below. Again, we require an estimate on the oscillation, which we shall obtain in Lemma~\ref{L:RP} and which yields
\[
 \osc_{a(x)} \phi_T(a;0) \le C(d,\lambda) |\nabla_x G_T(a;0,x)| |\nabla \phi_T(a;x) + \xi|.
\]
This will be substituted into the above SG-type inequality.
In contrast to the proof of Theorem~\ref{T1}, where a simple $\ell^2$-estimate of $\nabla\nabla G_T$ sufficed, we will see that we require a bound on $\nabla G_T$ including weights: In Lemma~\ref{L2}, we show that 
\[
    \sum_{x\in\Z^d}|\nabla_x G_T(a;0,x)|^{2q}\omega_q(x)
    \le C(d,\lambda,q)
    \begin{cases}
      \log T&\text{for }d=2,\\
      1&\text{for }d>2
    \end{cases}
\] 
for all $q\ge1$ close enough to $1$, and weight $\omega_q$ given by
\[
  \omega_q(x):=
  \begin{cases}
    (|x|+1)^{2(q-1)}+T^{1-q}(|x|+1)^{4(q-1)}&\text{for }d=2,\\
    (|x|+1)^{2d(q-1)}&\text{for }d>2.
  \end{cases}
\]
The case $d>2$ is relatively straight-forward and follows by testing the equation with weights and applying Hardy's inequality.
The case $d=2$ is critical for this estimate and we will prove it by reducing the problem via a perturbation argument to the constant-coefficient case; this approach involves a Helmholtz projection and is inspired by the work~\cite{Conlon-Spencer-11}. To make it rigorous, we require a Calder\'{o}n-Zygmund estimate in discrete weighted spaces which may be of independent interest and which is proved in Section~\ref{S:CZ}. With this estimate at hand, we may smuggle in the weight $\omega_q$ and apply H\"older's inequality with $q\approx 1$ and large dual exponent $p$ to obtain 
\begin{multline*}
\expec{ \bigg( \sum_{x\in\Z^d} |\nabla_x G_T(a;0,x)|^2 |\nabla \phi_T(a;x) + \xi|^2 \bigg)}\\ \le C(d,\lambda,q) \expec{ |\nabla \phi_T(a;x) + \xi|^{2p} }
    \begin{cases}
      \log T&\text{for }d=2,\\
      1&\text{for }d>2
    \end{cases}
\end{multline*}
as long as $p$ is large enough such that $\sum_x \omega_q^{1-p}(x) < \infty$.

\section{Auxiliary results and proofs}\label{S:P}
In this section we first present and prove some auxiliary results and
then turn to the actual proofs of our main results. We start in
Section~\ref{SS:W} with the
definition of the modified corrector and prove its existence and some
continuity properties. This invokes the elliptic Green's function,
which we introduce in the same section. Section~\ref{SS:G} and
Section~\ref{SS:E} contain the two key ingredients of our approach: In Section~\ref{SS:G}, we prove estimates on the oscillation
of the corrector and estimates on the gradient of the Green's
function; in
Section~\ref{SS:E}, we revisit LSI and SG, which quantify ergodicity
and are the only ingredients from probability theory in our approach.
Finally in
Sections~\ref{SS:T1} and \ref{SS:T2}, we present the proofs of
Theorems~\ref{T1} and \ref{T2}.
\subsection{Well-posedness of the modified corrector}\label{SS:W}
We define the modified
  corrector $\phi_T:\Omega\times\Z^d\to\R$ as the unique
  bounded solution to \eqref{eq:1}, i.e.\ for each
  $a\in\Omega$, we require
  $\phi_T(a,\cdot):\Z^d\to\R$ to solve \eqref{eq:1} and to be
  bounded, see Lemma~\ref{LMC} for details. Note that this definition is {\em
    pointwise in $a\in\Omega$} and does not invoke any
  probability measure on $\Omega$. This is in contrast to what is typically done in stochastic
  homogenization (e.g.\ in the seminal work
  \cite{Papanicolaou-Varadhan-79}, where $\phi_T$ is unambigously defined through
  an equation on the probability space $L^2_{\expec{\cdot}}(\Omega)$).
  We opt for the ``non-probabilistic'' definition, since
  later we need to estimate the oscillation in $a$ of $\phi_T$, which is
  most conveniently done when $\phi_T$ is defined for {\em all}
  $a\in\Omega$ and not only $\expec{\cdot}$-almost surely.

  However, since the right-hand
  side of~\eqref{eq:1} is only in $\ell^\infty(\Z^d)$, it is not clear
  a-priori whether \eqref{eq:1} admits a bounded solution. To settle this question we
  consider the elliptic Green's function
  $G_T:\Omega\times\Z^d\times\Z^d\to \R$ and prove integrability of $G_T$ in
  Lemma~\ref{L:Gint} below. The latter then implies existence of $\phi_T$
  together with some continuity properties, 
  see Lemma~\ref{LMC} below.

\begin{definition}[Green's function]
  \label{def:1}
   {Given $a\in\Omega$ and $y\in\Z^d$, the Green's function $G_T(a;x,y)$ associated to equation~\eqref{eq:1} is the
  unique solution in  $\ell^2(\Z^d)$ to
  \begin{equation}\label{T1.2}
  \frac{1}{T}G_T(a;\cdot,y)+\nabla^*(a\nabla G_T(a;\cdot,y))=\delta(\cdot-y)\qquad\text{in }\Z^d,
\end{equation}
where $\delta:\Z^d\to\{0,1\}$ denotes the Dirac function centered at $0$. }
\end{definition}
Equation~\eqref{T1.2} can also be expressed in its ``weak" formulation: For all $w\in\ell^2(\Z^d)$ we have that
\begin{equation}\label{P1.3}
  \frac{1}{T}\sum_{x\in\Z^d}G_T(a;x,y)
  w(x)+\sum_{x\in\Z^d}\nabla w(x)\cdot a(x)\nabla_x
G_T(a;x,y)=w(y).
\end{equation}
It immediately follows from the unique characterization of $G_T$ through
(\ref{T1.2}) that the Green's function is stationary:
\begin{equation}\label{P1.15}
\nabla\nabla G_T(a,x+z,y+z)=\nabla\nabla G_T(a(\cdot+z),x,y).
\end{equation}
Furthermore it is symmetric in the sense that
\begin{equation}\label{P1.16}
\nabla\nabla G_T(a;y',y)=\nabla\nabla G_T(a^t;y,y'),
\end{equation}
where $a^t$ denotes the transpose of $a$ in $\R^{d\times d}$.
This can be seen from applying (\ref{P1.3})  to $w(x)= G_T(a^t;x,y')$, yielding
the representation \[G_T(a^t;y,y')=\frac{1}{T}\sum_{x}
G_T(a^t;x,y') G_T(a;x,y)+ \sum_{x}\nabla_x G_T(a^t;x,y') \cdot a(x)\nabla_x G_T(a;x,y).\] On the other hand, choosing $w(x) = G_T(a;x,y)$
in the definition for $G_T(a^t;\cdot,\cdot)$ shows \[G_T(a;y',y)=\frac{1}{T}\sum_{x}
G_T(a;x,y) G_T(a^t;x,y') + \sum_{x}\nabla_x G_T(a;x,y) \cdot a^t(x)\nabla_x G_T(a^t;x,y').\] By definition of the transpose $a^t$, this
shows $G_T(a;y,y') = G_T(a^t;y',y)$ and hence \eqref{P1.16}.

The Green's function is useful since by linearity it encodes all the information for the solution $u$ to the equation
\begin{equation}\label{eq:7}
  \frac{1}{T}u+\nabla^*(a\nabla u)=f\qquad\text{in }\Z^d.
\end{equation}
Indeed, testing \eqref{eq:1} with $G_T(a;\cdot,y)$ and
integrating by parts \textit{formally} yields
\begin{equation}\label{eq:6}
  u(a;x)=\sum_{y\in\Z^d}G_T(a;x,y)f(y).
\end{equation}
 Of course, to make sense of this for $f=\nabla^*(a\xi)\in\ell^\infty(\Z^d)$, we need $G_T$ in $\ell^1(\Z^d)$. On the other hand, the definition of the Green's function only yields $G_T(\cdot,y) \in \ell^2(\Z^d)$ but this is not enough to establish well-posedness of~\eqref{eq:1}. It is not difficult to establish that $\sum_x G_T(x,y) = T$ for all $y\in\Z^d$ and $a\in\Omega$ but without the maximum principle, $G_T$ may be negative and it does not follow that $G_T$ is in $\ell^1(\Z^d)$. Therefore we need another argument to establish well-posedness of~\eqref{eq:1}. This is provided by the following lemma, which shows exponential decay of $G_T$ and in particular that $G_T$ is in $\ell^1(\Z^d)$.
\begin{lemma}\label{L:Gint}
  {There exist a large constant $C=C(d,\lambda,T)<\infty$ and a small constant $\delta=\delta(d,\lambda,T)>0$, both only depending on $d$, $\lambda$
  and $T$, such that
  \begin{equation*}
    \sum_{x\in\Z^d}\Big(|G_T(a;x,y)|^2+|\nabla_x G_T(a;x,y)|^2\Big)e^{\delta(d,\lambda,T)|x-y|}\leq C(d,\lambda,T)
  \end{equation*}
  for all $a\in\Omega$ and $y\in\Z^d$.}
\end{lemma}
Since we could not find a suitable reference for this estimate in the
discrete, non-symmetric case, we present a proof in the appendix. The proof is \textcolor{blue}{essentially done by testing with $e^{\frac{\delta}{2}|x|}$ (this is also known as  \emph{Agmon's positivity method}~\cite{Agmon})}. In the discrete setting this is inspired by~\cite[Proof of Lemma~3]{Gloria10}.
With this result at hand, we can provide well-posedness of the modified corrector $\phi_T$. In addition to well-posedness, Lemma~\ref{L:Gint} allows us to deduce $\phi_T(0) = \phi_T(a; 0) \in C_b(\Omega)$, which is necessary for the application of LSI~\eqref{eq:LSI} and SG~\eqref{eq:SG} to $\phi_T$.
\begin{lemma}[Modified corrector]
  \label{LMC}
  For all $a\in\Omega$ the modified corrector equation \eqref{eq:1} admits a unique
  bounded solution $\phi_T(a;\cdot)\in\ell^\infty(\Z^d)$. The so
  defined modified corrector $\phi_T:\Omega\times\Z^d\to\R$ satisfies
  $\phi_T(\cdot,x)\in C_b(\Omega)$ for all $x\in\Z^d$, 
  and
  \begin{equation}\label{LMC:2}
    |\phi_T(a;x)|\leq C(T,\lambda,d)|\xi|\qquad\text{for all
    $a\in\Omega$ and all $x\in\Z^d$}.
  \end{equation}
  Furthermore, $\phi_T$ is stationary, i.e.\ 
    \begin{equation}
      \label{LMC:1}
      \phi_T(a;x+z)=\phi_T(a(\cdot+z);x)\qquad\text{for all
    $a\in\Omega$ and all $x,z\in\Z^d$}.
    \end{equation}
\end{lemma}
\begin{proof}
{\bf Step 1}. Existence and uniqueness of $\phi_T$:
In this step, we argue that for arbitrary $f\in\ell^\infty(\Z^d)$
equation \eqref{eq:7} admits a unique solution $u$ and $u$ can be represented as in \eqref{eq:6}. The existence and
uniqueness of $\phi_T$ then follows by setting $f:=-\nabla^*(a\xi)$.
For the argument, note that by Lemma~\ref{L:Gint} we have
$G_T(a;\cdot,y)\in\ell^1(\Z^d)$. Hence, for every
$f\in\ell^\infty(\Z^d)$, equation \eqref{eq:6} defines a function
$u(a;\cdot)\in\ell^\infty(\Z^d)$ that solves \eqref{eq:7}. For the
uniqueness, let $\tilde u\in\ell^\infty(\Z^d)$ solve \eqref{eq:7}.
Testing \eqref{eq:7} with $G_T(a^t;\cdot,x)$ yields
\begin{eqnarray*}
  \sum_{y\in\Z^d}G_T(a^t;y,x)f(y)&=&\sum_{y\in\Z^d}G_T(a^t;y,x)\Big(\frac{1}{T}+\nabla^*(a\nabla))\Big)\tilde u(y)\\
  &=&\sum_{y\in\Z^d}\Big(\frac{1}{T}+\nabla^*(a^t\nabla)\Big)G_T(a^t;y,x)\tilde
  u(y)\\
  &=&\sum_{y\in\Z^d}\delta(x-y)\tilde u(y)=\tilde u(x).
\end{eqnarray*}
By symmetry the left-hand side is equal to $\sum_{y\in\Z^d}G_T(a;x,y)f(y)=u(a;x)$ and thus $u(a;\cdot)=\tilde
u(\cdot)$ follows.
\medskip

{\bf Step 2}. Argument for \eqref{LMC:2} and \eqref{LMC:1}:
The stationarity property \eqref{LMC:1} directly follows from
uniqueness and the
stationarity of the operator and the right-hand side
$-\nabla^*(a\xi)$. We turn to estimate \eqref{LMC:2}. By the Green's representation \eqref{eq:6}, which is valid by  Step~1,
and an integration by parts (possible since $G_T(x,\cdot) \in \ell^1(\Z^d)$), we have
\begin{equation*}
  \phi_T(a;x)=\sum_{y\in\Z^d}\nabla_yG_T(a;x,y)\cdot a(y)\xi.
\end{equation*}
We smuggle in the exponential weight from
Lemma~\ref{L:Gint}, use uniform ellipticity and the Cauchy-Schwarz inequality to get
\begin{align*}
  |\phi_T(a;x)|&\leq\sum_{y\in\Z^d}\left(|\nabla_yG_T(a;x,y)|e^{\frac{\delta}{2}|y|}\right)\left(|a(y)\xi|e^{-\frac{\delta}{2}|y|}\right)\\
  &\leq\left(\sum_{y\in\Z^d}|\nabla_yG_T(a;x,y)|^2e^{\delta|y|}\right)^{\frac{1}{2}}\left(\sum_{y\in\Z^d}e^{-\delta|y|}\right)^{\frac{1}{2}}|\xi|,
\end{align*}
where $\delta>0$ is given in Lemma~\ref{L:Gint}.
By symmetry, cf.~\eqref{P1.16}, and Lemma~\ref{L:Gint}, the
right-hand side is bounded by $C(d,\lambda,T)|\xi|$ and \eqref{LMC:2} follows.

{\bf Step 3}. Argument for $\phi_T(\cdot;x)\in C_b(\Omega)$:
Thanks to \eqref{LMC:2}, we only need to show that $\phi_T(a;x)$ is
continuous in $a$. Furthermore, by stationarity, cf.~\eqref{LMC:1}, it
suffices to consider $\phi_T(a;0)$. Now, consider a sequence $a_n\in\Omega$ that
converges to some $a\in\Omega$ in the product topology. We need to
show that $\phi_T(a_n;0)\to\phi_T(a;0)$. To that end, consider the function
\begin{equation*}
  \psi_n(x):=\phi_T(a_n;x)-\phi_T(a;x),
\end{equation*}
which can be characterized as the unique bounded solution to
\begin{equation*}
  \frac{1}{T}\psi_n+\nabla^*(a_n\nabla\psi_n)=\nabla^*((a-a_n)(\nabla\phi_T(a,\cdot)+\xi))\qquad\text{in }\Z^d.
\end{equation*}
Hence, by Step~1 we have
\begin{equation*}
  \psi_n(0)=\sum_{y\in\Z^d}\nabla_yG_T(a_n;0,y)\cdot (a(y)-a_n(y))(\nabla\phi_T(a,y)+\xi),
\end{equation*}
and thus Lemma~\ref{L:Gint} and the result of Step~2 yield
\begin{align*}
  |\psi_n(0)|&\leq\left(\sup_{y\in\Z^d}\sup_{a\in\Omega}|\nabla\phi_T(a,y)+\xi|\right)\\
  &
  \qquad\times\,\left(\sum_{y\in\Z^d}|\nabla_yG_T(a_n;0,y)|^2e^{\delta|y|}\right)^{\frac{1}{2}}\left(\sum_{y\in\Z^d}e^{-\delta|y|}|a(y)-a_n(y)|^2\right)^{\frac{1}{2}}\\
  &\le C(T,\lambda,d)\,\left(\sum_{y\in\Z^d}e^{-\delta|y|}|a(y)-a_n(y)|^2\right)^{\frac{1}{2}}.
\end{align*}
Since $a_n\to a$ in the product topology, i.e.\ $a_n(y)\to a(y)$ for all $y\in\Z^d$, the right-hand side vanishes as $n\to\infty$ by dominated convergence.
\end{proof}

\subsection{Oscillations and Green's function estimates}\label{SS:G}
In this section, we estimate the oscillation of the
corrector and its gradient, see Lemma~\ref{L:RP} below, and establish estimates
on the gradient of the elliptic Green's functions, see Lemma~\ref{L2} below. These bounds
are at the core of our analysis. Indeed, the proofs of Theorem~\ref{T1} and Theorem~\ref{T2} start with an
application of quantitative ergodicity: In Theorem~\ref{T1}, the LSI~\eqref{eq:LSI} in form of Lemma~\ref{L1} is applied
to $\zeta=\nabla_j\phi_T(0)+\xi_j$, while in Theorem~\ref{T2}, the SG~\eqref{eq:SG} in form of Lemma~\ref{LSGp} is
applied to $\zeta=\phi_T(0)$. Hence we require estimates for
  $\osc_{a(x)}(\nabla_j\phi_T(a;0)+\xi_j)$ and
  $\osc_{a(x)}\phi_T(a;0)$. Following \cite{GO1}, these expressions
are related to the elliptic Green's function:

\begin{lemma}
  \label{L:RP}
  For all $T>0$, $a\in\Omega$, $x\in\Z^d$ and $j=1,\ldots,d$ we have
  \begin{subequations}
    \begin{align}
      \label{eq:osc_phi}
      \osc_{a(x)} \phi_T(a;0) &\leq C(d,\lambda) |\nabla_x G_T(a;0,x)| |\nabla \phi_T(a;x) + \xi|,\\
     \label{T1.1}
      \osc_{a(x)}(\nabla_j \phi_T(a;0) + \xi_j) &\leq C(d,\lambda) |\nabla \nabla G_T(a;0,x)| |\nabla \phi_T(a;x) + \xi|.
    \end{align}
  \end{subequations}
\end{lemma}
\begin{proof}
Let $a\in\Omega$ and $x\in\Z^d$ be fixed. As in the definition of the
oscillation, let $\tilde a\in\Omega$ denote an arbitrary coefficient
field that differs from $a$ only at $x$, i.e.\ $\tilde a(y)=a(y)$ for all $y\neq x$.
We consider the difference $\phi_T(\tilde a;x)-\phi_T(a;x)$. Equation~\eqref{eq:1} yields
\[
\frac{1}{T} ( \phi_T(\tilde a;\cdot) - \phi_T(a;\cdot) ) + \nabla^* \big(\tilde a(\cdot) ( \nabla \phi_T(\tilde a;\cdot) - \nabla \phi_T(a;\cdot) \big) = \nabla^* \big(( a - \tilde a )(\cdot) (\nabla \phi_T(a;\cdot) + \xi )\big)
\]
and consequently the Green's function representation \eqref{P1.3} yields 
\begin{equation}\label{eq:phi_rep}
 \phi_T(\tilde a; y) - \phi_T(a;y) = \nabla_x G_T(\tilde a;y,x) \cdot ( a(x) - \tilde a(x) ) (\nabla \phi_T(a;x) + \xi )
\end{equation}
for all $y\in\Z^d$. In particular, taking the gradient w.~r.~t.~$y_j$ and then setting $y=x$ yields
\[
 |\nabla_j \phi_T(\tilde a; x) - \nabla_j \phi_T(a;x)| \le 2 |\nabla_{j}\nabla G_T(\tilde a;x,x)| |\nabla \phi_T(a;x) + \xi|
\]
since $a, \tilde a\in\Omega$ are uniformly bounded.

In view of \eqref{P1.18}, the mixed derivative of $G_T$ is bounded by $\lambda^{-1}$ and we obtain
\begin{equation}\label{eq:osc_x_phi_x}
 |\nabla_j \phi_T(\tilde a; x) - \nabla_j \phi_T(a;x)| \le 2 \lambda^{-1} |\nabla \phi_T(a;x) + \xi|.
\end{equation}
Exchanging $a$ and $\tilde a$ in \eqref{eq:phi_rep} yields
\begin{equation}\label{eq:phi_rep2}
 \phi_T(a; y) - \phi_T(\tilde a;y) = \nabla_x G_T(a;y,x) \cdot ( \tilde a(x) - a(x) ) (\nabla \phi_T(\tilde a;x) + \xi ).
\end{equation}
We take the absolute value to obtain
\[
 |\phi_T(a;0) - \phi_T(\tilde a;0)| \le 2 |\nabla_x G_T(a;0,x)| |\nabla \phi_T(\tilde a;x) + \xi|.
\]
On the right hand side, we plug in~\eqref{eq:osc_x_phi_x} to obtain
\[
 |\phi_T(a;0) - \phi_T(\tilde a;0)| \le C(d,\lambda) |\nabla_x G_T(a;0,x)| |\nabla \phi_T(a;x) + \xi|.
\]
Since $\tilde a(x)$ was arbitrary, it follows that
\begin{equation*}
 \osc_{a(x)} \phi_T(a;0) \le C(d,\lambda) |\nabla_x G_T(a;0,x)| |\nabla \phi_T(a;x) + \xi|,
\end{equation*}
which is precisely the claimed identity \eqref{eq:osc_phi}.
Taking the gradient with respect to $y_j$ in~\eqref{eq:phi_rep2} yields
\begin{equation*}
 \nabla_j\phi_T(a; y) - \nabla_j\phi_T(\tilde a;y) = \nabla_{y,j}\nabla_x G_T(a;y,x) \cdot ( \tilde a(x) - a(x) ) (\nabla \phi_T(\tilde a;x) + \xi ).
\end{equation*}
We take the absolute value and insert \eqref{eq:osc_x_phi_x} to obtain
\[
 |\nabla_j\phi_T(a; y) - \nabla_j\phi_T(\tilde a;y)| \le C(d,\lambda) |\nabla_{y,j}\nabla_x G_T(a;y,x)|  |\nabla \phi_T(a;x) + \xi |.
\]
and \eqref{T1.1} follows.
\end{proof}

In view of \eqref{eq:osc_phi} and \eqref{T1.1} it is natural that
integrability properties of $G_T$ are required. Next to quantitative
ergodicity, these Green's function
estimates are the second key ingredient in our approach. For Theorem~\ref{T1},
which invokes \eqref{T1.1}, a standard $\ell^2$-energy estimate
for $\nabla\nabla G_T$ suffices, see \eqref{P1.18}. For Theorem~\ref{T2}, which
invokes \eqref{eq:osc_phi}, some more  regularity of the Green's function  is
required. We need a spatially weighted estimate on the gradient
$\nabla G_T$ that is uniform in $a\in\Omega$. To this end, as announced in Section~\ref{SS:sketch2}, we define a weight
\begin{equation}\label{def_weight}
  \omega_q(x):=
  \begin{cases}
    (|x|+1)^{2(q-1)}+T^{1-q}(|x|+1)^{4(q-1)}&\text{for }d=2,\\
    (|x|+1)^{2d(q-1)}&\text{for }d>2,
  \end{cases}
\end{equation}
for every $q\ge 1$ and $T\ge 1$.
\begin{lemma}\label{L2}
  There exists $q_0>1$ only depending on $\lambda$ and $d$  such that
  \begin{align}\label{P1.18}
  \sum_{x\in\Z^d} |\nabla_x\nabla_{y,j} G_T(x,y)|^2 &\le\lambda^{-2},\quad j=1, \ldots, d,\\
  \sum_{x\in\Z^d}|\nabla_x G_T(a;x,0)|^{2q}\omega_q(x)
    &\le C(d,\lambda)
    \begin{cases}
      \log T&\text{for }d=2,\\
      1&\text{for }d>2
    \end{cases}\label{eq:L2}
  \end{align}
  for all $1\leq q \le q_0$.
\end{lemma}

Lemma~\ref{L2} establishes a weighted $\ell^{2q}$-estimate on the
gradient $\nabla G_T$ of the Green's function. For the application,
it is crucial that the integrability exponent $2q$ is larger than $2$. The weight is chosen in such a way that the estimate
remains valid for the constant coefficient Green's function $G_T^0(x):=G_T(\ones;x,0)$ (where we use the symbol $\ones$ to denote the identity in $\mathbb{R}^{d\times d}$) whose gradient behaves as
\begin{equation}\label{eq:decay-const}
  |\nabla G_T^0(x)| \le C(d) (|x|+1)^{1-d}\exp\Big(-c_0\frac{|x|+1}{\sqrt T}\Big)
\end{equation}
for some generic constant $c_0>0$, {which can easily be deduced from the well-known heat kernel bounds on the gradient of the parabolic Green's function (for lack of a better reference, we refer to~\cite[Theorem~1.1]{Delmotte-Deuschel} in the special case of a measure concentrating on $a(x)=\ones$) along the lines of~\cite[Proposition~3.6]{Mourrat}.}
With this bound at hand, the definition of the weight \eqref{def_weight} yields
\begin{equation}\label{eq:const}
  \sum_{x\in\Z^d}|\nabla G_T^0(x)|^{2q}\omega_q(x)
  \le C(d,q)
  \begin{cases}
    \log T&\text{for }d=2,\\
    1&\text{for }d>2
  \end{cases}
\end{equation}
for all $q>1$. Hence, Lemma~\ref{L2} says that the variable-coefficient Green's function exhibits (on a spatially averaged level) the same
decay properties as the constant-coefficient Green's function. In the diagonal, scalar case, Lemma~\ref{L2} is a consequence of \cite[Lemma~2.9]{GO1} and
can also be derived from the weighted estimates on the parabolic Green's
function in \cite[Theorem~3]{GNO1}. Although the arguments in
\cite{GO1,GNO1} rely on scalar techniques, Lemma~\ref{L2} also holds
in the case of systems. Indeed, our proof relies only on techniques which are also available for systems. 
The proof will be split into three parts: First we will provide a simple argument for~\eqref{P1.18} valid in all dimensions. Then we will prove~\eqref{eq:L2} in $d>2$ dimensions. The hardest part is the proof of~\eqref{eq:L2} if $d=2$ since this is the critical dimension.
\begin{proof}[Proof of~\eqref{P1.18}]
An application of $\nabla_{y,j}$ to \eqref{P1.3} yields the
following characterization for $\nabla_{y,j} G_T(a;\cdot,y)$
\begin{equation*}
  \frac{1}{T}\sum_{x\in\Z^d}\nabla_{y,j}
  G_T(a;x,y) w(x)+\sum_{x\in\Z^d}\nabla w(x) \cdot a(x)\nabla_x\nabla_{y,j}
  G_T(a;x,y)=\nabla_j w(y)
\end{equation*}
for all $w\in\ell^2(\Z^d)$.
Taking $w(\cdot):=\nabla_{y,j} G_T(\cdot,y) \in\ell^2(\Z^d)$ yields 
\[
 \frac{1}{T} \sum_{x\in\Z^d} |\nabla_{y,j} G_T(x,y)|^2 + \sum_{x\in\Z^d} \nabla_x \nabla_{y,j} G_T(x,y) \cdot a(x) \nabla_x \nabla_{y,j} G_T(x,y)
= \nabla_{j} \nabla_{j} G_T(y,y),
\]
where $\nabla_{j} \nabla_{j} G_T(y,y)=\nabla_{x,j} \nabla_{y,j} G_T(x,y)\big|_{x=y}$.
The first term on the l.~h.~s.~is positive and ellipticity yields
\[
 \lambda \sum_{x\in\Z^d} |\nabla_x \nabla_{y,j} G_T(x,y)|^2 \le |\nabla_{j} \nabla_{j} G_T(y,y)| \le \bigg( \sum_{x\in\Z^d} |\nabla_x \nabla_{y,j} G_T(x,y)|^2 \bigg)^\frac{1}{2}.
\]
Thus \eqref{P1.18} follows.
\end{proof}
\begin{proof}[Proof of~\eqref{eq:L2} in $d>2$ dimensions]
{\bf Step 1}. A priori estimate:
We prove
\begin{equation}\label{apriori_L2}
 |G_T(0,0)| + \sum_x |\nabla G_T(x,0)|^2 \le C(d,\lambda).
\end{equation}
The weak form of \eqref{P1.3} with $\zeta = G_T(\cdot,0)$ and ellipticity immediately yield
\[
 0 \le \lambda \sum_x |\nabla G_T(x,0)|^2 \le G_T(0,0),
\]
in particular $G_T(0,0) \ge 0$.
Now a Sobolev embedding in $d>2$ with constant $C(d)$ yields
\begin{align*}
 |G_T(0,0)| &\le \bigg(\sum_x |G_T(x,0)|^{\frac{2d}{d-2}}\bigg)^{\frac{d-2}{2d}}\\ &\le C(d) \bigg( \sum_x |\nabla G_T(x,0)|^2 \bigg)^{\frac{1}{2}} \le C(d,\lambda) |G_T(0,0)|^{\frac{1}{2}}.
\end{align*}
The Sobolev embedding is readily obtained from its continuum version
on $\R^d$ via a linear interpolation function on a triangulation subordinate to the lattice $\Z^d$. Hence $|G_T(0,0)|\le C(d,\lambda)$ and \eqref{apriori_L2} follows.

\medskip

{\bf Step 2}. A bound involving weights: In this step we show that there exists $\alpha_0(d) > 0$ such that
\begin{equation}\label{weight_hardy}
 \sum_x ( |x| + 1 )^{2\alpha-2} |G_T(x,0)|^2 \le C(d) \sum_x (|x|+1)^{2\alpha} |\nabla G_T(x,0)|^2
\end{equation}
for all $0 < \alpha \le \alpha_0$. (Note that both sides are well-defined for $G_T$.)
We start by recalling Hardy's inequality in $\R^d$ if $d>2$:
\[
 \int_{\R^d}\frac{|f|^2}{|x|^2} \;dx \le \Big(\frac{2}{d-2}\Big)^2 \int_{\R^d} |\nabla f|^2 \;dx
\]
for all $f\in H^1(\R^d)$. A discrete counterpart can be derived by interpolation w.~r.~t.~a triangulation subordinate to the lattice and yields
\begin{equation}\label{hardy}
 \sum_x ( |x| + 1 )^{2\alpha-2} |G_T(x,0)|^2  \le C(d) \sum_x \big|\nabla( ( |x| + 1 )^{\alpha} G_T(x,0) ) \big|^2.
\end{equation}
The discrete Leibniz rule $\nabla_i (fg)(x) = f(x+e_i)\nabla_i g(x) + g(x)\nabla_i f(x)$ 
yields
\[
\nabla_i(( |x| + 1 )^{\alpha} G_T(x,0) ) = ( |x+e_i| + 1 )^{\alpha} \nabla_i G_T(x,0) + G_T(x,0) \nabla_i( |x| + 1 )^{\alpha}.
\]
By the mean value theorem we obtain the simple inequality $|a^\alpha - b^\alpha| \le \alpha (a^{\alpha-1} + b^{\alpha-1}) |a-b|$ for all $a,b\ge 0$ {and we trivially have that
\[
 \frac{1}{2} (|x|+1) \le |x+e|+1 \le 2(|x|+1).
\]
The choice $a=|x+e|+1$ and $b=|x|+1$ thus yields
\[
 \nabla_i( |x| + 1 )^{\alpha} \le 3 \alpha (|x|+1)^{\alpha-1}
\]
}for all $0\le\alpha\le1$. Summation over $i=1,\ldots,d$ and the discrete Leibniz rule above consequently yield
\[
\big| \nabla \big( ( |x| + 1 )^{\alpha} G_T(x,0) \big) \big|^2 \le C(d) \Big( ( |x| + 1 )^{2\alpha} |\nabla G_T(x,0)|^2 + \alpha ( |x| + 1 )^{2\alpha-2} |G_T(x,0)|^2 \Big)
\]
for any $0\le\alpha\le1$. We substitute this estimate in Hardy's inequality \eqref{hardy} and take $\alpha=\alpha_0(d)$ small enough to absorb the last term into the l.~h.~s.~to obtain~\eqref{weight_hardy}, i.e.\
\[
 \sum_x (|x|+1)^{2\alpha_0-2} |G_T(x,0)|^2  \le C(d) \sum_x ( |x| + 1 )^{2\alpha_0} |\nabla G_T(x,0)|^2.
\]

\medskip

{\bf Step 3}. Improvement of Step~1 to include weights: Now we deduce
the existence of $\alpha_0 = \alpha_0(d,\lambda) > 0$ (smaller
  than $d$ and possibly smaller than $\alpha_0(d)$ from Step~2) such that
\begin{equation}\label{apriori_L2_alpha}
 \sum_x \big( |x| + 1 \big)^{2\alpha_0} |\nabla G_T(x,0)|^2 \le C(d,\lambda).
\end{equation}
To this end, we set $w(x) = (|x|+1)^{2\alpha} G_T(x,0)$ and
note that
\begin{equation*}
  \nabla_iw(x)=(|x|+1)^{2\alpha}\nabla_iG_T(x,0)+\nabla_i\Big((|x+e_i|+1)^{2\alpha}\Big) G_T(x+e_i,0).
\end{equation*}
Hence, \eqref{P1.3} yields (for $y=0$):
\begin{multline}\label{Green_weights_eq}
 \frac{1}{T} \sum_x (|x|+1)^{2\alpha} |G_T(x,0)|^2 + 
 \sum_x\sum_{i,j=1}^d G_T(x+e_i,0) \nabla_i \big((|x+e_i|+1)^{2\alpha}\big) \cdot a_{ij}(x) \nabla_j G_T(x,0)\\
 + \sum_x( |x| + 1 )^{2\alpha} \nabla G_T(x,0) \cdot a(x) \nabla G_T(x,0) = G_T(0,0).
\end{multline}
As in Step~2, we have that
\[
\big|\nabla_i \big((|x|+1)^{2\alpha}\big)\big| \le 4 \alpha (|x|+1)^{\alpha-1}(|x+e_i|+1)^\alpha.
\]
for all $0 \le \alpha \le 1$ and $i=1,\ldots,d$. Thus~\eqref{Green_weights_eq}, ellipticity, and H\"older's inequality yield 
\begin{multline*}
 \lambda\sum_x \big( |x| + 1 \big)^{2\alpha} |\nabla G_T(x,0)|^2  \le |G_T(0,0)|+\\C(d) \alpha \bigg(\sum_x |G_T(x,0)|^2 (|x|+1)^{2\alpha-2} \bigg)^\frac{1}{2}
 \bigg(\sum_x |\nabla G_T(x,0)|^2 (|x|+1)^{2\alpha} \bigg)^\frac{1}{2}.
\end{multline*}
We apply the result of Step 2 with $\alpha \le \alpha_0(d)$ and then possibly decrease $\alpha$ further to absorb the second term on the r.~h.~s. This is possible for $\alpha\le\alpha_0(d,\lambda)$ for some $\alpha_0(d,\lambda)>0$. By Step 1, we conclude  \eqref{apriori_L2_alpha}. By the discrete $\ell^{2q}-\ell^2$-inequality $\|f\|_{\ell^{2q}(\Z^d)} \le \|f\|_{\ell^{2}(\Z^d)}$, it follows that
\[
 \sum_x \big( |x| + 1 \big)^{2q\alpha_0} |\nabla G_T(x,0)|^{2q} \le C(d,\lambda)
\]
for all $q>1$. {Hence Lemma~\ref{L2} holds for $d>2$ with $\omega_q$ defined in~\eqref{def_weight} as long as $2d(q-1) \le 2q\alpha_0$, i.e.\ we may take $q_0 = \frac{d}{d-\alpha_0}$.}
\end{proof}
\begin{proof}[Proof of~\eqref{eq:L2} in $d=2$ dimensions]
Let us remark that the following proof is valid in all dimensions $d\ge2$. However, if $d>2$, we have the simpler proof above.

Fix $T>0$ and $a\in\Omega$. For convenience, we set
\begin{equation}\label{eq:5}
  G(x):=G_T(a;x,0)\qquad\text{and}\qquad G^0(x):=G_{\frac{T}{\lambda}}(\ones;x,0),
\end{equation}
where $\ones$ denotes the identity in $\R^{d\times d}$ and $\lambda$ denotes the constant of
ellipticity from Assumption~\ref{ass:ell}.
We first introduce some notation. For $1\leq q<\infty$ and $\gamma>0$, we denote by  $\ell^{q}_\gamma$  the space of vector fields $g:\Z^d\to\R^d$
with 
\begin{eqnarray*}
  \|g\|_{\ell^q_\gamma}&:=&\left(\sum_{x\in\Z^d}|g(x)|^q(|x|+1)^{\gamma}\right)^{\frac{1}{q}}<\infty.
\end{eqnarray*}
Likewise we denote by $\ell^{2q}_{{\omega_q}}$ the space of vector fields
with 
\begin{equation*}
  \|g\|_{\ell^{2q}_{\omega_q}}:=\left(\sum_{x\in\Z^d}|g(x)|^{2q}{\omega_q}(x)\right)^{\frac{1}{2q}}<\infty,
\end{equation*}
with ${\omega_q}$ defined by \eqref{def_weight}. We write $\|\mathcal
H\|_{B(X)}$ for the operator norm of a linear operator $\mathcal H: X\to X$ defined on a normed space $X$. 
\smallskip

{\bf Step 1}. Helmholtz decomposition:
We claim that the gradients of the variable coefficient Green's function
$G$ and of the constant coefficient Green's function $G^0$ from \eqref{eq:5} are
related by
\begin{equation}\label{helmholtz}
  (\id + \mathcal{H} \overline a) \nabla G = \lambda \nabla G^0
\end{equation}
where $\overline a = \lambda a - \ones$, $\mathcal H:=\nabla\mathcal
L^{-1} \nabla^*$  denotes the
modified Helmholtz projection,  $\mathcal L:= \frac{\lambda}{T} +
\nabla^*\nabla$, and $\id$ denotes the identity operator. Here and in the following, we tacitly identify
$\overline a$ with the multiplication operator that maps the vector
field $g:\Z^d\to\R^{d}$ to the vector field $(\overline a
g)(x):=\overline a(x)g(x)$. Moreover, since $G$ is integrable {in the sense of Lemma~\ref{L:Gint}, the operators} $\mathcal L^{-1}$,
and thus $\mathcal H$ and $(\id+\mathcal{H}\overline a)$ are bounded
linear operators on $\ell^2(\Z^d)$ (resp. $\ell^2(\Z^d,\R^d)$) and the
weighted spaces discussed in Step~2 below.

Identity \eqref{helmholtz} may be seen  by appealing to \eqref{T1.2} satisfied by $G$ and the
equation $\mathcal L G^0=\delta$ satisfied by $G^0$:
	\begin{align*}
	 (\id + \mathcal{H} \overline a) \nabla G
	 &=
	 \nabla G+\lambda \nabla\mathcal L^{-1} \nabla^* a\nabla G- \nabla\mathcal L^{-1} \nabla^* \nabla G\\
	 &=
	 \nabla G+\lambda \nabla\mathcal L^{-1} \left(\delta-\frac1TG\right)- \nabla\mathcal L^{-1}\left(\mathcal
L-\frac{\lambda}{T}\right) G\\
	 &=
	\lambda \nabla\mathcal L^{-1}\delta=\lambda \nabla G^0.
	\end{align*}
\medskip

{\bf Step 2.} Invertibility of $(\id+\mathcal H\overline a)$ in a weighted space:
In this step, we prove that there exists $q_0=q_0(d,\lambda)>1$ such
that the operator $(\id+\mathcal H\overline a):\ell^{2q}_{\omega_q}\to \ell^{2q}_{\omega_q}$ is invertible and
\begin{equation}\label{inv_Helmholtz}
  \|(\id+\mathcal H\overline
  a)\|_{B(\ell^{2q}_{\omega_q})}\leq C(d,\lambda)
\end{equation}
for all $1\leq q\leq q_0$
We split the proof into several sub-steps.
\medskip

{\it Step 2a.} Reduction to an estimate for $\mathcal H$:
We claim that it suffices to prove the following statement. There exists
$q_0=q_0(\lambda)>1$ such that
\begin{equation}\label{eq:8}
  \max \left\{ \| \mathcal{H} \|_{B(\ell^{2q}_{2q-2})}, \| \mathcal{H}
    \|_{B(\ell^{2q}_{4q-4})} \right\}
  \leq \frac{2-\lambda}{2(1-\lambda)}
\end{equation}
for all $1\leq q\leq q_0$.

Our argument is as follows: We only need to show that \eqref{eq:8}
implies that 
\begin{equation}
  \label{eq:10}
  \|\mathcal H\overline a\|_{B(\ell^{2q}_{\omega_q})}\leq \frac{2-\lambda}{2},
\end{equation}
since then $(\id+\mathcal H\overline a)$ can be inverted by a
Neumann-series. Since the $\|\cdot\|_{B(\ell^{2q}_{\omega_q})}$-norm is
submultiplicative, inequality \eqref{eq:10} follows from 
\begin{equation}\label{eq:9}
  \| \mathcal{H} \|_{B(\ell^{2q}_{\omega_q})}\leq
  \frac{2-\lambda}{2(1-\lambda)}\qquad\text{and}\qquad
  \|\overline a\|_{B(\ell^{2q}_{\omega_q})}\leq 1-\lambda.
\end{equation}
We start with the argument for the second inequality in \eqref{eq:9}. Thanks to  \eqref{ass:ell}, we have for all $a_0\in\Omega_0$
and $v\in\R^d$:
\begin{eqnarray*}
  |(\lambda a_0-\ones)v|^2&=&v\cdot((\lambda a_0-\ones)^t(\lambda a_0-\ones))v\\
  &=&\lambda^2|a_0v|^2-2v\cdot\frac{a_0+a_0^t}{2}v+|v|^2\ =\
  \lambda^2|a_0v|^2-2v\cdot a_0v+|v|^2\\
  &\stackrel{\eqref{ass:ell}}{\leq}&\lambda^2|v|^2-2\lambda|v|^2+|v|^2=(1-\lambda)^2|v|^2,
\end{eqnarray*}
{which shows~\eqref{eq:9} by definition of the (spectral) operator norm.}

Regarding the first inequality in \eqref{eq:9}, we note that $\|\cdot\|_{\ell^{2q}_{\omega_q}}^{2q} =
\|\cdot\|_{\ell^{2q}_{2q-2}}^{2q} + {T}^{1-q}
\|\cdot\|_{\ell^{2q}_{4q-4}}^{2q}$, as can been seen by recalling
definition \eqref{def_weight}. Hence,
\begin{align*}
  \| \mathcal{H} \|_{B(\ell^{2q}_{\omega_q})}^{2q}
  &=
  \sup_{\|g\|_{\ell^{2q}_{\omega_q}}\le1}\left( \|\mathcal Hg\|_{\ell^{2q}_{2q-2}}^{2q} + {T}^{1-q} \|\mathcal Hg\|_{\ell^{2q}_{4q-4}}^{2q}\right)\\
  &\le
  \max \left\{ \| \mathcal{H} \|_{B(\ell^{2q}_{2q-2})}^{2q}, \| \mathcal{H}
    \|_{B(\ell^{2q}_{4q-4})}^{2q} \right\}\sup_{\|g\|_{\ell^{2q}_{\omega_q}}\le1}\left( \|g\|_{\ell^{2q}_{2q-2}}^{2q} + {T}^{1-q}
    \|g\|_{\ell^{2q}_{4q-4}}^{2q}\right)\\
  &=
  \max \left\{ \| \mathcal{H} \|_{B(\ell^{2q}_{2q-2})}^{2q}, \| \mathcal{H}
    \|_{B(\ell^{2q}_{4q-4})}^{2q} \right\}
  \stackrel{\eqref{eq:8}}{<}\Big(\frac{2-\lambda}{2(1-\lambda)}\Big)^{2q},
\end{align*}
and \eqref{eq:9} follows.
\smallskip

{\it Step 2b.} Proof of \eqref{eq:8}:
A standard energy estimate yields
\begin{equation}\label{eq:12}
  \|\mathcal H\|_{B(\ell^2(\R^d,\Z^d))}\leq 1.
\end{equation}
{Indeed, given $g\in[\ell^2(\Z^d)]^d$, we have that $\mathcal{H} g = \nabla u$ where $u$ solves $\frac{\lambda}{T} u + \nabla^* \nabla u = \nabla^* g$. Testing with $u$ yields $\|\nabla u\|_{\ell^2(\Z^d)} \le \| g \|_{\ell^2(\Z^d)}$ which is just another way of writing~\eqref{eq:12}.}
In the following we prove the desired inequality \eqref{eq:8} by complex interpolation of
$B(\ell^2(\R^d,\Z^d))=B(\ell^2_0)$ with $B(\ell^{p}_\gamma)$ for
suitable $p$ and $\gamma$. In Proposition \ref{prop:weighted-dics-cz}
below (in Section~\ref{S:CZ}) we prove a Calder\'{o}n-Zygmund-type estimate
 for $\mathcal H$ in weighted spaces
and obtain
\begin{equation}
\| \mathcal{H} \|_{B(\ell^{p}_\gamma)} < \infty\quad\text{for all }
 2\le p\le \infty\text{ and }0 \le \gamma < \min\{2(p-1),{\textstyle\frac{1}{2}}\}.\label{eq:4}
\end{equation}
Fix such $p$ and $\gamma$ and $0 < \theta < 1$. A theorem due to Stein and Weiss \cite[Theorem 5.5.1]{Bergh-Lofstrom-76} that also holds in the discrete setting yields
\begin{equation}\label{stein_weiss}
 \| \mathcal{H} \|_{B(\ell^{p}_{\gamma'})} \le  \| \mathcal{H} \|_{B(\ell^{p}_\gamma)}^{1-\theta}  \| \mathcal{H} \|_{B(\ell^{p})}^\theta,\qquad\text{if $\gamma'= (1-\theta)\gamma$.}
\end{equation}
Likewise the classical Riesz-Thorin theorem \cite[Theorem 1.1.1]{Bergh-Lofstrom-76} yields
\begin{equation}\label{riesz_thorin}
 \| \mathcal{H} \|_{B(\ell^{p'}_{\gamma})} \le  \| \mathcal{H} \|_{B(\ell^{p}_\gamma)}^{1-\theta}  \| \mathcal{H} \|_{B(\ell^{2}_\gamma)}^\theta,\qquad\text{if $\frac{1}{p'} = \frac{1-\theta}{p} + \frac{\theta}{2}$.}
\end{equation}
In particular, the map $(p,\gamma)\mapsto\|\mathcal{H}\|_{\mathcal
  B(\ell_{\gamma}^{p})}$ is continuous at $(2,0)$: Given $\epsilon >
0$, we use \eqref{riesz_thorin} with $\gamma=0$ to find $p'>2$ such
that $\| \mathcal{H} \|_{\mathcal B(\ell^{p'})} \le 1+
\frac{\epsilon}{2}$. Then we apply \eqref{stein_weiss} to find
$\gamma' > 0$ such that $\max\{\| \mathcal{H} \|_{\mathcal
  B(\ell^{2}_{\gamma'})},\| \mathcal{H} \|_{\mathcal
  B(\ell^{p'}_{\gamma'})}\} \le 1+\epsilon$. Hence, we have $\| \mathcal{H} \|_{\mathcal
  B(\ell^{p}_{\gamma})} \le 1+\epsilon$ for the corner points
$(p,\gamma)$ of the square $[2,p']\times[0,\gamma']$. By~\eqref{riesz_thorin} resp.~\eqref{stein_weiss}, we may always decrease
either $p'$ resp.\ $\gamma'$ while achieving the same bound. Consequently we have that $\| \mathcal{H} \|_{\mathcal
  B(\ell^{p}_{\gamma})} \le 1+\epsilon$ for all $(p,\gamma)\in
[2,p']\times[0,\gamma']$. In particular, {letting $\epsilon=\frac{2-\lambda}{2(1-\lambda)}-1>0$}, there exists $q_0>1$ such
that $\| \mathcal{H} \|_{\mathcal B(\ell^{2q_0}_{2q_0-2})}\leq\frac{2-\lambda}{2(1-\lambda)}$ and the same
bound for $\| \mathcal{H} \|_{\mathcal B(\ell^{2q_0}_{4q_0-4})}$. By
monotonicity in the exponent, estimate \eqref{eq:8} follows for all
$1\leq q\leq q_0$. This completes the argument of Step~2.
\medskip


{\bf Step 3}. In this last step, we fix $d=2$ and derive the bound
\begin{equation}\label{nabla_G_bound}
 \sum_x |\nabla G(x)|^{2q} {\omega_q}(x)  = \| \nabla G \|_{\ell^{2q}_{\omega_q}}^{2q} \le C(\lambda,q) \log T
\end{equation}
for $q$ and ${\omega_q}$ as in Step 2.
The relation \eqref{helmholtz} and the estimate \eqref{inv_Helmholtz} yield
\[
 \| \nabla G \|_{\ell^{2q}_{\omega_q}} \le C(\lambda) \| \nabla G^0 \|_{\ell^{2q}_{\omega_q}}
\]
so that it is enough to consider the constant coefficient Green's
function whose behaviour is well-known and is given by 
(cf.~\eqref{eq:decay-const})
\begin{equation*}
  |\nabla G^0(x)| \le C (|x|+1)^{-1}\exp\bigg(-\frac{{\sqrt{\lambda}}|x|}{C\sqrt{T}}\bigg),
\end{equation*}
where $C$ is a universal constant.
Hence by
splitting $ \| \nabla G^0 \|_{\ell^{2q}_{\omega_q}}^{2q}$ into its contributions coming from $|x|\le \sqrt{T}$ and $|x| > \sqrt{T}$ and using
the definition of the weight ${\omega_q}$,  we have
	\begin{align*}
	\| \nabla G^0 \|_{\ell^{2q}_{\omega_q}}^{2q} &= \sum_x |\nabla G^0(x)|^{2q}\left((|x|+1)^{2q-2}+T^{1-q}(|x|+1)^{4q-4}\right) \\
	&\le C \sum_x (|x|+1)^{-2q} e^{-\frac{2q{\sqrt{\lambda}}|x|}{C\sqrt{T}}} \left((|x|+1)^{2q-2}+T^{1-q}(|x|+1)^{4q-4}\right) \\
	& \le C(\lambda ,q) \sum_{|x|\le\sqrt{T}} (|x|+1)^{-2} +C(\lambda ,q)\sum_{|x|>\sqrt{T}} {T}^{1-q}(|x|+1)^{2q-4} e^{-\frac{2q{\sqrt{\lambda}}|x|}{C\sqrt{T}}}\\
	&\le C(\lambda ,q) \log T+ C(\lambda ,q) \sum_{|x|>\sqrt{T}} {T}^{-1} \big({\textstyle\frac{|x|}{\sqrt{T}}}\big)^{2q-4} e^{-\frac{2{\sqrt{\lambda}}|x|}{C\sqrt{T}}}\\
	&\le C(\lambda ,q) \log T+C(\lambda ,q),
	\end{align*}
where we have used that $q > 1$.
\end{proof}

\subsection{Logarithmic Sobolev inequality and spectral gap
  revisited}\label{SS:E}

The LSI only enters the proof of Theorem~\ref{T1} in form of the following lemma borrowed from \cite{MO1}.
\begin{lemma}[Lemma~4 in \cite{MO1}]\label{L1}
  Let $\expec{\cdot}$ statisfy LSI~\eqref{eq:LSI} with constant $\rho>0$. Then we have that
\begin{equation}\label{T1.0}
  \langle|\zeta|^{2p}\rangle^\frac{1}{2p}\le C(\delta,p,\rho)\langle|\zeta|^2\rangle^{\frac{1}{2}}
  +\delta\Big\langle\bigg(\sum_{x\in\Z^d} \Big( \osc_{a(x)} \zeta\Big)^2 \bigg)^p\Big\rangle^\frac{1}{2p}
\end{equation}
for any $\delta>0$, $1\le p<\infty$ and $\zeta\in C_b(\Omega)$.
\end{lemma}
This inequality expresses a reverse Jensen inequality and allows to bound
high moments of $\zeta$ to the expense of some control on the
oscillations of $\zeta$. {The difference to SG lies in the fact that the improved integrability properties of LSI allow us to choose $\delta>0$ arbitrarily small.} In the proof of Theorem~\ref{T1}, we will apply \eqref{T1.0} to the random
variables $\zeta=\nabla_i \phi_T(0) + \xi_i$ for $i=1,\ldots,d$. The second moment of $\nabla_i \phi_T(0) + \xi_i$ will be controlled below, whereas the oscillation was already estimated Lemma~\ref{L:RP} and involves the second mixed derivatives of $G_T$.
\medskip

In the proof of Theorem~\ref{T2}, we just require the weaker statement of SG. To be precise, we will use an $L^{2p}_{\langle\cdot\rangle}$-version of SG which is the content of the following lemma. 
\begin{lemma}[cf.~Lemma~2 in \cite{GNO1}]
  \label{LSGp}
  Let $\expec{\cdot}$ statisfy SG~\eqref{eq:SG} with constant $\rho>0$. Then for
  arbitrary $1\le p<\infty$ and $\zeta\in
    C_b(\Omega)$ it holds that
  \begin{equation}\label{eq:SGp}
    \expec{|\zeta - \langle \zeta \rangle|^{2p}}\le C(p,\rho) \Big\langle\bigg(\sum_{x\in\Z^d} \Big( \osc_{a(x)} \zeta\Big)^2 \bigg)^p \Big\rangle.
  \end{equation}
\end{lemma}
The proof is a combination of the proofs of~\cite[Lemma~2]{GNO1} and~\cite[Lemma~4]{MO1}. We present it here for the convenience of the reader.
\begin{proof}
Without loss of generality assume that $\zeta\in C_b(\Omega)$ satisfies $\expec{\zeta}=0$. The triangle inequality and SG~\eqref{eq:SG} yield
\begin{align*}
 \expec{|\zeta|^{2p}} &\le 2\expec{ \big(|\zeta|^{p} - \expec{|\zeta|^p}\big)^2 } + 2\expec{|\zeta|^p}^2\\
 &\le \frac{2}{\rho}\expec{\sum_{x}\Big(\osc_{a(x)}|\zeta|^p\Big)^2}+2\expec{|\zeta|^{2p}}^{\frac{p-2}{p-1}}\expec{|\zeta|^2}^{\frac{p}{p-1}}.
\end{align*}
By Young's inequality, we may absorb $\langle |\zeta|^{2p} \rangle$ on the l.~h.~s.\ and we obtain that
\begin{equation}\label{eq:SG-p}
 \expec{|\zeta|^{2p}} \le \frac{4}{\rho}\expec{\sum_{x}\Big(\osc_{a(x)}|\zeta|^p\Big)^2}+C(p)\expec{|\zeta|^2}^{p}.
\end{equation}
We insert SG~\eqref{eq:SG}, note $\langle\zeta\rangle=0$ and apply Jensen's inequality to obtain that
\begin{equation}\label{eq:SG-1^p}
 \expec{|\zeta|^2}^{p} \le \rho^{-p} \expec{\sum_{x}\Big(\osc_{a(x)}\zeta\Big)^2}^p \le \rho^{-p} \expec{\bigg(\sum_{x}\Big(\osc_{a(x)}\zeta\Big)^2\bigg)^p}.
\end{equation}
In order to deal with the first term in~\eqref{eq:SG-p}, we note that the elementary inequality $|t^p - s^p| \le C(p) (t^{p-1} |t-s| + |t-s|^p)$ for all $t,s\ge0$ yields for every two coefficient fields $a,\tilde a \in \Omega$:
\[
\Big||\zeta(a)|^p - |\zeta(\tilde a)|^p\Big| \le C(p) \big(|\zeta(a)|^{p-1} |\zeta(a)-\zeta(\tilde a)| + |\zeta(a)-\zeta(\tilde a)|^p\big),
\]
where we have in addition used the triangle inequality in form of $\Big||\zeta(a)|-|\zeta(\tilde a)|\Big| \le |\zeta(a)-\zeta(\tilde a)|$. Letting $\tilde a \in \Omega$ run over the coefficient fields that coincide with $a$ outside of $x\in\Z^d$ yields
\[
 \osc_{a(x)}|\zeta|^p \le C(p) \bigg(|\zeta|^{p-1} \osc_{a(x)} \zeta + \Big(\osc_{a(x)} \zeta\Big)^p\bigg)
\]
Consequently we obtain 
\begin{align*}
 \expec{\sum_{x}\Big(\osc_{a(x)}|\zeta|^p\Big)^2} &\le C(p) \Bigg(\expec{|\zeta|^{2(p-1)}\sum_{x}\Big(\osc_{a(x)}\zeta\Big)^2} + C(p) \expec{\sum_{x}\Big(\osc_{a(x)} \zeta\Big)^{2p}}\Bigg)\\
 &\le C(p) \Bigg(\expec{|\zeta|^{2p}}^{\frac{p-1}{p}}\expec{\bigg(\sum_{x}\Big(\osc_{a(x)}\zeta\Big)^2\bigg)^p}^{\frac{1}{p}} + \expec{\bigg(\sum_{x}\Big(\osc_{a(x)}\zeta\Big)^2\bigg)^p}\Bigg)
\end{align*}
by H\"older's inequality and the discrete $\ell^2\subset\ell^{2p}$-inequality. Inserting this estimate as well as~\eqref{eq:SG-1^p} into~\eqref{eq:SG-p} yields
\[
 \expec{|\zeta|^{2p}} \le C(p,\rho) \Bigg(\expec{|\zeta|^{2p}}^{\frac{p-1}{p}}\expec{\bigg(\sum_{x}\Big(\osc_{a(x)}\zeta\Big)^2\bigg)^p}^{\frac{1}{p}} + \expec{\bigg(\sum_{x}\Big(\osc_{a(x)}\zeta\Big)^2\bigg)^p}\Bigg).
\]
Again, we may absorb the factor $\langle |\zeta|^{2p} \rangle$ on the l.~h.~s.\ using Young's inequality and thus conclude the proof of Lemma~\ref{LSGp}.
\end{proof}

\subsection{Proof of Theorem~\ref{T1}}\label{SS:T1}
{\bf Step 1}. We claim the following energy estimate:
\begin{equation}\label{T1.3}
  \expec{|\nabla\phi_T(0)+\xi|^2}\le C(\lambda) |\xi|^2.
\end{equation}
To see this, we multiply~\eqref{eq:1} with $\phi_T(0)$ and take the
expectation:
\[
\frac{1}{T}\expec{|\phi_T(0)|^2}+\expec{\phi_T(0)\nabla^*(a\nabla\phi_T)(0)}=-\expec{\phi_T(0)\nabla^*(a\xi)(0)}.
\]
Thanks to the stationarity of $\expec{\cdot}$ and the stationarity of
$\phi_T$, cf.~\eqref{LMC:1}, we have that
\begin{align*}
 \langle \phi_T(0) \nabla^* w(x) \rangle &= \sum_{i=1}^d \langle \phi_T(0) \big( w_i(x-e_i) - w_i(x) \big) \rangle\\ &= \sum_{i=1}^d \langle \big( \phi_T(e_i) - \phi_T(0) \big) w_i(x) \rangle = \langle \nabla \phi_T(0) \cdot w(x) \rangle 
\end{align*}
for all stationary vector fields $w : \Z^d \to \R^d$. This integration by parts property then yields
\[
\frac{1}{T}\expec{|\phi_T(0)|^2}+\langle \nabla\phi_T(0) \cdot a(0) \nabla\phi_T(0) \rangle =- \langle \nabla\phi_T(0) \cdot a(0) \xi \rangle.
\]
Since the first term on the left-hand side is non-negative, uniform
ellipticity, cf.~\eqref{ass:ell}, yields
\[
\expec{|\nabla \phi_T(0)|^2} \le  \lambda^{-2} |\xi|^2,
\]
and~\eqref{T1.3} follows from the triangle inequality.

\medskip

  {\bf Step 2}. We claim that
  \begin{equation}\label{P1.17}
    \bigg\langle\bigg(\sum_{x}|\nabla\nabla G_T(0,x)|^2|\nabla\phi_T(x)+\xi|^2\bigg)^p\bigg\rangle \le   
\lambda^{-2p}\langle|\nabla\phi_T(0)+\xi|^{2p}\rangle.
  \end{equation}
 We start by applying H\"older's inequality with exponent $p$ in space:
  \begin{multline*}
    \bigg(\sum_{x}|\nabla\nabla G_T(0,x)|^2|\nabla\phi_T(x)+\xi|^2\bigg)^p\\
    \le \bigg(\sum_{x}|\nabla\nabla G_T(0,x)|^2\bigg)^{p-1}
    \sum_{x}|\nabla\nabla G_T(0,x)|^2|\nabla\phi_T(x)+\xi|^{2p}.
  \end{multline*}
  We now apply $\langle\cdot\rangle$ to obtain
  \begin{align*}
    &\Big\langle\bigg(\sum_{x}|\nabla\nabla G_T(0,x)|^2 |\nabla\phi_T(x)+\xi|^2 \bigg)^p\Big\rangle\\
    &\le
    \bigg(\sup_{a\in\Omega}\sum_{x}|\nabla\nabla G_T(0,x)|^2\bigg)^{p-1}
    \sum_{x}\langle |\nabla\nabla G_T(0,x)|^2 |\nabla\phi_T(x)+\xi|^{2p}\rangle.
  \end{align*}
  At this stage, we appeal to the stationarity of $G_T$, cf.\
  (\ref{P1.15}), the stationarity of $\nabla\phi_T$, cf.~\eqref{LMC:1},
  and the stationarity of $\langle\cdot\rangle$ in
  form of
  \[
    \langle|\nabla\nabla G_T(0,x)|^2|\nabla\phi_T(x)+\xi|^{2p}\rangle
    =\langle|\nabla\nabla G_T(-x,0)|^2|\nabla\phi_T(0)+\xi|^{2p}\rangle,
  \]
which yields
  \begin{align*}
   &\Big\langle\bigg(\sum_{x}|\nabla\nabla G_T(0,x)|^2|\nabla\phi_T(x)+\xi|^2\bigg)^p\Big\rangle\\
   &\le \bigg(\sup_{a\in\Omega}\sum_{x}|\nabla\nabla G_T(0,x)|^2\bigg)^{p-1} \Big\langle \sum_{x} |\nabla\nabla G_T(-x,0)|^2|\nabla\phi_T(0)+\xi|^{2p} \Big\rangle\\
   &\le \bigg(\sup_{a\in\Omega}\sum_{x}|\nabla\nabla G_T(0,x)|^2\bigg)^{p-1} \bigg(\sup_{a\in\Omega}\sum_{x}|\nabla\nabla G_T(x,0)|^2 \bigg) \langle|\nabla\phi_T(0)+\xi|^{2p}\rangle.
  \end{align*}
  We conclude by appealing to symmetry, cf.~\eqref{P1.16}, and \eqref{P1.18}. Note that the transposed coefficient field $a^t$ satisfies $a^t\in\Omega$.

  \medskip

  {\bf Step 3}. Conclusion: The combination of \eqref{P1.17} and \eqref{T1.1} yields
  \begin{equation}\label{P1.35}
    \bigg\langle \bigg(\sum_{x}\Big(\osc_{a(x)}(\nabla_i\phi_T(0)+\xi_i)\Big)^2\bigg)^p\bigg\rangle^\frac{1}{p}
    \le C(d,\lambda) \langle|\nabla\phi_T(0)+\xi|^{2p}\rangle^\frac{1}{p}
  \end{equation}
for $i=1,\ldots,d$. We now appeal to Lemma~\ref{L1} with $\zeta=\nabla_i \phi_T(0) +
  \xi_i$, i.e.\
\[
   \langle|\nabla_i \phi_T(0) + \xi_i|^{2p}\rangle^\frac{1}{2p}\le C(\delta,p,\rho)\langle|\nabla_i \phi_T(0) + \xi_i|^2\rangle^{\frac{1}{2}}
  +\delta\Big\langle\bigg(\sum_{x\in\Z^d} \Big( \osc_{a(x)} (\nabla_i \phi_T(0) + \xi_i) \Big)^2 \bigg)^p\Big\rangle^\frac{1}{2p}.
\]
 On the r.~h.~s.~we insert the estimates (\ref{T1.3}) and (\ref{P1.35}) and
  sum in $i=1,\ldots,d$ to obtain (after redefining $\delta$)
  \begin{equation}\nonumber
    \sum_{i=1}^d\langle|\nabla_i\phi_T(0)+\xi_i|^{2p}\rangle^\frac{1}{2p}\le C(d,\lambda,\delta,p,\rho)|\xi|
    +\delta
    \langle|\nabla\phi_T(0)+\xi|^{2p}\rangle^\frac{1}{2p}.
  \end{equation}
  By the equivalence of finite-dimensional norms, it follows (again, after redefining $\delta$)
  \[
    \langle|\nabla\phi_T(0)+\xi|^{2p}\rangle^\frac{1}{2p}\le C(d,\lambda,\delta,p,\rho)|\xi|
    +\delta
    \langle|\nabla\phi_T(0)+\xi|^{2p}\rangle^\frac{1}{2p}.
  \]
  By choosing $\delta=\frac{1}{2}$, we may absorb the second term on
  the r.~h.~s.~into the l.~h.~s.~which completes the
  proof.\qed

\subsection{Proof of Theorem~\ref{T2}}\label{SS:T2}
As a starting point, we apply SG in its $p$-version Lemma~\ref{LSGp}:
We apply this inequality with $\zeta=\phi_T(0)$.  Since
$\expec{\phi_T(0)}=0$ (as can be seen by taking the expectation of
\eqref{eq:1} and using the stationarity of
  $\expec{\cdot}$ and $\phi_T$), estimate \eqref{eq:SGp} yields
\begin{equation*}
  \langle|\phi_T(0)|^{2p}\rangle\le \frac{1}{\rho}
  \Big\langle\bigg(\sum_{x}\Big( \osc_{a(x)} \phi_T(0) \Big)^2\bigg)^p\Big\rangle.
\end{equation*}
The oscillation estimate~\eqref{eq:osc_phi} yields
\[
 \langle|\phi_T(0)|^{2p}\rangle \le C(d,\lambda,\rho) \Big\langle \bigg( \sum_x |\nabla  G_T(0,x)|^2 |\nabla\phi_T(x)+ \xi|^2 \bigg)^p \Big\rangle.
\]
  With the help of H\"older's inequality we can introduce the weight
  $\omega_q$ from Lemma~\ref{L2} and get for the r.~h.~s.
  \begin{align*}
    &\Big\langle \bigg( \sum_x |\nabla G_T(0,x)|^2
        |\nabla\phi_T(x) + \xi|^2 \bigg)^p\Big\rangle\\
    &\le \Big\langle \bigg(\sum_x |\nabla G_T(0,x)|^{2q} \omega_q(x)\bigg)^{p-1} \sum_x
      |\nabla\phi_T(x)+ \xi|^{2p} \omega_q(x)^{-\frac{1}{q-1}} \Big\rangle\\
    &\le \bigg(\sup_{a\in\Omega} \sum_x |\nabla G_T(0,x)|^{2q}\omega_q(x) \bigg)^{p-1} \sum_x \langle|\nabla\phi_T(x) + \xi|^{2p}\rangle \omega_q^{-\frac{1}{q-1}}(x).
  \end{align*}
Due to the stationarity of $\nabla\phi_T+\xi$ and Lemma~\ref{L2} we obtain
  \begin{equation*}
 \langle |\phi_T(0)|^{2p} \rangle    \le C(d,\lambda,p)
    \begin{cases}
      (\log T)^{p-1} \big\langle
|\nabla\phi_T+\xi)(0)|^{2p}\big\rangle \sum_x \omega_q(x)^{-\frac{1}{q-1}}&\text{for }d=2,\\
      \big\langle
|\nabla\phi_T+\xi)(0)|^{2p}\big\rangle \sum_x \omega_q(x)^{-\frac{1}{q-1}}&\text{for }d>2.
    \end{cases}
  \end{equation*}
To conclude in the case of $d=2$, we simply insert \eqref{def_weight}
to bound (for $T\geq 2$)
\begin{align*}
 \sum_x \omega_q(x)^{-\frac{1}{q-1}} &\le C(p) \Bigg(\sum_{|x|\le\sqrt{T}} (|x|+1)^{2} +
\sum_{|x|>\sqrt{T}} T(|x|+1)^{-4} \Bigg)\\
& \le C(p) \Big( \log T + \frac{1}{\sqrt{T}} \Big) \le C(p) \log T.
\end{align*}
If $d>2$, we find that
\[
 \sum_x \omega_q(x)^{-\frac{1}{q-1}} = \sum_x (|x|+1)^{-2d} \le C(d),
\]
which finishes the proof.
\qed

\section{A weighted Calder\'{o}n-Zygmund estimate}\label{S:CZ}

In this section we present a discrete Calder\'{o}n-Zygmund
estimate on $\ell^p$-spaces with Muckenhoupt weights, which we used in
Step~2b of the proof of estimate~\eqref{eq:L2} in Lemma~\ref{L2} in
the case $d=2$, see \eqref{eq:4}. Although we require the estimate in this paper only in dimension $d=2$, we present it here for any dimension $d\ge
2$ since it may be of independent interest. The proof closely follows~\cite[Lemma~28]{GNO1prep}; the difference lies in the inclusion of weighted spaces which requires a bit more effort.

\begin{proposition}\label{prop:weighted-dics-cz}
Let $T>0$, let $g:\Z^d\to\R^d$ be a compactly supported function and let $u\in\ell^2(\Z^d)$ be the unique solution to
	\begin{equation}\label{eq:cz-1}
	{\frac{1}{T}} u+\nabla^*\nabla u=\nabla^*g\quad\text{on }\Z^d.
	\end{equation}
Then for all $1<p<\infty$ and all $0\le \gamma < \min\{d(p-1),1/2\}$ we
have
\[
\sum_{x\in\Z^d}|\nabla u(x)|^p(|x|+1)^\gamma\le C(d,p,\gamma) \sum_{x\in\Z^d}|g(x)|^p(|x|+1)^\gamma.
\]
\end{proposition}
This proposition is a discrete version of the well-known
\emph{continuum Calder\'{o}n-Zygmund estimate} with
Muckenhoupt weight:
\begin{proposition}[see \cite{Stein-70}]\label{prop:weighted-cont-cz}
Let $T>0$, let $g:\R^d\to\R^d$ be smooth and compactly supported, and
let $u:\R^d\to\R$ be the unique smooth and decaying solution to
\[
 {\frac{1}{T}} u - \Delta u= -\nabla \cdot g\quad\text{on }\R^d.
\]
Then for all $1<p<\infty$ and all $-d<\gamma<d(p-1)$ we have that
\[
 \int_{\R^d} |\nabla u(x)|^p|x|^\gamma  \;dx \le C(d,p,\gamma) \int_{\R^d} |g(x)|^p|x|^\gamma  \;dx.
\]
\end{proposition}

The rest of this section is devoted to the proof of Proposition~\ref{prop:weighted-dics-cz}. To simplify the upcoming argument, fix for the remainder of this
section two indices $j,\ell\in\{1,\dots,d\}$. By  linearity it
suffices to consider instead of \eqref{eq:cz-1} the equation
\begin{equation}\label{eq:cz-2}
  {\frac{1}{T}} u+\nabla^*\nabla u=\nabla^*_\ell g\quad\text{on }\Z^d
\end{equation}
for scalar $g$, and then to prove
\begin{equation}\label{eq:disc-weighted-cz}
  \sum_{x\in\cap\Z^d}|\nabla_j u(x)|^p(1+|x|)^\gamma\le C(d,p,\gamma)  \sum_{x\in\cap\Z^d}|g(x)|^p(1+|x|)^\gamma.
\end{equation}
The discrete estimate \eqref{eq:disc-weighted-cz} will be  obtained from
Proposition~\ref{prop:weighted-cont-cz} by a perturbation argument. 
More precisely, we compare the discrete equation \eqref{eq:cz-2}
and its continuum version in Fourier space. We denote the Fourier
transform on $\R^d$ by
	\[
	(\Fcal g)(\xi)
	=
	(2\pi)^{-d/2}\int_{\R^d} g(x)e^{-i\xi\cdot x}\ dx,\quad \xi\in\R^d,
	\]
        and for functions defined on the discrete lattice $\Z^d$ we define the \emph{discrete Fourier transform} as
	\[
	(\Fcal_{dis}g)(\xi)
	=
	(2\pi)^{-d/2}\sum_{x\in\Z^d}g(x)e^{-i\xi\cdot x},\quad \xi\in\R^d.
	\]
Note that $\Fcal_{dis}F$ is $(-\pi,\pi)^d$-periodic and that we have the inversion formula
\begin{equation}\label{eq:Finv}
(\Fcal^{-1}(\chi\Fcal_{dis}g))(x)=g(x)\qquad\text{for all }x\in\Z^d,
\end{equation}
where $\chi$ denotes the indicator function of the \emph{Brillouin zone}
$(-\pi,\pi)^d$ {which is the unit cell of the Fourier transform on a lattice}.

The Fourier multipliers corresponding to \eqref{eq:cz-2} and its continuum
version are given by
\begin{align*}
  \Mfrak_{T}^{cont}(\xi)
  =
  \frac{\xi_j\xi_\ell}{{\frac{1}{T}}+|\xi|^2},\qquad\qquad \Mfrak_{T}(\xi)
  =
  \frac{(e^{-i\xi_j}-1)(e^{i\xi_\ell}-1)}{{\frac{1}{T}}+\sum_{n=1}^d|e^{i\xi_n}-1|^2}.
\end{align*}
In particular,~\eqref{eq:cz-2} reads in Fourier space as
\begin{equation*}
  \nabla_j u=\mathcal F^{-1}(\chi\Mfrak_{T}\mathcal F_{dis}g)
\end{equation*}
and \eqref{eq:disc-weighted-cz} is equivalent to
\begin{equation}\label{eq:main-ineq}
  \sum_{x\in\Z^d}|(\Fcal^{-1}(\chi\ \Mfrak_{T}\Fcal_{dis}g))(x)|^p\,(|x|+1)^\gamma\le C(d,p,\gamma)\sum_{x\in\Z^d}|g(x)|^p\,(|x|+1)^\gamma.
\end{equation}
Finally, we state two auxiliary results that will be used in the subsequent
argument and which we prove at the end of this section. The first
result shows that the discrete and continuum norms for band-restricted
functions are equivalent. For brevity, we set
\begin{equation}\label{eq:norms}
 \|g\|_{\ell^p_\gamma} = \bigg( \sum_{x\in\Z^d} |g(x)|^p(|x|+1)^\gamma \bigg)^{\frac{1}{p}} \quad\text{and}\quad
 \|g\|_{L^p_\gamma} = \bigg( \int_{\R^d} |g(x)|^p|x|^\gamma  \;dx \bigg)^{\frac{1}{p}}.
\end{equation}
Furthermore, we use the notation
$\|\cdot\|_{\ell^p_\omega}$ (resp. $\|\cdot\|_{L^p_\omega}$), if
$(|x|+1)^\gamma$ (resp. $|x|^\gamma$) is replaced by a general weight
function $\omega$.
\begin{lemma}[Equivalence of discrete and continuous norms]\label{lemma:equiv-norms}
For all $L$ large enough, the $\ell^p_\gamma$-norm and the $L^p_\gamma$-norm are equivalent for functions supported on $[-\frac{1}{L},\frac{1}{L}]^d$ in
Fourier space, i.e.\
\[
 \frac{1}{C(d,p,\gamma)} \|g\|_{L^p_\gamma} \le \|g\|_{\ell^p_\gamma} \le C(d,p,\gamma) \|g\|_{L^p_\gamma}
\]
for all functions $g:=\mathcal{F}^{-1}(F):\R^d\to\C$ with $F$ supported on $[-\frac{1}{L},\frac{1}{L}]^d$ where we let without loss generality $\frac{1}{L} < \pi$.
\end{lemma}
The second result is a generalization of Young's convolution estimate
to weighted spaces.
\begin{lemma}[Young's convolution estimate on weighted spaces]\label{lemma:weighted-young}
Let $\omega:\Z^d\to\R$ satisfy
\begin{equation}\label{eq:23}
  \omega(x)\geq 1\qquad\text{and}\qquad \omega(x)\le \omega(y)\omega(x-y)\qquad\text{for all }x,y\in\Z^d.
\end{equation}
Then the estimate
\begin{equation}\label{eq:weighted-young}
	\|f \ast_{dis} g\|_{\ell_\omega^p}
	\le
	\|f\|_{\ell_\omega^q}\|g\|_{\ell_\omega^r},
	\quad
	1+\frac1p=\frac1q+\frac1r
\end{equation}
holds, where $\ast_{dis}$ denotes the discrete convolution on $\Z^d$:
\begin{equation*}
  (f\ast_{dis}g)(x):=\sum_{y\in\Z^d}f(x-y)g(y).
\end{equation*}
The same estimate
holds in the continuum case (with $\ast_{dis}$ and
$\|\cdot\|_{\ell^p_\omega}$ replaced by the usual
convolution $\ast$ and $\|\cdot\|_{L_\omega^p}^p$, respectively) as long
as $\omega$ satisfies \eqref{eq:23} for all $x,y\in\R^d$.
\end{lemma}

Now, we are ready to start the proof of
Proposition~\ref{prop:weighted-dics-cz} in earnest.

\textbf{Step 1}. Fourier multipliers:
We claim that the invoked Fourier multipliers satisfy
\begin{equation}\label{eq:step1}
 \Mfrak_{T}-\Mfrak_{T}^{cont}=\Mfrak_{T}\Mfrak^*_{T},
\end{equation}
where we define
\begin{equation}\label{eq:m*}
  \Mfrak^*_{T}:= 1 - \frac{1}{h(\xi_j) h(-\xi_\ell)} + \frac{|\xi|^2}{{\frac{1}{T}} + |\xi|^2} \ \sum_{k=1}^d \frac{|\xi_k|^2
(1-|h(\xi_k)|^2)}{|\xi|^2 h(\xi_j) h(-\xi_\ell)}
\end{equation}
and
\begin{equation}\label{eq:hz}
 h(z):=\begin{cases}\frac{e^{iz}-1}{iz}&0\neq z\in\C,\\1&z=0.\end{cases}
\end{equation}
Indeed, \eqref{eq:step1} is true for $\xi=0$. For $\xi\neq 0$ the definition of $h(z)$ yields that
\begin{align*}
  \Mfrak^*_{T} &= 1 - \frac{\mathfrak{M}^{cont}_{T}}{\mathfrak{M}_{T}}\\
 &= 1 - \frac{\xi_j\xi_\ell}{(e^{i\xi_j}-1)(e^{-i\xi_\ell}-1)} - \frac{1}{{\frac{1}{T}}+|\xi|^2} \ \sum_{k=1}^d \frac{\xi_j\xi_\ell({\frac{1}{T}} + |\xi_k|^2 -
({\frac{1}{T}} +
|e^{i\xi_k}-1|^2))}{(e^{i\xi_j}-1)(e^{-i\xi_\ell}-1)}\\
 &= 1 - \frac{1}{h(\xi_j) h(-\xi_\ell)} - \frac{|\xi|^2}{{\frac{1}{T}}+|\xi|^2} \ \sum_{k=1}^d \frac{{\frac{1}{T}} + |\xi_k|^2 - ({\frac{1}{T}} +
|\xi_k|^2|h(\xi_k)|^2)}{|\xi|^2 h(\xi_j) h(-\xi_\ell)}\\
 &= 1 - \frac{1}{h(\xi_j) h(-\xi_\ell)} + \frac{|\xi|^2}{{\frac{1}{T}} + |\xi|^2} \ \sum_{k=1}^d \frac{|\xi_k|^2 (1-|h(\xi_k)|^2)}{|\xi|^2 h(\xi_j)
h(-\xi_\ell)}.
\end{align*}
In order to prove uniformity in $T$ (recall that the assertion of Proposition~\ref{prop:weighted-dics-cz} does not involve $T$), we may split $\mathfrak{M}_{T}^*$ into two terms independent of $T$ and a simple prefactor involving $\frac{1}{T}$:
\begin{equation}
  \Mfrak^*_{T}
  =
  \Mfrak^*_1+\frac{|\xi|^2}{{\frac{1}{T}}+|\xi|^2} \Mfrak^*_2,\label{eq:19}
\end{equation}
where we have set
	\begin{align}\label{eq:M1a}
	\Mfrak^*_1
	&=
	1 - \frac{1}{h(\xi_j) h(-\xi_\ell)},\\
	\Mfrak^*_2
	&=\label{eq:M2a}
	\sum_{k=1}^d \frac{|\xi_k|^2 (1-|h(\xi_k)|^2)}{|\xi|^2 h(\xi_j) h(-\xi_\ell)}.
	\end{align}
\medskip

\textbf{Step 2}. Reduction by separating low and high frequencies:
We take a smooth cutoff function $\eta_1$ that equals one in $[-1,1]^d$ with compact support in $(-\pi,\pi)^d$. We then rescale it to
\[
\eta_L(\xi)=\eta_1(L\xi).
\]
Using the triangle inequality and $\chi\eta_L=\eta_L$, we separate the expression on the left hand side of \eqref{eq:main-ineq} into low and high frequencies:
\[
\|\Fcal^{-1}(\chi\ \Mfrak_{T}\Fcal_{dis}g)\|_{\ell^p_\gamma}
\le
\underbrace{\|\Fcal^{-1}(\eta_L \Mfrak_{T}\Fcal_{dis}g)\|_{\ell^p_\gamma}}_{I}
+
\underbrace{\|\Fcal^{-1}(\chi(1-\eta_L) \Mfrak_{T}\Fcal_{dis}g)\|_{\ell^p_\gamma}}_{II}.
\]
Term $I$ represents low frequencies (treated in Step 4) and term $II$
represents high frequencies (treated in Step 5). Hence, in order to
conclude, we only need to prove the following two statements:
\begin{enumerate}[(I)]
\item For all $L\geq L_0$ (where $L_0\geq 1$ only
  depends on $\gamma,p$ and $d$) we have
  \begin{equation}\label{eq:low}
  \|\Fcal^{-1}(\Mfrak_{T} \eta_L\Fcal_{dis}g)\|_{\ell^p_\gamma}
  \le C(d,\gamma,p)\,\|g\|_{\ell^p_\gamma}.
\end{equation}
\item For all $L\geq 1$ we have
  \begin{equation}\label{eq:high}
    \|\Fcal^{-1}(\chi(1-\eta_L) \Mfrak_{T}\Fcal_{dis}g)\|_{\ell^p_\gamma}
    \le C(d,\gamma,p,L)
    \|g\|_{\ell^p_\gamma}.
  \end{equation}
\end{enumerate}
We note that while the constants a-priori depend on the cutoff functions $\eta_1$ and $\zeta_1$ (the latter
will be introduced in Step~3), both may be constructed in a canonical way only depending on $d$.

\medskip

\textbf{Step 3}. A bound on the correction $\Mfrak^*_{T}$ for low frequencies:
This is perhaps the most important ingredient in the proof, as it is
here that we truly capture the difference between the discrete and
continuous settings. Recall that $\Mfrak^*_1$ and $\Mfrak^*_2$ are defined
in~\eqref{eq:M1a} and \eqref{eq:M2a}. In this step we prove that 
\begin{equation}\label{eq:m*-estimate}
 \|\Fcal^{-1}(\Mfrak^*_j\eta_L)\|_{\ell^1_\gamma}\le C(d,\gamma) L^{2\gamma-1},\quad j=1,2,
\end{equation}
for $L$  large enough.

We start the argument with the observation that $h(z)$, defined in \eqref{eq:hz}, and $h^{-1}(z)$ are both analytic in the disk $\{z\in\C:|z|<2\pi\}$ and we may write
\[
  \frac{1}{h(z)} = 1 + z r_1(z) \quad\text{and}\quad h(z) = 1 + z r_2(z)
\]
with two functions $r_1, r_2$ which are analytic on the disk $\{z\in\C:|z|<2\pi\}$.
	
\textbf{The term $\Mfrak_1^*$.}
This term becomes
\[
\Mfrak^*_1= 1 - \frac{1}{h(\xi_j) h(-\xi_\ell)} = \xi_\ell r_1(-\xi_\ell) - \xi_j r_1(\xi_j) + \xi_j \xi_\ell r_1(\xi_j) r_1(-\xi_\ell),
\]
which is  a linear combination of terms of the form $i\xi_m\phi(\xi)$, $m=1,\ldots,d$, with a (generic) analytic function $\phi$ on the disk $\{z\in\C:|z|<2\pi\}$.

\textbf{The term $\Mfrak_2^*$.}
Denoting the real part of $z\in\C$ by $\mathrm{Re}(z)$, we compute that
\[
  \Mfrak^*_2= \sum_{k=1}^d \frac{|\xi_k|^2 (1-|h(\xi_k)|^2)}{|\xi|^2 h(\xi_j) h(-\xi_\ell)} = \sum_{k=1}^d \frac{|\xi_k|^2
\big(2\xi_k \mathrm{Re}(r_2(\xi_k)) + |\xi_k|^2 |r_2(\xi_k)|^2\big)}{|\xi|^2 h(\xi_j) h(-\xi_\ell)},
\]
which is  a linear combination of terms of the form $\xi_m\frac{|\xi_n|^2}{|\xi|^2}\phi(\xi)$, $m,n=1,\ldots,d$, with a (generic) analytic function $\phi$ on the disk $\{z\in\C:|z|<2\pi\}$.

\medskip

Hence our problem reduces to showing that
\begin{equation}\label{eq:m*-est1}
 \left\|\Fcal^{-1} \left(i\xi_m\frac{|\xi_n|^2}{|\xi|^2}\phi(\xi)\eta_L\right) \right\|_{\ell^1_\gamma} \le C(d,\gamma,\phi) L^{2\gamma-1}
\end{equation}
and
\begin{equation}\label{eq:m*-est2}
 \left\| \Fcal^{-1} \left(i\xi_m \phi(\xi)\eta_L\right) \right\|_{\ell^1_\gamma} \le C(d,\gamma,\phi) L^{2\gamma-1}
\end{equation}
for any generic analytic function $\phi$ on the complex disc of radius
$2\pi$. For the argument consider the Schwartz functions
\begin{equation*}
  K_L=\Fcal^{-1}(\phi\eta_L)\qquad\text{and}\qquad  \hat K_L = \Fcal^{-1}(\phi({\textstyle \frac{\cdot}{L}})\eta_1),
\end{equation*}
and note that both are related through the scaling:
\[
K_L(x) = \frac{1}{L^d}\hat K_L({\textstyle \frac xL}).
\]
For what follows it is crucial to note that the family $\{\hat
K_L\}_{L\geq 1}$ is equibounded in the space of Schwartz space
functions, i.e.\ for all multi-indices $\alpha,\beta$ we have 
\begin{equation}\label{eq:KL-equi}
  \sup_x | x^\alpha \partial_x^\beta \hat{K}_L(x)| \le C(\phi,\alpha,\beta),
\end{equation}
{where $x^\alpha:=\prod_{i=1}^dx^{\alpha_i}_i$ and $\partial_x^\beta:=\prod_{i=1}^d\partial^{\beta_i}_{x_i}$.}
We now turn to the argument for \eqref{eq:m*-est1} and
\eqref{eq:m*-est2}. The latter is easily shown, in fact with a
slightly better decay rate of $L^{\gamma-1}$. Since $\gamma\geq 0$ and $L\geq1$, we have that
\begin{equation}\label{eq:weight_scale}
 (L|y|+1)^\gamma = L^\gamma (|y|+L^{-1})^\gamma\le L^\gamma (|y|+1)^\gamma,
\end{equation}
and the definition of $K_L$ yields
\begin{align*}
\left\| \Fcal^{-1} \left(i\xi_m \phi(\xi)\eta_L\right)
\right\|_{\ell^1_\gamma} &= \sum_{x\in\Z^d} |\partial_m K_L(x)|\
(|x|+1)^\gamma\\
&\le L^{\gamma-1}\left(L^{-d}\sum_{x\in \frac{1}{L}\Z^d} |\partial_m \hat{K}_L(x)|\ (|x|+1)^\gamma\right).
\end{align*}
Thanks to \eqref{eq:KL-equi} the term in the brackets on the
right-hand side is bounded by $C(d,\gamma,\phi)$ and
\eqref{eq:m*-est2} follows. To show~\eqref{eq:m*-est1}, we notice that
\[
\mathcal{F}^{-1}\Big({\frac{\xi_m}{|\xi|^2}}\Big) =
\frac{(2\pi)^{\frac{d}{2}}}{|S^{d-1}|}\frac{x_m}{|x|^d}\qquad\text{as
  a tempered distribution on $\R^d$},
\]
where $|S^{d-1}|$ denotes the surface area of the $d-1$-dimensional unit sphere $S^{d-1}\subset \R^d$. Therefore standard properties of the Fourier transform yield
\begin{equation}\label{eq:sing}
 \Fcal^{-1} \left(i\xi_m\frac{\xi_n^2}{|\xi|^2}\phi(\xi)\eta_L\right) = \frac{(2\pi)^{\frac{d}{2}}}{|S^{d-1}|} \partial_n^2 \left(\frac{x_m}{|x|^d} \ast K_L\right).
\end{equation}
Next we introduce a spatial cutoff $\zeta_L$ (as opposed to the frequency cutoff $\eta_L$), defined as follows: first define a smooth cutoff function $\zeta_1$ for $\{x\in\R^d:|x|\le1\}$ in $\{x\in\R^d:|x|\le2\}$ and its rescaled version
\[
 \zeta_L(x) = \zeta_1({\textstyle \frac xL}).
\]
By the triangle inequality and since the derivative in
\eqref{eq:sing} may fall on either term in the convolution, for
\eqref{eq:m*-est1} we only need to argue that
\begin{equation}\label{eq:m*-est22}
 \sum_{x} \left| \left(\frac{\zeta_Lx_m}{|x|^d} \ast \partial_n^2K_L\right)(x) \right| (|x|+1)^\gamma \le C(d,\gamma,\phi) L^{2\gamma-1}
\end{equation}
and 
\begin{equation}\label{eq:m*-est3}
 \sum_{x}\left| \left( \partial_n^2 \frac{(1-\zeta_L)x_m}{|x|^d} \ast K_L \right)(x) \right| (|x|+1)^\gamma
 \le C(d,\gamma,\phi) L^{2\gamma-1}.
\end{equation}
By definition of the (continuous) convolution, thanks to
\begin{equation*}
  (|x|+1)^\gamma\leq (|x-y|+1)^\gamma(|y|+1)^\gamma\qquad\text{for all }x,y\in\Z^d,\ \gamma\geq0,
\end{equation*}
by a change of variables and~\eqref{eq:weight_scale}, we obtain that
\begin{align}\nonumber
  \text{[l.h.s. of \eqref{eq:m*-est22}]}=\,&\sum_{x\in\Z^d} \bigg| \int_{\R^d} \frac{\zeta_L(y) y_m}{|y|^d} \partial_n^2 K_L(x-y) \;dy \bigg| (|x|+1)^\gamma\\
 \le\,& \sum_{x\in\Z^d} \int_{\R^d} \Big| \frac{\zeta_L(y) y_m}{|y|^d} \partial_n^2 K_L(x-y) \Big| (|x-y|+1)^\gamma(|y|+1)^\gamma\;dy\nonumber\\
 =\,& \sum_{x\in\frac1L\Z^d} \int_{\R^d} \Big| L \frac{\zeta_1(y) y_m}{|y|^d} L^{-2-d} \partial_n^2 \hat{K}_L(x-y) \Big| (L|x-y|+1)^\gamma(L|y|+1)^\gamma\;dy.\label{eq:step3.1}
\end{align}
Hence~\eqref{eq:weight_scale} yields
\begin{align*}
  &\text{[l.h.s. of \eqref{eq:m*-est22}]}\\
  &\leq\,L^{2\gamma-1}L^{-d}\sum_{x\in\frac1L\Z^d} \int_{\R^d} \Big| \frac{\zeta_1(y) y_m (|y|+1)^\gamma}{|y|^d} \Big| (|x-y|+1)^\gamma
  \big| \partial_n^2 \hat{K}_L(x-y) \big| \;dy\\
  &\leq\,L^{2\gamma-1}\int_{|y|\leq 2} |y|^{1-d}(|y|+1)^\gamma\Big(L^{-d}\sum_{x\in\frac1L\Z^d} (|x-y|+1)^\gamma
  \big| \partial_n^2 \hat{K}_L(x-y) \big|\Big) \;dy.
\end{align*}
The Schwartz property~\eqref{eq:KL-equi} yields
\begin{equation*}
  \Big(L^{-d}\sum_{x\in\frac1L\Z^d} (|x-y|+1)^\gamma
  \big| \partial_n^2 \hat{K}_L(x-y) \big|\Big)\leq C(d,\gamma,\phi),
\end{equation*}
and thus
\begin{align*}
  \text{[l.h.s. of \eqref{eq:m*-est22}]}\,\leq\,
  C(\phi)L^{2\gamma-1} \int_{|y|\leq 2}
  |y|^{1-d}(|y|+1)^\gamma\;dy\leq C(d,\gamma,\phi)\,L^{2\gamma-1},
\end{align*}
which completes the argument for \eqref{eq:m*-est22}.
The second term \eqref{eq:m*-est3} is bounded similarly: by the same triangle inequality and change of variables that allowed us to arrive
at \eqref{eq:step3.1}, we obtain a bound on the l.~h.~s.\ of \eqref{eq:m*-est3} by
\[
 L^{-1-d}\sum_{x\in\frac1L\Z^d} \int_{\R^d} \Big| \partial_n^2 \frac{(1-\zeta_1(x-y)) (x_m-y_m)}{|x-y|^d} \Big| (L|x-y|+1)^\gamma
|\hat{K}_L(y)|
(L|y|+1)^\gamma \;dy.
\]
We insert \eqref{eq:weight_scale} again to obtain a bound by
\[
 L^{2\gamma-1} \frac{1}{L^d}\sum_{x\in\frac1L\Z^d} \int_{\R^d} \Big| \partial_n^2 \frac{(1-\zeta_1(x-y)) (x_m-y_m)}{|x-y|^d} \Big|
(|x-y|+1)^\gamma |\hat{K}_L(y)| (|y|+1)^\gamma \;dy.
\]
This time, we use that $\left|\partial_n^2\Big((1-\zeta_1(x-y)) (x_m-y_m)|x-y|^{-d}\Big)\right| \, (|x-y|+1)^\gamma$ is integrable for large $x-y$ and
vanishes for $|x-y| \le 1$, to obtain that
\[
  \frac{1}{L^d}\sum_{x\in\frac1L\Z^d} \Big| \partial_n^2 \frac{(1-\zeta_1(x-y)) (x_m-y_m)}{|x-y|^d} \Big| (|x-y|+1)^\gamma \le C(d,\gamma,\phi).
\]
Consequently, it remains to bound
\[
 L^{2\gamma-1} \int_{\R^d} |\hat{K}_L(y)| (|y|+1)^\gamma \;dy,
\]
which, thanks to \eqref{eq:KL-equi}, is clearly bounded by $C(d,\gamma,\phi)L^{2\gamma-1}$.
\medskip

\textbf{Step 4}. Low frequencies -- proof of \eqref{eq:low}:
We assume that $L$ is large enough, so that we can apply
Lemma~\ref{lemma:equiv-norms} to deduce the equivalence of the norm
$\ell^p_\gamma$ and $L^p_\gamma$.
For brevity we set $F= \eta_L\Fcal_{dis}g$.
Equation~\eqref{eq:step1} yields
	\begin{equation}\label{eq:disc-cont}
	\| \Fcal^{-1}(\Mfrak_{T} F) \|_{\ell^p_\gamma}
	\le
	\| \Fcal^{-1}(\Mfrak_{T}^{cont} F) \|_{\ell^p_\gamma}
	+
	\| \Fcal^{-1}(\Mfrak_{T}\Mfrak^*_{T} F) \|_{\ell^p_\gamma}.
	\end{equation}
With help of the continuum Calder\`on-Zygmund estimate, cf.~Proposition~\ref{prop:weighted-cont-cz}, and the equivalence of
discrete and continuous norms, see Lemma~\ref{lemma:equiv-norms}, we
get for the first term:
	\[
	\|\Fcal^{-1}(\Mfrak_{T}^{cont}F)\|_{L^p_\gamma}\leq C
	\|g\|_{\ell^p_\gamma}.
	\]
Hence, we only need to estimate
the term $\Fcal^{-1}(\Mfrak_{T}\Mfrak^*_{T} F)$. First we notice that by definition of $F$ and $\eta_{L}$, we have that $F=\eta_{L/2}F$.
Since the Fourier transform turns multiplication into convolution, we have
\begin{align} \label{eq:M_frak-split}
	&\Fcal^{-1}(\Mfrak_{T}\Mfrak_{T}^*F)
	\stackrel{\eqref{eq:19}}{=}
	\Fcal^{-1} \left(\Mfrak_{T} \left(\Mfrak_1^* + \frac{|\xi|^2}{{\frac{1}{T}}+|\xi|^2} \Mfrak^*_2 \right) \eta_{\frac{L}{2}} F \right)\\\nonumber
	&=
	(2\pi)^{d/2} \left(\Fcal^{-1} \left( \Mfrak_1^*\eta_{\frac{L}{2}} \right)  \ast_{dis} \Fcal^{-1} (\Mfrak_{T} F) + \Fcal^{-1} \left( \Mfrak_2^* \eta_{\frac{L}{2}} \right) \ast_{dis} \Fcal^{-1}\left( \frac{|\xi|^2}{{\frac{1}{T}}+|\xi|^2} \Mfrak_{T} F \right) \right).
\end{align}
We estimate the right-hand side using the Young's inequality of Lemma \ref{lemma:weighted-young}.
For the first term, we get
	\[
	\|\Fcal^{-1}(\Mfrak_1^*\eta_{\frac{L}{2}})\ast_{dis}\Fcal^{-1}(\Mfrak_{T} F)\|_{\ell^p_\gamma}
	\le
	\|\Fcal^{-1}(\Mfrak_1^*\eta_{\frac{L}{2}})\|_{\ell_\gamma^1}\|\Fcal^{-1}(\Mfrak_{T} F)\|_{\ell_\gamma^p},
	\]
and likewise for the second term:
\[
 \left\| \Fcal^{-1}\left( \Mfrak_2^* \eta_{\frac{L}{2}} \right) \ast_{dis} \Fcal^{-1}\left( \frac{|\xi|^2}{{\frac{1}{T}}+|\xi|^2} \Mfrak_{T} F\right)
\right\|_{\ell^p_\gamma}
\le \left\| \Fcal^{-1}\left( \Mfrak_2^*
\eta_{\frac{L}{2}} \right) \right\|_{\ell^1_\gamma} \left\| \Fcal^{-1}\left( \frac{|\xi|^2}{{\frac{1}{T}}+|\xi|^2} \Mfrak_{T}
F\right)\right\|_{\ell^p_\gamma}.
\]
In both cases, the first term is bounded by~\eqref{eq:m*-estimate}, see Step 3. Hence, we have shown
\begin{align}\label{eq:M1}
 \left\|\Fcal^{-1}\left(\Mfrak_1^*\eta_{\frac{L}{2}} \right) \ast_{dis} \Fcal^{-1}\left(\Mfrak_{T} F\right)\right\|_{\ell^p_\gamma} &\le
C L^{2\gamma-1} \left\|\Fcal^{-1}\left(\Mfrak_{T} F\right)\right\|_{\ell^p_\gamma},\\
  \left\| \Fcal^{-1}\left( \Mfrak_2^* \eta_{\frac{L}{2}} \right) \ast_{dis} \Fcal^{-1}\left(\frac{|\xi|^2}{{\frac{1}{T}}+|\xi|^2} \Mfrak_{T} F\right)
\right\|_{\ell^p_\gamma} &\le C L^{2\gamma-1}
\left\|
\Fcal^{-1}\left( \frac{|\xi|^2}{{\frac{1}{T}}+|\xi|^2} \Mfrak_{T} F\right)\right\|_{\ell^p_\gamma}.\label{eq:M2}
\end{align}
We may use the equivalence of norms for band-restricted functions, cf.~Lemma~\ref{lemma:equiv-norms}, and then write the last term as another convolution to obtain that
\begin{align*}
  \left\| \Fcal^{-1}\left( \frac{|\xi|^2}{{\frac{1}{T}}+|\xi|^2} \Mfrak_{T} F\right)\right\|_{\ell^p_\gamma} &\le C \left\|
    \Fcal^{-1}\left( \frac{|\xi|^2}{{\frac{1}{T}}+|\xi|^2} \right) \ast \Fcal^{-1}\left( \Mfrak_{T} F\right)\right\|_{L^p_\gamma}\\
  &\le C \left\| \Fcal^{-1}\left(\Mfrak_{T} F\right)\right\|_{L^p_\gamma},
\end{align*}
where for the second inequality we used  the continuum Calder\'on-Zygmund estimate with Muckenhoupt weights for the Fourier-multiplier $|\xi|^2/({\frac{1}{T}}+|\xi|^2)$ which follows from Proposition \ref{prop:weighted-cont-cz}.
Combining \eqref{eq:M_frak-split}, \eqref{eq:M1} and \eqref{eq:M2}
and using the equivalence of norms yet again, we arrive at
	\begin{align*}
	\|\Fcal^{-1}(\Mfrak_{T}\Mfrak_{T}^*F)\|_{\ell^p_\gamma}
	&\le C
	L^{2\gamma-1}\|\Fcal^{-1}(\Mfrak_{T} F)\|_{L_\gamma^p}\\ &\le C
	L^{2\gamma-1}\|\Fcal^{-1}(\Mfrak_{T} F)\|_{\ell_\gamma^p}.
	\end{align*}
Hence, for $L$ sufficiently large the right-hand side may be absorbed
into the left-hand side of \eqref{eq:disc-cont}, and \eqref{eq:low} follows.

\medskip

\textbf{Step 5}. High frequencies -- proof of~\eqref{eq:high}: By the weighted convolution estimate of Lemma~\ref{lemma:weighted-young}, we have that
\begin{align*}
  \|\mathcal F^{-1}(\Mfrak_{T}(1-\eta_L)\chi\mathcal F_{dis}
  g)\|_{\ell^p_\gamma}
  &=
  \|\mathcal F^{-1}(\Mfrak_{T}(1-\eta_L)\chi)\ast_{dist}\mathcal F^{-1}(\chi\mathcal F_{dis}
  g)\|_{\ell^p_\gamma}\\
  &\le
  \|\mathcal F^{-1}(\Mfrak_{T}(1-\eta_L)\chi)\|_{\ell^1_\gamma}\|\mathcal F^{-1}(\chi\mathcal F_{dis}
  g)\|_{\ell^p_\gamma}.
\end{align*}
where we haved used that $\chi^2 = \chi$ by definition. {By the Fourier inversion formula~\eqref{eq:Finv}, the right-hand side equals $\|\mathcal F^{-1}(\Mfrak_{T}(1-\eta_L)\chi)\|_{\ell^1_\gamma}\|g\|_{\ell^p_\gamma}$ whereof we just need to estimate the first term. We have that}
\begin{align*}
  \|\mathcal F^{-1}(\Mfrak_{T}(1-\eta_L)\chi)\|_{\ell^1_\gamma}
  &= \sum_{x\in\Z^d}|\mathcal F^{-1}(\Mfrak_{T}(1-\eta_L)\chi)(x)|(1+|x|)^\gamma\\
  &=
  \sum_{x\in\Z^d}|\mathcal
  F^{-1}(\Mfrak_{T}(1-\eta_L)\chi)(x)|(1+|x|)^{\gamma+2d}(1+|x|)^{-2d}\\
  &\le C\sup_{x\in\Z^d}\big|\mathcal
  F^{-1}(\Mfrak_{T}(1-\eta_L)\chi)(x)(1+|x|)^{\gamma+2d}\big|.
\end{align*}
We rewrite this result using the definition of the Fourier transform
and integration by parts. Let $x\in\Z^d$ and let $\alpha\in\N^d$ be an
arbitrary multi-index such that $|\alpha| \ge \gamma + 2d$. {Then we have that:}
\begin{align*}
  x^{2\alpha}\mathcal
  F^{-1}(\Mfrak_{T}(1-\eta_L)\chi)(x)&=(2\pi)^{-d}\int_{(-\pi,\pi)^d}\Mfrak_{T}(\xi)(1-\eta_L)(\xi)x^{2\alpha}e^{i\xi\cdot
  x}\,d\xi\\
  &=(2\pi)^{-d}\int_{(-\pi,\pi)^d}\Mfrak_{T}(\xi)(1-\eta_L)(\xi)i^{2|\alpha|}\partial_\xi^{2\alpha}e^{i\xi\cdot
  x}\,d\xi\\
&=(2\pi)^{-d}\int_{(-\pi,\pi)^d}i^{2|\alpha|}\partial_\xi^{2\alpha}\big(\Mfrak_{T}(1-\eta_L)\big)(\xi)e^{i\xi\cdot
  x}\,d\xi.
\end{align*}
For the integration by parts when passing from the second to third lines of the last identity, we used that $\Mfrak_{T}(\xi)(1-\eta_L(\xi))$ and $\exp(i\xi\cdot x)$ are
$(-\pi,\pi)^d$-periodic function of $\xi$. It remains to argue that the latter integral is bounded by a constant $C(L,\alpha)$.
The main difficulty lies in checking that the estimate is uniform in
$T\geq 1$.
Since the integral over the Brillouin zone is finite, it suffices to show that
\begin{equation}\label{eq:CZ3}
  \sup_{\xi\in(-\pi,\pi)^d\setminus (-\frac{1}{L},\frac{1}{L})^d}|\partial_\xi^{\alpha}\Mfrak_{T}(\xi)|\le C(L,\alpha)
\end{equation}
for all multi-indices $\alpha\in\N^d$. Note that
\begin{equation*}
  \Mfrak_{T}(\xi) = \frac{\sum_{j=1}^d |\exp(i\xi_j)-1|^2}{{\frac{1}{T}} + \sum_{j=1}^d |\exp(i\xi_j)-1|^2} \Mfrak_0(\xi)
\end{equation*}
and $\Mfrak_0$ is smooth away from the origin so that
\[
 \sup_{\xi\in(-\pi,\pi)^d\setminus (-\frac{1}{L},\frac{1}{L})^d} |\partial_\xi^{\alpha} \Mfrak_0(\xi)| \le C(L,\alpha)
\]
for all multi-indices $\alpha\in\N^d$. Furthermore, we have that
\[
\sup_{\xi\in(-\pi,\pi)^d\setminus (-\frac{1}{L},\frac{1}{L})^d} \frac{1}{{\frac{1}{T}}+\sum_{j=1}^d|\exp(i\xi_j)-1|^2} \le C(d,L)
\]
and 
\[
\partial_\xi^\alpha\left(\frac{\sum_{j=1}^d|\exp(i\xi_j)-1|^2}{{\frac{1}{T}}+\sum_{j=1}^d|\exp(i\xi_j)-1|^2}\right) =
\frac{\phi(\xi)}{({\frac{1}{T}}+\sum_{j=1}^d|\exp(i\xi_j)-1|^2)^k}
\]
for some (generic) smooth function $\phi$ and some $k\ge 0$, both depending only on the multi-index $\alpha$ and $d$. Hence we have that
\[
 \sup_{\xi\in(-\pi,\pi)^d\setminus (-\frac{1}{L},\frac{1}{L})^d}
\bigg| \partial_\xi^{\alpha}\left(\frac{\sum_{j=1}^d|\exp(i\xi_j)-1|^2}{{\frac{1}{T}}+\sum_{j=1}^d|\exp(i\xi_j)-1|^2}\right) \bigg| \le C(L,\alpha).
\]
Since $\alpha$ was arbitrary, estimate \eqref{eq:CZ3} follows from the Leibniz rule.
\qed

\bigskip

\begin{proof}[Proof of Lemma \ref{lemma:weighted-young}]
First we write $|f(x-y)g(y)|$ as
	\[
	|f(x-y)g(y)|
	=
	\underbrace{|f(x-y)|^{\frac qp}|g(y)|^{\frac rp}}_{I}
	\underbrace{|f(x-y)|^{1-\frac qp}}_{II}\underbrace{|g(y)|^{1-\frac rp}}_{III}
	\] 
and apply a H\"older inequality to the terms $I,II$ and $III$ with exponents $p,\frac{pq}{p-q},\frac{pr}{p-r}$ to obtain:
	\[
	\sum_{y\in\Z^d} f(x-y)g(y)
	\le
	\bigg(\sum_{y\in\Z^d}|f(x-y)|^q|g(y)|^r\bigg)^{\frac 1p}
	\bigg(\sum_{y\in\Z^d}|f(x-y)|^q\bigg)^{\frac 1q-\frac 1p}
	\bigg(\sum_{y\in\Z^d}|g(y)|^r\bigg)^{\frac 1r-\frac 1p}.
	\]
Therefore
	\begin{align*}
	\sum_{x\in\Z^d} \Big|\sum_{y\in\Z^d} f(x-y)g(y)\Big|^pw(x)
	&\le
	\bigg(\sum_{x,y\in\Z^d}|f(x-y)|^q|g(y)|^rw(x)\bigg)\|f\|_{\ell^q}^{p-q}\|g\|_{\ell^r}^{p-r}\\
	&\le
	\left(\|f\|^q_{\ell_w^q}\|g\|_{\ell_w^r}^r\right)\|f\|_{\ell^q_w}^{p-q}	\|g\|_{\ell^r_w}^{p-r}\\
	&=
	\textcolor{blue}{\|f\|_{\ell_w^q}^p\|g\|_{\ell_w^r}^p},
	\end{align*}
where in the second inequality we used the assumption \eqref{eq:23}.
\end{proof}

\begin{proof}[Proof of Lemma \ref{lemma:equiv-norms}]
For convenience we set $Q:=(-\tfrac{1}{2},\tfrac{1}{2})^d$ and without
loss of generality we assume that $L\geq 1$.

\textbf{Step 1.} We claim that for all
$z\in\Z^d$ and $1\leq p<\infty$ we have
\begin{align}\label{eq:22}
  \sup_{x\in (z+Q)}|g(x)|&\leq
  C(d,p)\|g\|_{L^p(z+Q)},\\
  \label{eq:24}
  \|g\|_{L^p(z+Q)}&\leq C(d,p)\left(|g(z)|+L^{-1}\|g\|_{L^p(z+Q)}\right).
\end{align}

By translation invariance it suffices to consider $z=0$. Thanks to the
Sobolev embedding of $W^{n,p}(Q)$ into $L^\infty(Q)$ for $n>d$, we get
\begin{equation}\label{eq:20}
  \sup_{x\in Q}|g(x)|\leq
  C(d,n,p)\|g\|_{L^p(Q)}+\|\nabla^ng\|_{L^p(Q)}.
\end{equation}
We argue that the band restriction implies for all $n\geq 1$ that
\begin{equation}\label{eq:21}
  \|\nabla^ng\|_{L^p(Q)}\leq C(d,n)L^{-n}\|g\|_{L^p(Q)},
\end{equation}
which combined with \eqref{eq:20} and $L\geq 1$ yields \eqref{eq:22}.

Estimate \eqref{eq:21} can be seen as follows: Recall that $g=\mathcal
F^{-1}F$ where $F$ is supported in  $[-\tfrac{1}{L},\tfrac{1}{L}]$. Let
$\eta_1$ denote a smooth cutoff function that is one in $[-1,1]^d$ and
compactly supported in $(-2,2)^d$, say. Let $\phi_1:=\mathcal F^{-1}\eta_1$ and note that
for all $L>0$ we have
\[
(\mathcal F^{-1}\eta_L)(x)=\phi_L\qquad\text{where
}\eta_L(\xi):=\eta_1(L\xi)\text{ and }
\phi_L(x):=L^{-d}\phi_1(\tfrac{x}{L}). 
\]
In view of the band restriction of $F$ and its definition we have $g={\mathcal F}^{-1}F={\mathcal F}^{-1}(\eta_L F)
=(2\pi)^\frac{d}{2}{\mathcal F}^{-1}\eta_L*{\mathcal F}^{-1}F=\phi_L*g$. 
We thus obtain the representation $\nabla^ng=\nabla^n(\phi_L*g)=(\nabla^n\phi_L)*g$ with $\nabla^n\phi_L(x)
=L^{-n}\frac{1}{L^d}\nabla^n\phi_1(\frac{x}{L})$, which yields the inequality
\begin{equation}\nonumber
  \|\nabla^ng\|_{L^p}\le\|\nabla^n\phi_L\|_{L^1}\|g\|_{L^p}=L^{-n}\|\nabla^n\phi_1\|_{L^1}\|g\|_{L^p},
\end{equation}
and thus the estimate \eqref{eq:21}, since $\phi_1$ is a Schwartz
function that can be chosen only depending on $d$.

{Estimate \eqref{eq:24} may be seen as follows: A simple application of the mean-value theorem yields
\[
 \left(\int_{Q}|g(x)-g(0)|^p\,dx\right)^{\frac{1}{p}} \leq C(d,p)\sup_{x\in Q}|\nabla g(x)|.
\]
Then the Sobolev embedding~\eqref{eq:20} with $g$ replaced by $\nabla g$ yields
\[
 \left(\int_{Q}|g(x)-g(0)|^p\,dx\right)^{\frac{1}{p}} \leq C(d,n,p)\|\nabla g\|_{L^p(Q)}+\|\nabla^{n+1}g\|_{L^p(Q)}.
\]
Finally, we insert estimate~\eqref{eq:21} (with $n$ replaced by $n+1$) to obtain that
\[
 \left(\int_{Q}|g(x)-g(0)|^p\,dx\right)^{\frac{1}{p}} \leq C(d,n,p)(L^{-1}+L^{-(n+1)})\|g\|_{L^p(Q)},
\]
which easily turns into the desired estimate~\eqref{eq:24} at $z=0$.}

\medskip

\textbf{Step 2.} We claim that there exists $L_0=L_0(d,p)$ such that for all $L\geq L_0$ and $z\in\Z^d$ we have
\begin{equation}
  \label{eq:26}
  \frac{1}{C(d,p,\gamma)}|g(z)|^p(|z|+1)^\gamma\leq\int_{z+Q}|g(x)|^p(|x|+1)^\gamma\,dx\leq C(d,p,\gamma)|g(z)|^p(|z|+1)^\gamma.
\end{equation}
For the argument first note that for all $z\in\Z^d$ and $x\in z+Q$ we have
\begin{equation}\label{eq:25}
  (|z|+1)^\gamma\leq C(d,\gamma)(|x|+1)^\gamma\qquad\text{and}\qquad 
  (|x|+1)^\gamma\leq C(d,\gamma)(|z|+1)^\gamma.
\end{equation}
{Indeed, since $\max_{y\in Q}|y|+1 = \frac{1}{2}\sqrt{d}+1$ we have that
\[
  (|z|+1)^\gamma\leq (|x|+|z-x|+1)^\gamma\leq (|x|+{\textstyle\frac{1}{2}\sqrt{d}+1})^\gamma\leq ({\textstyle\frac{1}{2}\sqrt{d}+1})^\gamma(|x|+1)^\gamma,
\]
and
\[
  (|x|+1)^\gamma\leq (|z|+|x-z|+1)^\gamma\leq ({\textstyle\frac{1}{2}\sqrt{d}+1})^\gamma(|z|+1)^\gamma.
\]
Hence the result~\eqref{eq:22} of Step~1 yields
\[
  |g(z)|^p(|z|+1)^\gamma
  \le\left(\sup_{x\in
      z+Q}|g(x)|\right)^p(|z|+1)^\gamma
  \le C(d,p)\int_{z+Q}|g(x)|^p(|z|+1)^\gamma\,dx
\]
Estimate~\eqref{eq:25} thus yields the desired first inequality
\[
  |g(z)|^p(|z|+1)^\gamma
  \le C(d,p,\gamma)\int_{z+Q}|g(x)|^p(|x|+1)^\gamma\,dx.
\]
For the second estimate in~\eqref{eq:26}, we note that, by absorption,~\eqref{eq:24}  implies existence of $L_0=L_0(d,p)$ such that
\[
  \int_{z+Q}|g(x)|^p\,dx\leq C(d,p)|g(z)|^p
\]
for all $L\geq L_0$.
Hence another application of~\eqref{eq:25} yields as desired
\begin{equation*}
  \int_{z+Q}|g(x)|^p(|x|+1)^\gamma\,dx
  \leq C(d,\gamma)\int_{z+Q}|g(x)|^p\,dx\ (|z|+1)^\gamma
  \leq C(d,p,\gamma)|g(z)|^p(|z|+1)^\gamma
\end{equation*}
 for all $L\geq L_0$.}

\medskip

\textbf{Step 3.} Conclusion:
The estimate $\|g\|_{L^p_\gamma}^p\leq
C(d,p,\gamma)\|g\|_{\ell^p_\gamma}^p$ follows from the second part of \eqref{eq:26} by
summation in $z\in\Z^d$. For the opposite inequality, estimate~\eqref{eq:22} and H\"older's inequality yield
\[
  \frac{1}{C(d,q)}|g(0)|^p \le \left(\int_Q|g|^q\,dx\right)^{\frac{p}{q}}
  \leq\left(\int_Q|g|^p|x|^\gamma\,dx\right)\left(\int_Q|x|^{-\frac{q}{p-q}\gamma}\,dx\right)^{\frac{p-q}{p}}
\]
for all $1\leq q<p$.
Thanks to the assumption $0\leq \gamma<d(p-1)$, we can find $1\leq q<p$
such that the second integral on the right-hand side is finite, so that
\begin{equation}\label{eq:27}
  |g(0)|^p\leq C(d,p,\gamma)\int_Q|g(x)|^p|x|^\gamma\,dx.
\end{equation}
(Note that this is the only place where the upper bound on $\gamma$ is
required.) We conclude by~\eqref{eq:27} and~\eqref{eq:26} that
\begin{align*}
  \|g\|^p_{\ell^p_\gamma}&=\sum_{z\in\Z^d}|g(z)|(|z|+1)^\gamma=|g(0)|+\sum_{z\in\Z^d\setminus\{0\}}|g(z)|^p(|z|+1)^\gamma\\
&\leq C(d,p,\gamma)\left(\int_Q|g(x)|^p|x|^\gamma\,dx+\int_{\R^d\setminus Q}|g(x)|^p(|x|+1)^\gamma\right)\\
&\leq C(d,p,\gamma)\int_{\R^d}|g(x)|^p|x|^\gamma\,dx,
\end{align*}
where in the last line we have used that $|x|+1 \le 3 |x|$ for all $|x| \ge \frac{1}{2}$.
\end{proof}

\section{Applications}\label{sec:app}

We present three  applications of our results. First, for $d\geq 3$ we prove a bound (with optimal scaling) on the $H^1$-error of the discrete two-scale expansion \eqref{eq:11}. Secondly, we consider an approximation of the homogenized coefficients by a spatial average of the energy density associated with the modified corrector and estimate its $p$-variance. 

The third application are \textit{annealed Green's function estimates} in the spirit of \cite{MO1}. More precisely, we present an argument that allows to lift estimates on the first moment of the annealed Green's function (and its gradients) to higher moments. Next to the moment bounds on the corrector and the estimates on the Green's function developed in the previous sections, these \textit{annealed Green's function estimates} play also an important role in the argument for the first two applications.

\subsection{Quantitative two-scale expansion}\label{sec:qtwoscale}
In this paragraph we prove an optimal quantitative two-scale expansion under the following assumptions:
\begin{enumerate}[(a)]
\item $d\geq 3$,
\item $\expec{\cdot}$ is stationary and satisfies LSI,
\item suboptimal decay of the annealed Green's function:
  \begin{equation}\label{ASS:annealed-theta}
    \exists \theta>\frac{d+2}{d}\,:\,\sup_{x\in\Z^d}(|x|+1)^\theta\expec{|\nabla\nabla G(x,0)|}<\infty.
  \end{equation}
\end{enumerate}
Note that (a) and (b) ensure the existence of stationary correctors  $\phi_1,\ldots,\phi_d$ satisfying
\begin{equation*}
  \expec{\phi^2_i}\leq C(d,\lambda,\rho),
\end{equation*}
cf. Corollary \ref{cor:1}.  
Property \eqref{ASS:annealed-theta} is satisfied (with optimal exponent $\theta=d$), if $\nabla^*(a\nabla)$ satisfies a maximum principle (e.g. see \cite{MO2}). We expect \eqref{ASS:annealed-theta} to hold for general elliptic, stationary \& ergodic coefficients. Yet, we could not find a reference.
\medskip

In the following, for $u:\e\Z^d\to\R$ and $F:\e\Z^d\to\R^d$ we set
\begin{eqnarray*}
  \nabla^\e u(x):=(\nabla_1^\e u(x),\ldots,\nabla_d^\e u(x)),\qquad \nabla_i^\e u(x):=\frac{u(x+\e e_i)-u(x)}{\e},\\
  \nabla^{\e*} F(x):=\sum_{i=1}^d\nabla_i^{\e*}F_i(x),\qquad \nabla_i^{\e*}F_i(x):=\frac{F_i(x-\e e_i)-F(x)}{\e}.
\end{eqnarray*}
\begin{proposition}\label{prop:two-scale}
 There exists a constant $C$ such that for all $\mu^2>0$, $f\in \ell^2(\e\Z^d)$ and $\e>0$ the following property holds: Let $u_\e(a;\cdot)$ and $u_{\hom}^\e$ denote the unique solutions in $\ell^2(\e\Z^d)$ to the equations
  \begin{equation}\label{twoscale:eq1}
    \begin{aligned}
      \mu^2 u^\e(a;\cdot)+\nabla^{\e*}(a(\tfrac{\cdot}{\e})\nabla^\e u^\e(a;\cdot))=&f\qquad\text{in }\Z^d,\\
      \mu^2 u_{\hom}^\e+\nabla^{\e*}(a_{\hom}\nabla^\e
      u_{\hom}^\e)=&f\qquad\text{in }\Z^d.
    \end{aligned}
  \end{equation}
  Then the two-scale expansion
  \begin{equation*}
    z^\e(a;x)=u^\e(a;x)-\left(u_{\hom}^\e(x)+\e\phi_i(a;\frac{x}{\e})\nabla^\e_iu_{\hom}^\e(x)\right)
  \end{equation*}
  satisfies
  \begin{equation*}
    \expec{\left(\e^d\sum_{x\in\e\Z^d}|z^\e(x)|^2+|\nabla^\e z^\e(x)|^2\right)^2}^\frac{1}{2}\leq C(1+\mu^2)\e\left(\e^d\sum_{x\in\e\Z^d}|f(x)|^2\right)^\frac{1}{2}.
  \end{equation*}
\end{proposition}
The previous proposition extends the two-scale expansion in \cite{GNO3} in several directions: In \cite{GNO3} the equations in \eqref{twoscale:eq1} are considered on a torus (i.e. with periodic boundary conditions) and it is assumed that the coefficients are diagonal and i.~i.~d..  Recently, in \cite{GNO4} an estimate for the two-scale expansion based on the non-stationary corrector has been obtained in the continuum case.
\medskip

The proof of Proposition~\ref{prop:two-scale} relies on a classical representation of the energy density associated with $z_\e$ with help of the corrector function $\phi_i$ and an additional quantity, which we refer to as the ``flux correctors''. It is associated to the flux differences
\begin{equation*}
  q=(q^1,\ldots,q^d),\qquad
  q^i(a;x):=a(x)(e_i+\nabla\phi_i(a;x))-a_{\hom}e_i.
\end{equation*}
In the deterministic, periodic case, a classical result shows that $q^i$ can be represented as the divergence of a (skew symmetric, matrix-valued) flux corector $\sigma^i$, which in our setting can be defined with the help of the equation
\begin{equation}\label{eq:flux-corrector}
  \nabla^*\nabla  \sigma^i_{\alpha\beta}=\nabla_\alpha q^i_\beta-\nabla_\beta q^j_\alpha\qquad\text{in }\Z^d.
\end{equation}
In the following we show that under assumptions (a) - (c) equation \eqref{eq:flux-corrector} admits a \textit{stationary} solution with \textit{finite second moments}:
\begin{lemma}\label{L:flux-corrector}
For all $i=1,\ldots,d$ and $\expec{\cdot}$-almost every $a\in\Omega$, the equation \eqref{eq:flux-corrector} admits a unique solution $\sigma^i:\Omega\times\Z^d\to\R^{d\times d}$ with $\expec{|\sigma^i|^2}<\infty$ and $\expec{\sigma^i}=0$. Furthermore, it satisfies the following properties:
\begin{enumerate}[(a)]
\item (skew-symmetry). $\sigma^i_{\alpha\beta}=-\sigma^i_{\beta\alpha}$ almost everywhere in $\Omega\times\Z^d$,
\item (representation of $q^i$). $\sum_{\beta=1}^d\nabla^*_\beta\sigma^i_{\alpha\beta}=q^i_\alpha$ almost everywhere in $\Omega\times\Z^d$.
\end{enumerate}
\end{lemma}
We postpone the argument to the end of this section and continue with the proof of the proposition:

\begin{proof}[Proof of Proposition~\ref{prop:two-scale}]
  A direct calculation yields
  \begin{equation*}
    a(\tfrac{\cdot}{\e})\nabla^\e z^\e = a(\tfrac{\cdot}{\e})\nabla^\e u^\e-a_{\hom}\nabla^\e u^\e_{\hom}-(\nabla^\e_i u^\e_{\hom})q^i(\tfrac{\cdot}{\e})-\e a(\tfrac{\cdot}{\e})[\phi_i(\tfrac{\cdot}{\e})]\nabla^\e\nabla_i^\e u^\e_{\hom},
  \end{equation*}
  where the $d\times d$-matrix $[\phi_i(\tfrac{x}{\e})]:=\sum_{j=1}^d\phi_i(\tfrac{x}{\e}+e_j)e_j\otimes e_j$ accounts for the discrete product rule. By appealing to \eqref{twoscale:eq1} we obtain
  \begin{equation*}
    \mu^2z^\e+\nabla^{\e,*}(a(\tfrac{\cdot}{\e})\nabla^\e z^\e) =-\e \nabla^{\e,*}\left(\,(\nabla^\e_i u^\e_{\hom})q^i(\tfrac{\cdot}{\e})+a(\tfrac{\cdot}{\e})[\phi_i(\tfrac{\cdot}{\e})]\nabla^\e\nabla_i^\e u^\e_{\hom}\right)-\e\mu^2\phi_i(\tfrac{\cdot}{\e})\nabla_i^\e u^\e_{\hom},
  \end{equation*}
  and testing with $z^\e$ yields
  \begin{equation}
    \label{P:two-scale:eq2}
    \begin{aligned}
      \sum_{\e\Z^d}\left(\mu^2|z^\e|^2+\lambda |\nabla^\e
        z^\e|^2\right) \leq\ & -\sum_{\e\Z^d}(\nabla^\e_i
      u^\e_{\hom})q^i(\tfrac{\cdot}{\e}) \cdot \nabla^\e
      z^\e\\
      &+\e\sum_{\e\Z^d}a(\tfrac{\cdot}{\e})[\phi_i(\tfrac{\cdot}{\e})]\nabla^\e\nabla^\e_i
      u^\e_{\hom}\cdot \nabla^\e
      z^\e\\
      &-\e\mu^2\sum_{\e\Z^d}\phi_i(\tfrac{\cdot}{\e})\nabla^\e_iu^\e_{\hom}z^\e.
    \end{aligned}
  \end{equation}
  We claim that
  \begin{equation}\label{P:two-scale:eq3}
    \sum_{\e\Z^d}(\nabla^\e_i
      u^\e_{\hom})q^i(\tfrac{\cdot}{\e}) \cdot \nabla^\e
      z^\e=
      \e\sum_{\e\Z^d}\Big((\nabla^{\e,*}_\alpha\nabla^\e_i
    u^\e_{\hom})\sigma^i_{\alpha\beta}(\tfrac{\cdot}{\e}-e_\alpha)\Big)\nabla^{\e}_\beta z^\e
  \end{equation}
  Indeed, thanks to the corrector equation in the form of $\nabla^{\e,*}q^i(\tfrac{\cdot}{\e})=0$ and identity $\e\nabla^{\e,*}\sigma^i(\tfrac{\cdot}{\e})=q^i(\tfrac{\cdot}{\e})$, cf. Lemma~\ref{L:flux-corrector}, we have
  \begin{eqnarray*}
    &&\sum_{\e\Z^d}(\nabla^\e_i
    u^\e_{\hom})q^i(\tfrac{\cdot}{\e}) \cdot \nabla^\e
    z^\e\\
    &=&
    \sum_{\e\Z^d}\nabla^{\e,*}\Big((\nabla^\e_i
    u^\e_{\hom})q^i(\tfrac{\cdot}{\e})\Big)z^\e\\
    &=&
    \sum_{\e\Z^d}(\nabla^\e_i
    u^\e_{\hom})\Big(\nabla^{\e,*}q^i(\tfrac{\cdot}{\e})\Big)z^\e +
    \sum_{\e\Z^d}\Big((\nabla^{\e,*}_\alpha\nabla^\e_i
    u^\e_{\hom})q^i(\tfrac{\cdot}{\e}-e_\alpha)\Big)z^\e\\
    &=&
    \e\sum_{\e\Z^d}\Big((\nabla^{\e,*}_\alpha\nabla^\e_i
    u^\e_{\hom})\nabla^{\e,*}_\beta \sigma^i_{\alpha\beta}(\tfrac{\cdot}{\e}-e_\alpha)\Big)z^\e\\
    &=&
    \e\sum_{\e\Z^d}\nabla^{\e,*}_\beta\Big((\nabla^{\e,*}_\alpha\nabla^\e_i
    u^\e_{\hom})\sigma^i_{\alpha\beta}(\tfrac{\cdot}{\e}-e_\alpha)\Big)z^\e
    -
    \e\sum_{\e\Z^d}\Big((\nabla^{\e,*}_\beta\nabla^{\e,*}_\alpha\nabla^\e_i
    u^\e_{\hom})\sigma^i_{\alpha\beta}(\tfrac{\cdot}{\e}-e_\alpha-e_{\beta})\Big)z^\e
  \end{eqnarray*}
  The last term on the right-hand side vanishes, since $(\nabla^{\e,*}_\beta\nabla^{\e,*}_\alpha\nabla^\e_i
  u^\e_{\hom})$ is symmetric and $\sigma^i_{\alpha\beta}(\tfrac{\cdot}{\e}-e_\alpha-e_\beta)$ is skew-symmetric in $\alpha$ and $\beta$, cf. Lemma~\ref{L:flux-corrector}. This proves \eqref{P:two-scale:eq3}.
  The combination of \eqref{P:two-scale:eq2} and \eqref{P:two-scale:eq3}, and Young's inequality yield for some constant $C(d)>0$ that only depends on $d$:
  \begin{eqnarray*}
    &&\frac{1}{C(d)}\sum_{\e\Z^d}\left(\mu^2|z^\e|^2+\lambda |\nabla^\e
      z^\e|^2\right) \\
    &\leq& \e^2\lambda\sum_{\e\Z^d}\left(\phi_i^2(\tfrac{\cdot}{\e}+e_\alpha)|\nabla^\e_\alpha\nabla^\e_i u^\e_{\hom}|^2+|\sigma^i_{\alpha\beta}(\tfrac{\cdot}{\e}-e_\alpha)|^2|\nabla^{\e,*}_\alpha\nabla^\e_i u^\e_{\hom}|^2\right)\\
    &&+\e^2\mu^2\sum_{\e\Z^d}\phi_i^2(\tfrac{\cdot}{\e})|\nabla_i^\e u^\e_{\hom}|^2\\
    &=& \e^2\lambda\sum_{\e\Z^d}\left((\phi_i^2(\tfrac{\cdot}{\e}+e_\alpha)+|\sigma^i(\tfrac{\cdot}{\e})|^2)|\nabla^\e_\alpha\nabla^\e_i u^\e_{\hom}|^2\right)+\e^2\mu^2\sum_{\e\Z^d}\phi_i^2(\tfrac{\cdot}{\e})|\nabla_i^\e u^\e_{\hom}|^2.
  \end{eqnarray*}
  Since $u^\e_{\hom}$ is deterministic and because the correctors $\phi_i$ and $\sigma^i$ are stationary, taking the expectation yields
  \begin{eqnarray*}
    &&\frac{1}{C(d)}\expec{\sum_{\e\Z^d}\left(\mu^2|z^\e|^2+\lambda |\nabla^\e
        z^\e|^2\right)} \\
    &\leq& \e^2\expec{\sum_i|\phi_i|^2+|\sigma^i|^2}\left(\sum_{\e\Z^d}\lambda|\nabla^\e\nabla^\e u^\e_{\hom}|^2+\mu^2|\nabla^\e u^\e_{\hom}|^2\right).
  \end{eqnarray*}

  The first term is bounded thanks to $\expec{|\phi_i|^2}<\infty$ and Lemma~\ref{L:flux-corrector}. The second term is estimated by $C(d,\lambda)(1+\mu^2)\sum_{\e\Z^d}f^2$ as can be seen by standard discrete-$H^2$-regularity for the constant coefficient operator $(\mu^2+\nabla^*a_{\hom}\nabla)$.
\end{proof}

\begin{proof}[Proof of Lemma~\ref{L:flux-corrector}]
  \step 1 Existence.

  For $T>0$ let $\sigma_{T,\alpha,\beta}^i$ denote the unique bounded solution to 
  \begin{equation}\label{L:flux:eq1}
    \frac{1}{T}\sigma_{T,\alpha\beta}^i+\nabla^*\nabla\sigma_{T,\alpha\beta}^i=\nabla_\alpha q^i_\beta-\nabla_\beta q^i_\alpha.
  \end{equation}
  Since $q^i$ is stationary and has vanishing expectation, $\sigma_{T,\alpha,\beta}$ is stationary and has vanishing expectation. 
  As shown below in Step~3, we have
  \begin{equation}\label{L:flux:eq3}
    \sup_{T\geq 1}\expec{|\sigma^i_{T,\alpha\beta}|^2}<\infty,
  \end{equation}
  and thus we recover the sought for solution $\sigma^i_{\alpha\beta}$ in the limit $T\uparrow\infty$.
  \medskip

  \step 2 Properties (a) and (b).

  Property (a) is obvious, since the right-hand side of the equation defining $\sigma$ is skew-symmetric. We prove (b) and first claim that
  \begin{equation}\label{eq:two-scale-1}
    \nabla^*\nabla\left(\nabla^*_\beta\sigma_{\alpha\beta}^i-q_\alpha^i\right)=0.
  \end{equation}
  Indeed, since $\nabla^*q^i=0$ (by the corrector equation), the claim follows from identity
  \begin{eqnarray*}
    \nabla^*\nabla\nabla^*_\beta\sigma_{\alpha\beta}^i&=&\nabla^*_\beta\nabla^*\nabla\sigma_{\alpha\beta}^i=\nabla^*_\beta\nabla_\alpha q^i_\beta-\nabla^*\beta\nabla_\beta q_\alpha^i\\
    &=&\nabla_\alpha\nabla^*q^i-\nabla^*\nabla q_\alpha^i,
  \end{eqnarray*}
  which implies that $\zeta=q^i_\alpha-\nabla^*_\beta\sigma_{\alpha\beta}^i=0$ (since $\zeta$ is stationary, $\expec{\zeta}=0$ and the kernel of $\nabla^*\nabla$ restricted to this class of functions is trivial).
  \medskip

  \step 3 Proof of \eqref{L:flux:eq3}.

  We first notice that $\sigma_{T,\alpha\beta}^i$ admits the Green's function representation
  \begin{equation}\label{L:flux:eq2}
    \sigma_{T,\alpha\beta}^i(a;0)=\sum_{z\in\Z^d}\left(\nabla_\alpha^*G^0_T(z)q^i_\beta(z)-\nabla_\beta^*G^0_T(z)q^i_\alpha(z)\right),
  \end{equation}
  where $G^0_T$ denotes the Green's function associated with the operator $\frac{1}{T}+\nabla^*\nabla$. Since $\expec{\sigma^i_{T,\alpha\beta}}=0$, (SG) yields
  \begin{equation*}
    \expec{|\sigma^i_{T,\alpha,\beta}|^2}\leq\frac{1}{\rho}\expec{\sum_{x\in\Z^d}(\osc_{a(x)}\sigma^i_{T,\alpha\beta})^2},
  \end{equation*}
  Hence, we only need to argue that 
  \begin{equation}\label{L:flux:eq4}
    \sup_{T\geq 1}\expec{\sum_{x\in\Z^d}(\osc_{a(x)}\sigma^i_{T,\alpha\beta})^2}<\infty.
  \end{equation}
  From \eqref{T1.1} (which we apply in the limit $T\to\infty$) and the definition of $q^i$ we deduce that
  \begin{equation*}
    \osc_{a(x)}q^i(a;z)\leq C(d,\lambda)\left(\delta(x-z)+|\nabla\nabla G(a;z,x)|\right)|\nabla\phi_i(a;x)+e_i|.
  \end{equation*}
  Combined with \eqref{L:flux:eq2} and the pointwise estimate $|\nabla^*G_T^0(z)|\leq C(d)(|z|+1)^{1-d}$ on the constant-coefficient Green's function, we get
  \begin{multline*}
    |\osc_{a(x)}\sigma_{T,\alpha,\beta}(0)|^2\leq C(d,\lambda)\left(\sum_{z\in\Z^d}(|z|+1)^{(1-d)}\left(\delta(x-z)+|\nabla\nabla G(a;z,x)|\right)|\nabla\phi_i(a;x)+e_i|\right)^2,
  \end{multline*}
  and thus
  \begin{eqnarray*}
    \expec{\sum_{x\in\Z^d}(\osc_{a(x)}\sigma_{T,\alpha\beta})^2}&\leq& C(d,\lambda)\Bigg(\sum_{x\in\Z^d}(|x|+1)^{2(1-d)}\expec{|\nabla\phi_i+e_i|^2}\\
    &&+\sum_{x\in\Z^d}\expec{|\nabla\phi_i(x)+e_i|^2\left(\sum_{z\in\Z^d}(|z|+1)^{(1-d)}|\nabla\nabla G(z,x)|\right)^2}\Bigg).
  \end{eqnarray*}
  Since $2(1-d)<-d$, the first term on the right-hand side is finite. We estimate the second term. By stationarity and the identities $\nabla\nabla G(a;z,x)=\nabla\nabla G(a(\cdot+x);z-x,0)$ and $\phi^i(a,x)=\phi^i(a(\cdot+x),0)$ we have
  \begin{eqnarray*}
    &&\sum_{x\in\Z^d}\expec{|\nabla\phi_i(x)+e_i|^2\left(\sum_{z\in\Z^d}(|z|+1)^{(1-d)}|\nabla\nabla G(z,x)|\right)^2}\\
    &=&
    \expec{|\nabla\phi_i(0)+e_i|^2\sum_{x\in\Z^d}\left(\sum_{z\in\Z^d}(|z|+1)^{(1-d)}|\nabla\nabla G(z-x,0)|\right)^2}
  \end{eqnarray*}
  Since $\theta>\frac{d+2}{d}$, cf. \eqref{ASS:annealed-theta}, we can find exponents $p,q$ such that
  \begin{equation*}
    \frac{d}{\theta}<p<\frac{2d}{d+2},\qquad q>\frac{d}{d-1},\qquad
    1+\frac{1}{2}=\frac1q+\frac1p.
  \end{equation*}
  H\"older's inequality (with exponents $(\frac{p}{2},\frac{p}{p-2})$) and Young's convolution estimate yield
  \begin{eqnarray*}
    &&\expec{|\nabla\phi_i(0)+e_i|^2\sum_{x\in\Z^d}\left(\sum_{z\in\Z^d}(|z|+1)^{(1-d)}|\nabla\nabla G(z-x,0)|\right)^2}\\
    &\leq&\expec{|\nabla\phi_i+e_i|^{\frac{2p}{p-2}}}^{\frac{p-2}{2}}
    \expec{\left(\sum_{x\in\Z^d}\left(\sum_{z\in\Z^d}(|z|+1)^{(1-d)}|\nabla\nabla G(z-x,0)|\right)^2\right)^{\frac{p}{2}}}^{\frac{2}{p}}\\
    &\leq&\expec{|\nabla\phi_i+e_i|^{\frac{2p}{p-2}}}^{\frac{p-2}{2}}\left(\sum_{x\in\Z^d}(|x|+1)^{q(1-d)}\right)^{\frac{2}{q}}
    \expec{\sum_{x\in\Z^d}|\nabla\nabla G(z-x,0)|^{2p}}
  \end{eqnarray*}
  Note that the choice of $q$ implies
  \begin{equation*}
    \sum_{x\in\Z^d}(|x|+1)^{q(1-d)}<\infty.
  \end{equation*}
  Furthermore, from \eqref{ASS:annealed-theta}, Proposition~\ref{P:annealed} (which we apply with the weight $\omega(r):=(|r|+1)^\theta$) and the choice of $p$, we learn that
  \begin{equation*}
    \expec{\sum_{x\in\Z^d}|\nabla\nabla G(z-x,0)|^{p}}\leq\left(\sup_{y\in\Z^d}\expec{|\nabla\nabla G(y,0)|^{p}}(|y|+1)^{p\theta}\right)\sum_{x\in\Z^d}(|z-x|+1)^{-p\theta }<\infty.
  \end{equation*}
  By Theorem~\ref{T1} any moment of $|\nabla\phi^i+e_i|$ is finite and thus \eqref{L:flux:eq4} follows.
\end{proof}

\subsection{Quantitative approximation of the homogenized coefficients}
Fix unit vectors $\xi,\xi^*\in\R^d$ and $T>0$. Let $\phi(a;x)$ (resp. $\phi^*(a;x)$) denote the modified correctors (resp. adjoint corrector) associated with $\xi$ (resp. with $\xi^*$):
\begin{align*}
  &\frac{1}{T}\phi(a;\cdot)+\nabla^*(a(\nabla\phi(a;\cdot)+\xi))=0\\
  &\frac{1}{T}\phi^*(a;\cdot)+\nabla^*(a^*(\nabla\phi^*(a;\cdot)+\xi^*))=0
\end{align*}
We introduce the (stationary) energy density
\begin{equation*}
  \mathcal E_T(a;x):=\frac{1}{T}\phi^*(a;x)\phi(a;x)+(\nabla\phi^*(a;x)+\xi^*)\cdot a(\nabla\phi(a;x)+\xi).
\end{equation*}
It approximates $a_{\hom}$ in the sense that 
\begin{equation*}
  \xi^*\cdot a_{\hom}\xi=\lim\limits_{T\to\infty}\expec{\mathcal E_T},
\end{equation*}
Furthermore, thanks to ergodicity, the average $\expec{\mathcal E_T}$ can be approximated by a spatial average: Let $\eta$ denote a smooth, non-negative function with support in $(-1,1)^d$  and $\int_{(-1,1)}\eta=1$. For $L\gg 1$ set $\eta_L(x):=L^{-d}\eta(\frac{\cdot}{L})$ and define
\begin{equation*}
  \mathcal E_{T,L}(a):=\sum_{x\in\Z^d}\eta_L(x)\mathcal E_T(a;x).
\end{equation*}
Then Birkhoff's individual ergodic theorem yields
\begin{equation*}
  \expec{\mathcal E_T}=\lim\limits_{L\to\infty}\mathcal E_{T,L}(a)\qquad\text{$\expec{\cdot}$-almost surely.}
\end{equation*}
In this section we prove a rate for this convergence.
\begin{proposition}\label{P:4}
  Let $d\geq 2$. Suppose that $\expec{\cdot}$ is stationary and satisfies (SG), then for all $1\leq p<\infty$ we have
  \begin{equation*}
    \expec{|\mathcal E_{T,L}-\expec{\mathcal E_{T}}|^{2p}}^{\frac{1}{2p}}\leq C(d,\lambda,\rho,p)(\mu_{T,L}+1)L^{-\frac{d}{2}}.
  \end{equation*}
  where
  \begin{equation*}
    \mu_{T,L}:=L^{-\frac{d+2}{2}}\left(\sum_{z\in\Z^d}\left(\sum_{|x|\leq L}\expec{|\nabla_z G_T(z-x,0)|^{4p}}^{\frac{1}{4p}}\right)^2\right)^{\frac{1}{2}}
  \end{equation*}
\end{proposition}

\begin{remark}
  If the annealed Green's function satisfies the \textit{Delmotte-Deuschel bounds}
  \begin{equation*}
    \expec{|\nabla G(x,0)|}\leq C(|x|+1)^{1-d},\qquad \expec{|\nabla_x\nabla_y G(x,y)|}\leq C(|x-y|+1)^{-d},
  \end{equation*}
  (see discussion in the next section), then an application of Proposition~\ref{P:annealed} shows that 
  \begin{equation*}
    \mu_{T,L}\lesssim 
    \begin{cases}
      \ln(1+\frac{L}{\sqrt{T+1}})&d=2,\\
      1&d>2.
    \end{cases}
  \end{equation*}
  Combined with Proposition~\ref{P:4} we recover the optimal scaling in $L$ for the $2p$-variance of $\mathcal E_{T,L}$.
\end{remark}

\begin{proof}[Proof of Proposition~\ref{P:4}]
In the following we write $\lesssim$ if $\leq$ holds up to a constant only depending on $d,\lambda,\rho,p$ and $\eta$.

\step 1 
We claim that
\begin{eqnarray*}
  \osc_{a(z)}\mathcal E_{T,L}(a)
  &\lesssim&
  \left(\sum_{x\in\Z^d}|\nabla\eta_L(x)| H(a;x,z)\right)+|\eta_L(z)|(|\nabla\phi^*(a;z)|+1)(|\nabla\phi(a;z)|+1),
\end{eqnarray*}
where
\begin{multline*}
  H(a;x,z):=|\nabla_z G(a;z,x)|(|\nabla\phi^*(a;z)|+1)(|\nabla\phi(a;x)|+1) \\
  + |\nabla_z G(a;x,z)|(|\nabla\phi(a;z)|+1)(|\nabla\phi^*(a;x)|+1).
\end{multline*}
There are three main ingredients needed to show this assertion. First, we  show that for  two arbitrary coefficient fields $a$ and $\tilde a$  that only differ at some fixed $z\in\Z^d$, the difference $ \mathcal E_{T,L}(\widetilde{a})-\mathcal E_{T,L}({a})$ satisfies
\begin{align}
  \mathcal E_{T,L}(\widetilde{a})-\mathcal E_{T,L}({a})
  &=
  \sum_{x\in\Z^d}(\phi^*(\tilde a;x)- \phi^*(a;x))\nabla\eta_L\cdot\widetilde{a}(\xi+\nabla\phi(\tilde a;x))\notag\\
  &-\sum_{x\in\Z^d} (\phi(a;x)-\phi(\tilde a;x))\nabla\eta_L\cdot{a}(\xi^*+\nabla\phi^*(\tilde a;x)) \label{energy-difference}\\
  &+\eta_L(z)(\xi^*+\nabla\phi^*(a;z)\cdot(\widetilde{a}(z)-a(z))(\xi+\nabla\phi(\tilde a;z)),\notag
\end{align}
then we show that one can replace $\nabla\phi(\tilde a;x)$ by $\nabla\phi(a;x)$, and finally we estimate the difference $\phi^*(\tilde a;x)- \phi^*(a;x)$. Let us first derive the expression \eqref{energy-difference}. For brevity we write $\widetilde\phi$ for $\phi$ evaluated at $\widetilde{a}$. Then we have
\begin{align*}
  \mathcal E_{T,L}(\widetilde{a})-\mathcal E_{T,L}(a)
  &=
  \sum_{x\in\Z^d}\eta_L(x)\Bigg[\frac{1}{T}\left(\widetilde{\phi}^*\widetilde{\phi}-\phi^*\phi\right)\\
  &\qquad\qquad\qquad+(\nabla\widetilde\phi^*+\xi^*)\cdot \widetilde{a}(\nabla\widetilde\phi+\xi)-(\nabla\phi^*+\xi^*)\cdot a(\nabla\phi+\xi)\Bigg]\\
  &=
  \sum_{x\in\Z^d}\eta_L(x)\Bigg[\frac{\widetilde{\phi}}{T}\left(\widetilde{\phi}^*-\phi^*\right)+\nabla(\widetilde\phi^*-\phi^*)\cdot \widetilde{a}(\nabla\widetilde\phi+\xi)\\
  &\qquad\qquad\qquad-\frac{\phi^*}{T}\left({\phi}-\widetilde\phi\right)-(\nabla\phi^*+\xi^*)\cdot a\nabla(\phi-\widetilde\phi)\\
  &\qquad\qquad\qquad+(\nabla\phi^*+\xi^*)\cdot (\widetilde{a}-a)(\nabla\widetilde\phi+\xi)\Bigg]\\
  &=
  \sum_{x\in\Z^d}\eta_L(x)\Bigg[\frac{\widetilde{\phi}}{T}\left(\widetilde{\phi}^*-\phi^*\right)+(\widetilde\phi^*-\phi^*)\nabla^* \widetilde{a}(\nabla\widetilde\phi+\xi)\\
  &\qquad\qquad\qquad-\frac{\phi^*}{T}\left({\phi}-\widetilde\phi\right)-(\nabla\phi^*+\xi^*) \nabla^*a(\phi-\widetilde\phi)\Bigg]\\
  &\qquad+\eta_L(z)(\nabla\phi^*(z)+\xi^*)\cdot (\widetilde{a}(z)-a(z))(\nabla\widetilde\phi(z)+\xi)\\
  &\qquad+\sum_{x\in\Z^d}\nabla\eta_L(x)\Bigg[(\widetilde\phi^*-\phi^*) \cdot\widetilde{a}(\nabla\widetilde\phi+\xi)-(\phi-\widetilde\phi) \cdot a(\nabla\phi^*+\xi^*)\Bigg],
\end{align*}
which gives \eqref{energy-difference}. A straightforward estimate of \eqref{energy-difference} yields
\begin{align}
 | \mathcal E_{T,L}(\widetilde{a})-\mathcal E_{T,L}({a})|
  &\lesssim
  \sum_{x\in\Z^d}\left(\osc_{a(z)}\phi^*(a;x)\right)|\nabla\eta_L|(|\nabla\phi(\tilde a;x)|+1)\notag\\
  &\qquad +\sum_{x\in\Z^d}\left(\osc_{a(z)}\phi(a;x)\right)|\nabla\eta_L|(|\nabla\phi^*(a;x)|+1) \label{energy-difference-estimate}\\
  &\qquad+|\eta_L(z)|(1+|\nabla\phi^*(a;z)|)(|\nabla\phi(\tilde a;z)|+1).\notag
\end{align}
Now, by Lemma \ref{L:RP} and Lemma \ref{L2} we have the estimates
  \begin{equation*}
  |\nabla\phi(\tilde a;x)-\nabla\phi(a;x)|\leq 2|\nabla\nabla G(\tilde a;x,z)||\nabla\phi(a;z)+\xi|\lesssim |\nabla\phi(a;z)|+1
  \end{equation*}
and
  \begin{equation*}
        \osc_{a(z)} \phi(a;x)  \leq C(d,\lambda) | \nabla_z G(a;x,z)| (|\nabla \phi(a;z)| + 1).
  \end{equation*}
Moreover, by the symmetric properties of the Green's function we have
  \begin{align*}
        \osc_{a(z)} \phi^*(a;x) & \leq C(d,\lambda) | \nabla_z G(a^*;x,z)| (|\nabla \phi^*(a;z)| + 1) \\&= C(d,\lambda) | \nabla_z G(a;z,x)| (|\nabla \phi^*(a;z)| + 1).
  \end{align*}
Using these estimates in \eqref{energy-difference-estimate} we have
\begin{align*}
 | \mathcal E_{T,L}(\widetilde{a})-\mathcal E_{T,L}({a})|
  &\lesssim
  \sum_{x\in\Z^d}| \nabla_z G(a;z,x)| (|\nabla \phi^*(a;z)| + 1)|\nabla\eta_L|(|\nabla\phi(a;x)|+1)\\
  & \qquad+\sum_{x\in\Z^d}| \nabla_z G(a;x,z)| (|\nabla \phi(a;z)| + 1)|\nabla\eta_L|(|\nabla\phi^*(a;x)|+1) \\
  &\qquad+|\eta_L(z)|(|\nabla\phi^*(a;z)|+1)(|\nabla\phi(a;z)|+1)
\end{align*}
which (after taking the supremum on the left hand side) proves the assertion of Step 1.

\step 2 Conclusion.

Thanks to the spectral gap inequality, we have
\begin{equation*}
  \expec{|\mathcal E_{T,L}-\expec{\mathcal E_{T,L}}|^{2p}}\lesssim\expec{\left(\sum_{z\in\Z^d}(\osc_{a(z)}\mathcal E_{T,L})^2\right)^p}.
\end{equation*}
By Step 1 we have
\begin{eqnarray*}
  &&  \expec{\left(\sum_{z\in\Z^d}(\osc_{a(z)}\mathcal E_{T,L})^2\right)^p}\\
  &\lesssim  &
  \expec{\left(\sum_{z\in\Z^d}\left(\sum_{x\in\Z^d}|\nabla\eta_L(x)| H(x,z)\right)^2\right)^p}
  +\expec{\left(\sum_{z\in\Z^d}|\eta_L(z)|^2(|\nabla\phi^*(a;z)|+1)(|\nabla\phi(a;z)|+1)^2\right)^p}.
\end{eqnarray*}
We estimate the second term by appealing Jensen's inequality in the form of $(\sum\eta_L g)^p\leq \sum_z\eta_L |g|^p$ and the boundedness of moments of $\nabla\phi$ and $\nabla\phi^*$:
\begin{eqnarray*}
  &&\expec{\left(\sum_{\Z^d}|\eta_L|^2(|\nabla\phi^*|+1)(|\nabla\phi|+1)^2\right)^p}\\
  &\leq&\sum_{\Z^d}|\eta_L||\eta_L|^p\expec{(|\nabla\phi^*|+1)^{p}(|\nabla\phi|+1)^{2p}}\\
  &\leq&C(d,\lambda,\rho,p)\|\eta_L\|_{\ell^{\infty}}^p\sum_{\Z^d}|\eta_L| \leq C(d,\lambda,\rho,p)\|\eta\|_{\ell^{\infty}}L^{-pd},
\end{eqnarray*}
the latter holds, since $\eta_L=L^{-d}\eta(\frac{\cdot}{L})$. Next, we estimate the first term starting with expanding the power and an application of H\"older's inequality in probability with exponents $(2p,2p,\frac{p}{p-1})$:
\begin{eqnarray*}
  A&:=&\expec{\left(\sum_{z\in\Z^d}\left(\sum_{x\in\Z^d}|\nabla\eta_L(x)| H(x,z)\right)^2\right)^p}\\
  &=&
  \sum_{z'\in\Z^d}\expec{\left(\sum_{x'\in\Z^d}|\nabla\eta_L(x')| H(x',z')\right)^2\left(\sum_{z\in\Z^d}\left(\sum_{x\in\Z^d}|\nabla\eta_L(x)| H(x,z)\right)^2\right)^{p-1}}\\
  &=&
  \sum_{z'\in\Z^d}\sum_{x'\in\Z^d}\sum_{x''\in\Z^d}|\nabla\eta_L(x')||\nabla\eta_L(x'')| \expec{H(x',z') H(x'',z')\left(\sum_{z\in\Z^d}\left(\sum_{x\in\Z^d}|\nabla\eta_L(x)| H(x,z)\right)^2\right)^{p-1}}\\
  &\leq&
  \sum_{z'\in\Z^d}\sum_{x'\in\Z^d}\sum_{x''\in\Z^d}|\nabla\eta_L(x')||\nabla\eta_L(x'')| \expec{|H(x',z')|^{2p}}^{\frac{1}{2p}}\expec{|H(x'',z')|^{2p}}^{\frac{1}{2p}}\, A^{\frac{p-1}{p}}\\
\end{eqnarray*}

Division by $A^{\frac{p-1}{p}}$ yields
\begin{eqnarray*}
  A&\leq&\left(\sum_{z\in\Z^d}\left(\sum_{x\in\Z^d}|\nabla\eta_L(x)| \expec{|H(x,z)|^{2p}}^{\frac{1}{2p}}\right)^{2}\right)^p
\end{eqnarray*}
Since all moments of $\nabla\phi$ and $\nabla\phi^*$ are finite, we deduce that
\begin{equation*}
  \expec{|H(a;x,z)|^{2p}}^{1/2p}\leq C(d,\lambda,\rho,p)\expec{|\nabla G_z(x-z,0)|^{4p}}^{\frac{1}{4p}},
\end{equation*}
and since $|\nabla\eta_L(x)|\leq L^{-(d+1)}|\nabla\eta(\frac{\cdot}{L})|$ we get
\begin{equation*}
  A\lesssim L^{-2p(d+1)}\left(\sum_{z\in\Z^d}\left(\sum_{|x|\leq L} \expec{|\nabla G(x-z,0)|^{4p}}^{\frac{1}{4p}}\right)^{2}\right)^p=L^{-pd}\mu_{T,L}^{2p}.
\end{equation*}
\end{proof}

\subsection{Annealed Green's function estimates}\label{S:annealed:green}
Let $G^0$ denote the Green's function associated with the discrete Laplacian $\nabla^*\nabla$. We have
\begin{equation*}
  |\nabla G^0(x)|\leq C(d)(|x|+1)^{1-d},\qquad |\nabla_x\nabla_y G^0(x-y)|\leq C(d)(|x-y|+1)^{-d}.
\end{equation*}
It is well-known that both estimates do not hold for heterogeneous coefficients $a\in\Omega$ with a constant that only depends on $d$ and $\lambda$. However, if $\expec{\cdot}$ is a stationary ensemble \textit{on diagonal, uniformly elliptic matrices}, then the following \textit{Delmotte-Deuschel-bounds} hold (cf. \cite{Delmotte-Deuschel}):
\begin{equation}\label{del-deu}
  \expec{|\nabla G(a;x,0)|}\leq C(|x|+1)^{1-d},\qquad \expec{|\nabla_x\nabla_y G(a;x,y)|}\leq C(|x-y|+1)^{-d},
\end{equation}
with $C=C(d,\lambda)$.
Recently, under the additional assumption that $\expec{\cdot}$ satisfies LSI, one of the authors and Otto showed in \cite{MO1} that \eqref{del-deu0} can be lifted from first to arbitrary moments $p<\infty$, i.e.
\begin{equation}\label{del-deu-p}
  \expec{|\nabla G(a;x,0)|^p}^{\frac{1}{p}}\leq C(|x|+1)^{1-d},\quad \expec{|\nabla_x\nabla_y G(a;x,y)|^{p}}^{\frac{1}{p}}\leq C(|x-y|+1)^{-d}
\end{equation}
with $C=C(d,\lambda,\rho,p)$. In this section we extend this lifting argument to (non-diagonal), general elliptic coefficients. The main difficulty is to circumvent arguments that rely on the maximum principle.
As we explain below, both \eqref{del-deu} \& \eqref{del-deu-p} are obtained in \cite{Delmotte-Deuschel} and \cite{MO1} by arguments that crucially rely on the restriction to diagonal matrices (which yields a maximum principle): As explained in \cite{MO1} estimate \eqref{del-deu} easily follows from an analogous statement on the heat kernel due to  Delmotte and Deuschel \cite{Delmotte-Deuschel}. They start with heat kernel upper bounds (that rely on the maximum principle) for the semigroup generated by $\nabla^*(a\nabla)$. In \cite{MO2}  an alternative argument is given that is purely based on elliptic regularity theory (yet, in the discrete case their argument makes explicit use of diagonality).

\begin{proposition}\label{P:annealed}
  Let $d\geq 2$ and suppose that $\expec{\cdot}$ is stationary and satisfies LSI. Let
  $\omega:(0,\infty)\to(0,\infty)$ denote a monotonically increasing weight satisfying
  \begin{equation*}
    \Lambda_0:=\sup_{r\geq 0}\left(\frac{\omega(r)}{\omega(\frac{1}{2}r)}\right)<\infty.
  \end{equation*}
  For $1\leq p<\infty$ define
  \begin{eqnarray*}
    \Lambda_{1,p}:=\sup_{y\in\Z^d}\big(\omega(|y|)\expec{|\nabla G_T(y,0)|^{p}}^\frac{1}{p}\big),\qquad
    \Lambda_{2,p}:=\sup_{y\in\Z^d}\big(\omega(|y|)\expec{|\nabla\nabla G_T(y,0)|^p}^\frac{1}{p}\big).
  \end{eqnarray*}
  Then for any $1\leq p<\infty$ there exists $C=C(d,\lambda,\rho,p,\Lambda_0)<\infty$ s.t.
  \begin{eqnarray*}
    \Lambda_{2,p}\leq C\Lambda_{2,1},\qquad \Lambda_{1,p} \leq C\left(\Lambda_{1,1}+\Lambda_{2,1}\right)
  \end{eqnarray*}
  for all $T\geq 1$.
\end{proposition}

The proof closely follows the argument in \cite{MO1}, which relies on a quenched regularity statement (see \cite[Lemma 6]{MO1}), whose proof uses the maximum principle and appeals to De Giorgi-Nash-Moser regularity theory. We replace the argument by the following statement:
\begin{lemma}\label{L:green}
  There exists $q_0>1$ and $\alpha_0>1$ only depending on $\lambda$ and $d$ such that for any $(q,\alpha)\in[1,q_0]\times[0,\alpha_0]$ we have
  \begin{eqnarray*}
    \sup_{a\in\Omega}\sum_{x\in\Z^d}\big(|\nabla_x\nabla_y G(a;x,y)|(|x-y|+1)^\alpha\big)^{2q}<\infty,\\
  \end{eqnarray*}
  and if $\frac{2q}{2q-1}(1+\alpha)<d$ in addition:
  \begin{equation*}
    \sup_{a\in\Omega}\sum_{x\in\Z^d}\big(|\nabla G(a;x,0)|(|x|+1)^\alpha\big)^{2q}<\infty,
  \end{equation*}
\end{lemma}
The proof of Lemma~\ref{L:green} basically follows  the perturbation argument in the proof of estimate \eqref{eq:L2}. We leave it to the reader.

\begin{proof}[Proof of Proposition~\ref{P:annealed}]
It suffices to consider the case when $\Lambda_{2,1}$ (resp. $\Lambda_{1,1}$ and $\Lambda_{2,1}$) are finite, since otherwise the estimates are trivial.
Fix $\alpha\in(0,\alpha_0]$ and set $\tilde\omega(z):=(|z|+1)^{\alpha}$. W.~l.~o.~g. we may assume that $p>1$ is so large such that $q=\frac{p}{p-1}\leq p_0$ and $p\alpha>1$. Note that this implies $\frac{2q}{2q-1}(\alpha+1)<d<2p\alpha$. Hence, Lemma~\ref{L:green} yields the (deterministic) regularity estimate
\begin{equation}\label{P:annealed:1}
  \sup_{a\in\Omega}\sum_{z\in\Z^d}(|\nabla G(z,0)|\tilde\omega(z))^{2q}+  \sup_{a\in\Omega}\sum_{z\in\Z^d}(|\nabla G(z,0)|\tilde\omega(z))^{2q} + \sum_{z\in\Z^d}\tilde\omega(z)^{-2p}<\infty.
\end{equation}

\step 1

We claim that there exists a constant $C_0=C(d,\lambda,p)<\infty$ s.t. for all $x\in\Z^d$:
\begin{eqnarray*}
  \expec{\left(\sum_{z\in\Z^d}|\nabla\nabla G(x,z)|^2|\nabla\nabla G(z,0)|^2\right)^p}^{\frac{1}{2p}} &\leq& C_0\, \omega(|x|)^{-1}\Lambda_0\Lambda_{2,2p},\\
  \expec{\left(\sum_{z\in\Z^d}|\nabla\nabla G(x,z)|^2|\nabla G(z,0)|^2\right)^p}^{\frac{1}{2p}} &\leq& C_0\, \omega(|x|)^{-1}\Lambda_0\left(\Lambda_{1,2p} + \Lambda_{2,2p}\right).
\end{eqnarray*}
The argument for both estimates is similar. We only discuss the second estimate, which is a bit more difficult, and start with a deterministic estimate that follows from two applications of  H\"older's inequality and \eqref{P:annealed:1}:
\begin{eqnarray*}
  &&\sum_{z\in\Z^d}|\nabla \nabla G(x,z)|^2|\nabla G(z,0)|^2 \\
  &\leq&  \sum_{z:|z|\leq \frac{1}{2}|x|}|\nabla \nabla G(x,z)|^2|\nabla G(z,0)|^2 
  + \sum_{z:|z|> \frac{1}{2}|x|}|\nabla \nabla G(x,z)|^2|\nabla G(z,0)|^2\\
  &\leq&\left(\sum_{z:|z|\leq \frac{1}{2}|x|}|\nabla \nabla G(x,z)|^{2p}\tilde\omega(z)^{-2p}\right)^{\frac{1}{p}}\left(\sum_{z\in\Z^d}\left(|\nabla G(z,0)|\tilde\omega(z)\right)^{2q}\right)^{p-1}\\
    &&+\left(\sum_{z\in\Z^d}\left(|\nabla \nabla G(x,z)|\tilde\omega(x-z)\right)^{2q}\right)^{p-1}\left(\sum_{z:|z|>\frac{1}{2}|x|}|\nabla G(z,0)|^{2p}\tilde\omega(x-z)^{-2p}\right)^{\frac{1}{p}}\\
    &\stackrel{\eqref{P:annealed:1}}{\lesssim}&\left(\sum_{z:|z|\leq \frac{1}{2}|x|}|\nabla \nabla G(x,z)|^{2p}\tilde\omega(z)^{-2p}\right)^{\frac{1}{p}}+\left(\sum_{z:|z|>\frac{1}{2}|x|}|\nabla G(z,0)|^{2p}\tilde\omega(x-z)^{-2p}\right)^{\frac{1}{p}}.
\end{eqnarray*}
Taking the $p$th moment and smuggling in the weight $\omega$ yields
\begin{eqnarray*}
  &&\expec{\left(\sum_{z\in\Z^d}|\nabla \nabla G(x,z)|^2|\nabla G(z,0)|^2\right)^p}\\
  &\lesssim&\sum_{z:|z|\leq \frac{1}{2}|x|}\expec{|\nabla \nabla G(x,z)|^{2p}}\tilde\omega(z)^{-2p} + \sum_{z:|z|>\frac{1}{2}|x|}\expec{|\nabla G(z,0)|^{2p}}\tilde\omega(x-z)^{-2p}\\
  &\leq&\left(\sup_{z} \expec{|\nabla \nabla G(x,z)|^{2p}} \omega(|x-z|)^{2p}\right)
  \sum_{z:|z|\leq \frac{1}{2}|x|}\omega(|x-z|)^{-2p}\tilde\omega(z)^{-2p}
  \\
  && + 
  \left(\sup_{z} \expec{|\nabla G(z,0)|^{2p}} \omega(|z|)^{2p}\right)
  \sum_{z:|z|>\frac{1}{2}|x|}\omega(|z|)^{-2p}\tilde\omega(x-z)^{-2p}\\
  &=&\Lambda_{2,2p}^{2p}
  \sum_{z:|z|\leq \frac{1}{2}|x|}\omega(|x-z|)^{-2p}\tilde\omega(z)^{-2p}
 + \Lambda_{1,2p}^{2p}\sum_{z:|z|>\frac{1}{2}|x|}\omega(|z|)^{-2p}\tilde\omega(x-z)^{-2p},
\end{eqnarray*}
where the last identity holds thanks to the identity $\expec{|\nabla\nabla G(x,y)|^{2p}}=\expec{|\nabla\nabla G(x-y,0)|^{2p}}$, which itself is a consequence of the stationarity of $\expec{\cdot}$. By monotonicity of $\omega$ we have 
\begin{equation*}
  \begin{aligned}
    &|z|\leq \frac{1}{2}|x|\quad\Rightarrow\quad     |x-z|\geq \frac{1}{2}|x|\quad \Rightarrow\quad \omega(|x-z|)^{-2p}\leq\omega(\frac{1}{2}|x|)^{-2p}\leq \Lambda_0^{2p}\omega(|x|)^{-2p},\\
    &|z|>\frac{1}{2}|x|\quad\Rightarrow\quad  \omega(|z|)^{-2p}\leq\omega(\frac{1}{2}|x|)^{-2p}\leq\Lambda_0^{2p}\omega(|x|)^{-2p},
  \end{aligned}
\end{equation*}
and since $\tilde\omega^{-2p}$ is summable (cf. \eqref{P:annealed:1}), the claimed estimate follows.
\medskip

\step 2 Conclusion.

From the LSI (cf. \cite[Lemma 4]{MO1}) we deduce that for all $\delta>0$:
\begin{eqnarray*}
  \expec{|\nabla G(x,0)|^{2p}}^{\frac{1}{2p}}
  &\leq &
  C(d,\delta,p,\rho)\expec{|\nabla G(x,0)|}+\delta\max_{i}\expec{\left(\sum_{z\in\Z^d}\big(\osc_{a(x)}\nabla_{i}G(x,0)\big)^2\right)^p}^{\frac{1}{2p}}\\
  \expec{|\nabla\nabla G(x,0)|^{2p}}^{\frac{1}{2p}}
  &\leq &
  C(d,\delta,p,\rho)\expec{|\nabla\nabla G(x,0)|}+\delta\max_{i,j}\expec{\left(\sum_{z\in\Z^d}\big(\osc_{a(x)}\nabla_{i}\nabla_jG(x,0)\big)^2\right)^p}^{\frac{1}{2p}}.
\end{eqnarray*}
Furthermore, we note that
\begin{eqnarray*}
  \osc_{a(z)}\nabla_{i} G(a;x,0)&\leq& C(d,\lambda)|\nabla\nabla G(a;x,z)||\nabla G(a;z,0)|,\\
  \osc_{a(z)}\nabla_{i}\nabla_j G(a;x,0)&\leq& C(d,\lambda)|\nabla\nabla G(a;x,z)||\nabla\nabla G(a;z,0)|,
\end{eqnarray*}
as can be seen by an argument similar to Lemma~\ref{L:RP} (see also \cite[Lemma~5]{MO1}).
The combination of the previous estimates with Step~1 yields the two estimates
\begin{eqnarray*}
  \expec{|\nabla G(x,0)|^{2p}}^{\frac{1}{2p}}
  &\leq&
  C(d,\delta,p,\rho)\expec{|\nabla G(x,0)|}+\delta C(d,\lambda,p)\omega(|x|)^{-1}\Lambda_0(\Lambda_{1,2p}+\Lambda_{2,2p}),\\
  \expec{|\nabla\nabla G(x,0)|^{2p}}^{\frac{1}{2p}}
  &\leq&
  C(d,\delta,p,\rho)\expec{|\nabla\nabla G(x,0)|}+\delta C(d,\lambda,p)\omega(|x|)^{-1}\Lambda_0\Lambda_{2,2p},
\end{eqnarray*}
We multiply with $\omega(|x|)$, take the sup in $x$ and  finally get (by choosing $\delta=\delta(d,\lambda,p,\Lambda_0)$ sufficiently small):
\begin{eqnarray*}
  \Lambda_{1,2p}
  &\leq &
  C(d,\lambda,\rho,p,\Lambda_0)\Lambda_{1,1}+\frac{1}{2}(\Lambda_{1,2p}+\Lambda_{2,2p}),\\
  \Lambda_{2,2p}
  &  \leq &
  C(d,\lambda,\rho,p,\Lambda_0)\Lambda_{2,1}+\frac{1}{2}\Lambda_{2,2p}.
\end{eqnarray*}

\end{proof}

\appendix

\section{Proof of Lemma~\ref{L:Gint}}
Thanks to the shift-invariance $G_T(a;x,y)=G_T(a(\cdot+y);x-y,0)$, it
suffices to prove the estimate for $y=0$. We set for brevity
\begin{equation*}
  G(x):=G_T(a;x,0)
\end{equation*}
and recall that $G$ is the unique solution in $\ell^2(\Z^d)$ to 
\begin{equation}\label{eq:2}
  \frac{1}{T}G+\nabla^*(a\nabla G)=\delta.
\end{equation}
By discreteness and the standard energy estimate, we have
\begin{equation*}
  \frac{1}{T}|G(0)|^2\leq \frac{1}{T}\sum_{x}|G(x)|^2 + \lambda \sum_x |\nabla G(x)|^2 \leq G(0).
\end{equation*}
Hence, $0\leq G(0)\leq T$ and we have that
\begin{equation}\label{eq:3}
  \sum_{x}\big(|G(x)|^2+|\nabla G(x)|^2\big)\leq C(T,\lambda).
\end{equation}
Formally we may upgrade~\eqref{eq:2} to the statement of Lemma~\ref{L:Gint} by testing the
equation with $e^{\frac{\delta }{2}|x|}G(x)$. Since that is not an
admissible $\ell^2(\Z^d)$ test function, we appeal to an approximation
of the form $\zeta G$ where
\begin{equation}\label{eq:16}
  \zeta(x):=\eta(x)e^{\delta g(x)},
\end{equation}
and $\eta,g:\Z^d\to\R$ are bounded, compactly supported and
non-negative functions, and $g$ mimics the behavior of the
linearly growing function $x\mapsto \frac{|x|}{2}$.  The truncation
via $\eta$ and the discrete Leibniz rule introduce error terms. In
order to treat these terms in a convenient way, we will appeal to test functions
$\eta$ and $g$ that  additionaly satisfy the following property:
\begin{equation}\label{eq:test-fun-supp}
  \begin{aligned}
\nabla_i\eta(x)\neq0\ \Rightarrow\ g(x)=g(x+e_i)=0\text{ for
      all $i=1,\ldots,d$ and $x\in\Z^d$.}
  \end{aligned}
\end{equation}
After these remarks we turn to the proof of \eqref{eq:3}. We first
establish a chain
rule inequality for test functions in the form of \eqref{eq:16}
assuming \eqref{eq:test-fun-supp}. In
Step~2 we test \eqref{eq:2} by $G\zeta$, and finally in Step~3 we
conclude by explicitly defining a sequence of test functions approaching $e^{\frac{\delta }{2}|x|}$. 
\medskip

{\bf Step 1.} Choice of test functions: For an arbitrary parameter $R\ge3$, say, we first construct the appropriate test functions $\eta$ and $g$. Let
\begin{enumerate}[(a)]
\item $\eta:\R^d\to\R$ be a smooth function satisfying
	\[
	{\eta}(x)=\begin{cases}
	1& \text{if }|x|\in[0,R],\\
	0& \text{if }|x|\in[R+1,\infty),
	\end{cases}
	\quad\text{such that }|\nabla\eta|\leq 2,
	\]
\item and $g:\R^d\to\R$ be a smooth function satisfying
	\[
	{g}(x)=\begin{cases}
	\frac{x}{2}& \text{if }|x|\in[0,\frac{R}{2}],\\
	0& \text{if }|x|\in[R-1,\infty),
	\end{cases}
	\quad\text{such that }|\nabla g|\leq 2.
	\]
\end{enumerate}
Furthermore we define $\zeta$ through~\eqref{eq:16}.
By construction $\eta$ and $g$ satisfy \eqref{eq:test-fun-supp} and there
exists a constant $C=C(d)>0$ independent of $R$ such that
\begin{equation}\label{eq:test-fun-bounds}
  \|\nabla\eta\|_{\ell^\infty(\Z^d)}+\|\nabla g\|_{\ell^\infty(\Z^d)}\leq C(d).
\end{equation}
Thus we have that
\begin{equation}\label{eq:14}
  |\nabla_i\zeta(x)| \le C(d)\big(\min\{|\zeta(x)|,|\zeta(x+e_i)|\} \delta + 1\big),
\end{equation}
Indeed, this is seen by writing $|\nabla_i\zeta|$ in the following two equivalent forms: On the one hand, an application of the discrete Leibniz rule
\[
\nabla_i(fg)(x)=\nabla_if(x)g(x)+f(x+e_i)\nabla_ig(x)
\]
yields
\begin{align*}
  \nabla_i\zeta(x)
  &=
  \eta(x+e_i)e^{\delta g(x+e_i)} -  \eta(x)e^{\delta g(x)}\\
  &=
  \eta(x+e_i)(e^{\delta g(x+e_i)}-e^{\delta g(x)})+(\eta(x+e_i)-\eta(x))e^{\delta g(x)}\\
  &=\eta(x+e_i)e^{\delta g(x+e_i)}(1-e^{-\delta \nabla_ig(x)})+\nabla_i\eta(x),
\end{align*}
since $(\eta(x+e_i)-\eta(x))e^{\delta g(x)}=\nabla_i\eta(x)$ by~\eqref{eq:test-fun-supp}.
On the other hand, a similar calculation yields
\begin{align*}
  \nabla_i\zeta(x)
  &=
  \eta(x)\nabla_i(e^{\delta g(x)})+\nabla_i\eta(x)e^{\delta g(x+e_i)}\\
  &=
  \zeta(x)(e^{\delta \nabla_ig(x)}-1)+ \nabla_i\eta(x).
\end{align*}
Therefore~\eqref{eq:14} follows from~\eqref{eq:test-fun-bounds}.

\medskip

{\bf Step 2}. Testing the equation with $\zeta$: We claim that there exists $\delta=\delta(d,\lambda,T) > 0$ such that
\begin{equation}\label{eq:18}
    \sum_x |\zeta(x)|^2(|G(x)|^2+|\nabla G(x)|^2) \leq C(d,\lambda,T) \Big(G(0) + \sum_{x}(|G(x)|^2+|\nabla G(x)|^2) \Big).
\end{equation}
for all $R\ge3$, say.
Our argument is as follows: The discrete Leibniz rule yields
\begin{align*}
  |\zeta(x)|^2\nabla_iG(x)
  &=
  \nabla_i(\zeta^2G)(x)-G(x+e_i)\nabla_i(\zeta^2(x))\\
  &=
  \nabla_i(\zeta^2G)(x)-G(x+e_i)\nabla_i\zeta(x)\big(\zeta(x)+\zeta(x+e_i)\big)\\
  &=
  \nabla_i(\zeta^2G)(x)-G(x+e_i)\nabla_i\zeta(x)\big(2\zeta(x)+\nabla_i \zeta(x)\big).
\end{align*}
Together with ellipticity of $a$, cf.~\eqref{ass:ell}, we obtain that
\begin{align*}
  &\frac{1}{T}\sum_x |G(x)|^2|\zeta(x)|^2+\lambda\sum_x|\zeta(x)|^2|\nabla G(x)|^2\\\nonumber
  &\leq \frac{1}{T}\sum_x |G(x)|^2|\zeta(x)|^2+\sum_x|\zeta(x)|^2\nabla
  G(x)\cdot a(x)\nabla G(x)\\
  &=
  \frac{1}{T}\sum_x |G(x)|^2|\zeta(x)|^2+\sum_x\nabla(\zeta^2
  G)(x)\cdot a(x)\nabla G(x)\\
  &\qquad\qquad -\sum_{x,i,j}G(x+e_i)\nabla_i\zeta(x)\big(2\zeta(x)+\nabla_i \zeta(x)\big)\,a_{ij}(x)\nabla_j G(x).
  \end{align*}
By the defining equation~\eqref{eq:2} for $G$, the second-to-last line equals $G(0)|\zeta(0)|^2=G(0)$ (for the choice of test function in Step~1). Therefore Young's inequality and $|a|\le 1$ yield
\begin{multline}\label{eq:15}
  \frac{1}{T}\sum_x |G(x)|^2|\zeta(x)|^2+\lambda\sum_x|\zeta(x)|^2|\nabla G(x)|^2\\
  \le G(0) + \frac{1}{2\epsilon}\sum_{x,i} |G(x+e_i)|^2 |\nabla_i\zeta(x)|^2 + \epsilon\sum_{x,i} \big(2|\zeta(x)|^2+|\nabla_i \zeta(x)|^2\big) |\nabla_i G(x)|^2
\end{multline}
for all $\epsilon>0$.
The gradient estimate~\eqref{eq:14} of the test function $\zeta$ yields
\[
  \sum_{x,i}|G(x+e_i)|^2|\nabla_i\zeta(x)|^2 \leq C(d) \Big(\delta\sum_{x}|G(x)|^2|\zeta(x)|^2 + \sum_{x}|G(x)|^2\Big),
\]
as well as
\[
  \sum_{x}|\nabla_i\zeta(x)|^2|\nabla_i G(x)|^2 \leq C(d) \Big(\delta\sum_{x}|\nabla G(x)|^2|\zeta(x)|^2 + \sum_{x}|\nabla G(x)|^2\Big).
\]
Inserting the last two estimates into~\eqref{eq:15} yields
\begin{align*}
  &\frac{1}{T}\sum_x |G(x)|^2 |\zeta(x)|^2+\lambda\sum_x|\zeta(x)|^2|\nabla G(x)|^2\\
  &\le G(0) + \Big( \frac{C(d)\delta}{2\epsilon} + 2\epsilon \Big)\sum_{x}|\zeta(x)|^2 \big(|\nabla G(x)|^2 + |G(x)|^2 \big)\\ &\qquad\qquad\qquad\qquad+ C(d)\Big(\frac{1}{2\epsilon} + \epsilon \Big) \sum_{x} \big(|\nabla G(x)|^2 + |G(x)|^2 \big)
\end{align*}
for all $\epsilon, \delta>0$. An appropriate choice of $\epsilon$ and $\delta$, for instance $\epsilon = \sqrt{\delta}$ with $\delta = \delta(d,\lambda,T)$ small enough, allows us to absorb the sums involving $\zeta$ on the left-hand side and we obtain~\eqref{eq:18}.
\medskip

\textbf{Step 3.} Conclusion: We substitute the definition~\eqref{eq:16} into~\eqref{eq:18} and recall the construction of $\eta$ and $g$ in Step~1 to obtain that
\[
    \sum_{x\in\Z^d:|x|\le \frac{R}{2}} (|G(x)|^2+|\nabla G(x)|^2) e^{\delta(d,\lambda,T) |x|} \leq C(d,\lambda,T) \Big(G(0) + \sum_{x\in\Z^d}(|G(x)|^2+|\nabla G(x)|^2) \Big).
\]
for all $R\ge3$.
By~\eqref{eq:3}, the right-hand side is bounded by $C(d,\lambda,T)$ and therefore the claim follows upon letting $R\uparrow\infty$.

\qed%


\end{document}